\documentclass[notitlepage,leqno,10pt]{article}

\usepackage{latexsym}
\usepackage{amsmath}
\usepackage{amsfonts}
\usepackage{color}
\usepackage{amsthm}
\usepackage{amssymb,mathrsfs,enumerate}

\newcommand{\ep}{\varepsilon}

\newcommand{\bg }{\bar{g}}
\newcommand{\Sp }{\mathbb{S}^{n-1}}

\renewcommand{\a }{\alpha }
\renewcommand{\b }{\beta }
\renewcommand{\d}{\delta }
\newcommand{\D }{\Delta }

\newcommand{\e }{\varepsilon }
\newcommand{\g }{\gamma}

\newcommand{\n }{\nabla }

\newcommand{\Sig}{\Sigma}
\renewcommand{\t }{\theta }

\newcommand{\boa}{\boldsymbol{a}}
\newcommand{\tboa}{\tilde{\boldsymbol a}}

\newcommand{\intbar}{\mathop{\int\makebox(-13.5,0){\rule[4pt]{.7em}{0.3pt}}%
\kern-6pt}\nolimits}

\newcommand{\be}{\begin{eqnarray}}
\newcommand{\ee}{\end{eqnarray}}
\newcommand{\hs}{\hspace{1cm}}

\newcommand{\disp}{\displaystyle} 
\newcommand{\Ru}{R_1}
\newcommand{\Rd}{R_2}
\newcommand{\ru}{r_1}
\newcommand{\rd}{r_2}
\newcommand{\etu}{\eta_1}
\newcommand{\etd}{\eta_2}
\newcommand{\eti}{\eta_i}
\newcommand{\gc}{g_{cyl}}
\newcommand{\vu}{v_{D_{\eta_{1}}}}

\newcommand{\R}{\mathbb{R}}

\newcommand{\bigo}[1]{\mathcal{O} \big( #1 \big)}
\newcommand{\nor}[2]{\|{#1}\|_{#2}}
\definecolor{coloras}{rgb}{0.,0.67,0}
\def\bea{\begin{eqnarray*}}
\def\eea{\end{eqnarray*}}
\def\f{\frac}

\author{Lorenzo MAZZIERI$^{a}$ and Antonio SEGATTI$^{b}$}

\date{}

\title{\bf Constant $\sigma_{k}$-curvature metrics with Delaunay type ends}

\begin{document}

\linespread{1.02}

\parindent=0pt

\newtheorem{lem}{Lemma}[section]
\newtheorem{pro}[lem]{Proposition}
\newtheorem{thm}{Theorem}
\newtheorem{rem}[lem]{Remark}
\newtheorem{cor}[lem]{Corollary}
\newtheorem{df}[lem]{Definition}
\newtheorem{claim}[lem]{Claim}
\newtheorem{conj}[lem]{Conjecture}
\newtheorem{ass}[lem]{Assumption}
\numberwithin{equation}{section}
\newtheorem{ackn}{Acknowledgments\!\!}\renewcommand{\theackn}{}

\maketitle

\begin{center}

\

{\small

\noindent $^a$ SISSA - International School for Advanced Studies

Via Beirut 2-4,

I-34014 Trieste - Italy
}

\

{\small

\noindent $^b$ Dipartimento di Matematica F. Casorati -
Universit\`a di Pavia\\
via Ferrata 1,\\
I-27100 Pavia - Italy

}

\end{center}

\footnotetext[1]{E-mail addresses: antonio.segatti@unipv.it, mazzieri@sissa.it}

\begin{center}
{\bf Abstract}

\end{center}

In this paper we produce families of complete non compact Riemannian metrics with positive constant $\sigma_k$-curvature equal to $2^{-k} {n \choose k}$ by performing the connected sum of a finite number of given $n$-dimensional Delaunay type solutions, provided $2 \leq 2k < n$. The problem is equivalent to solve a second order fully nonlinear elliptic equation.

\begin{center}


\bigskip

\centerline{\bf AMS subject classification:  53C24, 53C20,
53C21, 53C25}

\end{center}

\parindent=0pt

\centerline{}

\vspace{-1cm}

%\tableofcontents

\section{Introduction and statement of the result}\label{s:intro}

In recent years much attention has been given to the study of the Yamabe problem for $\sigma_{k}$--curvature, briefly the $\sigma_{k}$--Yamabe problem. To introduce the analytical formulation, we first recall some background material from Riemmanian geometry. Given $(M,g)$, a compact Riemannian manifold of dimension $n\geq 3$, we denote respectively by $Ric_g$, $R_g$ the Ricci tensor and the scalar curvature of $(M,g)$. The Schouten tensor of $(M,g)$ is defined as follows  
\begin{eqnarray*}
A_g & := & \tfrac{1}{ n-2} \,\, \big( \, Ric_g \,\, - \,\, \tfrac{1}{2(n-1)} \, R_g g \, \big) \,\, . 
\end{eqnarray*}
If we denote by $\lambda_{1}, \ldots, \lambda_{n}$ the eigenvalues of the symmetric endomorphism $g^{-1}A_{g}$, then the $\sigma_k$-curvature of $(M,g)$ is defined as the $k$-th symmetric elementary function of $\lambda_{1},\ldots,\lambda_{n}$, namely
\begin{eqnarray*}
\sigma_k(g^{-1} A_{g}) \,\,\, := \, \sum_{i_1\, <\,\ldots\,< \, i_k}\lambda_{i_i}\cdot \, \ldots \, \cdot \lambda_{i_k} \,\, \quad \hbox{for $1\leq k \leq n$} \quad \quad  & \hbox{and} & \quad \quad 
\sigma_0 (g^{-1} A_{g}) \,\,\,:= \,\,\,1 .
\end{eqnarray*}
The $\sigma_{k}$--Yamabe problem on $(M,g)$
%on a given compact Riemannian manifold $(M,g)$ of dimension $n\geq 3$ 
consists in finding metrics with constant $\sigma_{k}$--curvature in the same conformal class of $g$. The case $k=1$ is the well known Yamabe problem, whose progressive resolution is due to Yamabe \cite{Yamabe}, Trudinger \cite{trudinger}, Aubin \cite{aubin} and Schoen \cite{schoen}. In order to present the existence results for $k\geq 2$, when the equation becomes fully nonlinear, 
we recall the notion of $k$-admissibility, which is a sufficient condition to insure the ellipticity of the 
equation. A metric $g$ on $M$ is said to be $k$--admissible if it belongs to the $k$--th positive cone $\Gamma^{+}_{k}$, where
$$
g\in\Gamma^{+}_{k}\quad\Longleftrightarrow\quad \sigma_{j}(g^{-1}A_{g})>0\quad\hbox{for}\quad j=1,\ldots,k.
$$
Under the assumption that $g$ is $k$--admissible the (positive) $\sigma_{k}$--Yamabe problem on closed manifolds has been solved in the case $k=2$, $n=4$ by Chang, Gursky and Yang \cite{cgy1} \cite{cgy2}, for locally conformally flat manifolds by Li and Li \cite{ll} (see also Guan and Wang \cite{gw}), and for $2k>n$ by Gursky and Viaclovsky \cite{gv}. For $2 \leq 2k \leq n$ the problem has been solved by Sheng, Trudinger and Wang \cite{stw} under the extra--hypothesis that the operator is variational. We point out that for 
%$k=n/2$ and $(M,g)$ non local conformally flat  as well as for 
$k=1,2$ this hypothesis is always fulfilled, whereas for $k\geq 3$ it has been shown in \cite{bg} that this extra assumption is equivalent to the locally conformally flatness. 

\medskip

Of interest in this paper is the construction of complete non compact locally conformally flat metrics with constant (positive) $\sigma_k$-curvature. These solutions can be regarded as singular solutions for the $\sigma_k$-equation on the complement of a discrete set $\Lambda$ on the standard $n$-dimensional sphere. To put our result in perspective, we recall that for $k=1$, the first examples of conformal constant (positive) scalar curvature metrics with isolated singularities have been obtained by Schoen in \cite{schoen2}. Later, Mazzeo and Pacard (see \cite{mp0} and \cite{mp}) have produced different families of solutions on the complement of a singular set $\Lambda$ consisting of a finite disjoint union of closed smooth submanifolds of arbitrary dimension between $0$ and ${(n-2)}/{2}$. Another existence result, in the case where $\Lambda$ is given by an even number of points, is due to Mazzeo, Pollack and Uhlenbeck \cite{mpu}. We will return on this later, since our construction is closely related to their work. For $2\leq k<n/2$, the first examples of complete non compact metrics lying in the $k$-th positive cone and having constant $\sigma_k$-curvature have been obtained by the first author in a joint work with Ndiaye \cite{mn}, assuming that the points of the singular set have a symmetric disposition.

\medskip

Some comments about the asymptotic behavior of the singular solutions are now in order. For $k=1$ it follows from the works of Caffarelli, Gidas and Spruck \cite{cgs} and Korevaar, Mazzeo, Pacard and Schoen \cite{kmps} that every complete non compact locally conformally flat metric with constant positive scalar curvature must be asymptotic to a radial solution. In a recent work, Han, Li and Teixeira \cite{hlt} have shown that this fact is true also for metrics of constant $\sigma_k$ curvature lying in the $k$-th positive cone, provided $2 \leq k < n/2$. Notice that for $k\geq n/2$ the singularity is always removable. 
For these reason, it is clear that complete radial solutions are going to play a fundamental role in our construction. These particular solutions are also known as Delaunay-type metrics and have been classified by Chang, Han and Yang in \cite{chang-han-yang05} and we recall them briefly in Section \ref{Delaunay-Schwartz type metrics}. Essentially, they are conformally cylindrical metrics with a periodic conformal factor, whose minimum will be referred as Delaunay-parameter. 

\medskip

As anticipated, the construction presented in this paper is inspired by \cite{mpu} and consists in performing the connected sum of a finite number of Delaunay-type metrics. The solutions obtained in this way are quite different from the ones produced in \cite{mn}, which roughly speaking looks like a spherical central body with several Delaunay-type ends having small Delaunay parameters. 
In the present construction, the Delaunay parameters are not forced to be small, hence our solutions may possibly belong to a different connected component of the moduli space.  

\medskip

To fix the notations, we recall that the connected sum of two $n$--dimensional Riemannian manifolds $(D_1,g_1)$ and $(D_2,g_2)$ is the topological operation which consists in removing an open ball from 
both $D_1$ and $D_2$ and identifying the leftover boundaries, obtaining a new manifold with possibly different topology. Formally, if $p_i \in D_i$ and for a small enough $\e>0$ we excise the ball $B(p_i, \e)$ from $D_i$, $i=1,2$, the (pointwise) connected sum $M_\e$ of $D_1$ and $D_2$ along $p_1$ and $p_2$ with {\em necksize} $\e$ is the topological manifold defined as
$$
M_{\e}\,\, := \,\, D_{1}\sharp_{\e}D_{2} \,\, = \,\,\left[D_{1}\setminus B(p_{1},\e)\,\cup\, D_{2}\setminus B(p_{2},\e)\right]\big / \sim  \,\, ,
$$
where $\sim$ denotes the identification of the two boundaries $\partial B(p_i,\e)$, $i=1,2$. Of course the new manifold $M_\e$ can be endowed with both a differentiable structure and a metric structure, as it will be explicitly done in Section \ref{s:as}. Even though from a topological point of view the value of the {\em necksize} is forgettable, it will be important to keep track of it, when we will deal with the metric structure.

\medskip

Concerning the solvability of the Yamabe equation ($k=1$) on the pointwise connected sum of manifolds with constant scalar curvature, we recall the results of Joyce \cite{joyce} for the compact case and the already mentioned work of Mazzeo, Pollack and Uhlenbeck \cite{mpu} for the non compact case. For $2 \leq k < n/2$ and compact manifolds, a connected sum result has been provided by the first author in a joint work with Catino \cite{cat-maz}. Our main result here is the following
\begin{thm}
\label{main1}
Let $(D_{1},g_{1}),\ldots,(D_{N},g_{N})$ be a collection of $n$-dimensional Delaunay-type solutions (see Proposition \ref{collecting-Del}) 
to the positive $\sigma_k$-Yamabe problem, with $2\leq 2k <n$. Then, there exists a positive real number $\e_0 >0$ only depending on $n$, $k$, 
and the $C^2$--norm of the coefficients of 
$g_1$ and $g_2$ such that, for every $\e \in (0, \e_0]$, the connected sum $M_{\e}=D_{1}\sharp_{\e} \ldots \sharp_\e D_{N}$ can be endowed with a metric $\widetilde{g}_{\e}$ with constant $\sigma_{k}$--curvature equal to $2^{-k} {n \choose k}$. Moreover $\Vert \widetilde{g}_{\e}- g_{i}\Vert_{C^{r}(K_{i})}\rightarrow 0$ for 
any $r>0$ and 
any compact set $K_{i}\subset D_{i}\setminus\{p_{i}\}$,  the $p_i$'s, $i=1,\ldots,N$, being the points about which the connect sum is performed.
\end{thm}
Some comments about the strategy of the proof are in order. Incidentally, we notice that the constant $2^{-k}{n \choose k}$ arises naturally as the $\sigma_k$--curvature of the $n$-dimensional standard sphere, so we will end up with a family of metrics $\{\widetilde{g}_{\e} \}_\e$ parametrized in terms of the {\em necksize} which satisfy 
\begin{eqnarray}
\label{sigmak=const}
\sigma_k \big( \, \widetilde{g}_\e^{-1}  A_{\widetilde{g}_\e} \big) & = &  2^{-k} \, \hbox{${n \choose k}$} .
\end{eqnarray}
To show the existence of these solutions, we start by writing down (see Section \ref{s:as}) an explicit family of approximate solution metrics $\{g_\e \}_\e$ (still parametrized by the {\em necksize}) on $M_\e$. This metrics are complete and non compact, since they coincide with the original Delaunay-type metrics $g_i$ on $D_i \setminus B(p_i, \e)$, $i=1,2$, and are close to a model metric on the remaining piece of the connected sum manifold, which in the following will be referred as neck region. The metric which we are going to use as a model in the neck region is described in Section \ref{Delaunay-Schwartz type metrics}. It is a complete metric on $\R \times \Sp$ with zero $\sigma_k$--curvature and yields a natural generalization of the scalar flat Schwarzschild metric. It has been successfully employed in \cite{cat-maz} to treat the connected sum of constant scalar curvature manifolds and for the local analysis on the neck region we will refer to this work.

\medskip

The next step in our strategy amounts to look for a suitable correction of the approximate solutions to the desired exact solutions. This will be done by means of a global conformal perturbation. At the end it will turn out that for sufficiently small values of the parameter $\e$ such a correction can actually be found together with a very precise control on its size and this will ensure the smooth convergence of the new solutions $\tilde{g}_\e$ to the former metrics $g_i$ on the compact subsets of $M_i \setminus \{p_i\}$, $i=1,2$. We point out that it is also important to control the asymptotic behavior of such a perturbation in order to preserve the completeness of the approximate solutions. Typically, one is led to search for corrections which present a decay at infinity.

\medskip

The main point in the correction procedure is to provide invertibility for the linearized operator about the approxiamte solutions, together with uniform (with respect to the {\em necksize} parameter $\e$) {\em a priori} bounds. This  will enable us to carry out the perturbative nonlinear analysis (Section \ref{s:nonlinear}) by proving the convergence of a Newton iteration scheme. The uniformity of the {\em a priori} bound will follow from the use of weighted function spaces with a weighting function acting on the 
neck region, in analogy with the analysis contained in \cite{cat-maz}. On the other hand, the invertibility issue is quite different from the compact case. In fact, in order to obtain the desired Fredholm properties for the linearized operator, we will further introduce weighting functions with gradient supported outside of a compact region of $M_\e$. In analogy with the case $k=1$ (see \cite{mpu}), the analysis is complicated by the lack of coercivity of the linearized operator. This is due to the conformal invariance of the $\sigma_k$-equation. In fact, 
the functions which are responsible for this lack of coercivity arise as infinitesimal generators (Jacobi fields) of conformal transformations. As it will be made clear in Section \ref{s:nonlinear}, the geometrical interpretation of the Jacobi fields will be exploited in order to insure the completeness, after the perturbation, of the exact solutions.

\begin{ackn} 
The first author is partially supported by the Italian project FIRB--IDEAS ``Analysis and Beyond'', the second author is supported by the  Italian PRIN 2006 {\sl ``Problemi a frontiera libera, transizioni di fase e modelli di isteresi''}. 
%Moreover, the authors
%gratefully acknowledge the Casalesi's Clan for the financial support through the program ''Scudo Fiscale'' of the italian government.
\end{ackn}

\

\section{Notations and preliminaries}

\medskip

We fix now the notations that will be used throughout this paper. Let $(M,\bar g)$ be a compact smooth $n$--dimensional Riemannian manifold without boundary an let $2 \leq 2k <n$. Taking advantage of this second assumption, we introduce the following formalism for the conformal change 
\begin{eqnarray*}
\bar{g}_u  & := & u^{\frac{4k}{n-2k}}\, \bar{g} ,
\end{eqnarray*}
where the conformal factor $u>0$ is a positive smooth function. In this context $\bar g$ will be referred as the background metric. At a first time the $\sigma_k$--equation for the conformal factor $u$ can be formulated as 
\begin{eqnarray*}
\sigma_k \big(\, \bg_u^{-1} A_{\bg_u} \big) & = &  
2^{-k} \, \hbox{${n \choose k}$}.
\end{eqnarray*}
We recall that the Schouten tensor of $\bg_u$ is related to the one of $A_{\bg}$ by the conformal transformation law
\bea
A_{\bg_{u}} & = & A_{\bg} - 
\tfrac{2k}{n-2k}  u^{-1}{\nabla^2 u}  + 
\tfrac{2kn}{(n-2k)^2} u^{-2}  {du \otimes du}  - 
\tfrac{2k^2}{(n-2k)^2}  u^{-2}  {|du|^2}  \bg  ,
\eea
where $\nabla^2$ and $|\cdot|$ are computed with respect to the background metric $\bg$.
For technical reasons, it is convenient to set
\begin{eqnarray*}
\label{endo}
B_{\bg_u}  & := & \tfrac{n-2k}{2k} \,  u^{\frac{2n}{n-2k}} \, \bg_u^{-1} \cdot A_{\bg_u}
\end{eqnarray*}
and to reformulate the $\sigma_k$--equation as 
\begin{eqnarray}
\label{eq}
  \mathcal{N} (u, \bg)  & := & \sigma_k  \left( B_{\bg_{u}}
  \right)  - \hbox{ ${n \choose k}$} \big( \tfrac{n-2k}{4k} \big)^k u^{\frac{2kn}{n-2k}}  \,\,\,\,   = \,\,\,\,  0 .
\end{eqnarray}
We notice that if two metrics $\bg$ and ${g}$ are related by $\bg=(v/u)^{4k/(n-2k)}{g}$, then the nonlinear operator enjoys the following {\em conformal equivariance property}
\be\label{confequi}
\mathcal{N} \, (u,\bg)& = & (v/u)^{-\frac{2kn}{n-2k}} \, \mathcal{N} \, (v,g).
\ee 
The linearized operator of $\mathcal{N} (\,\cdot \,, \bg)$ about $u$ is defined as
\be
\label{linope}
\mathbb{L}(u,\bg)\,[w]  &:= & \left.
\frac{d}{ds} \right|_{s=0} \mathcal{N}\, (u  + sw,\bg).
\ee 
Most part of the analysis in this paper (Sections \ref{analisi-del} and \ref{global-linear-analysis}) is concerned with the study of the mapping properties of the linearized operator about the approximate solutions $g_\e$'s, that we will write in the form $u_\e^{4k/(n-2k)} \bg$.
As a direct consequence of the property \eqref{confequi}, we have the following \emph{conformal equivariance property} for the linearized operator
\begin{equation}\label{confequilin}
\mathbb{L}(u,\bg)[w]\,\,=\,\,(v/u)^{-\frac{2kn}{n-2k}}\,
\mathbb{L}(v,g)[(v/u)\, w].
\end{equation}

\

\section{Delaunay and Schwarzschild type metrics on $\R\times \mathbb{S}^{n-1}$}

\label{Delaunay-Schwartz type metrics}

\noindent We start this section with the description of a particular family of complete metrics on the cylinder $\R \times \mathbb{S}^{n-1}$ with constant $\sigma_k$-curvature equal to $2^{-k}{n \choose k}$. These metrics are conformal to the cylindrical one $g_{cyl}$  on the whole cylinder $\R \times \mathbb{S}^{n-1}$ (notice that in the following, the cylindrical metric will also be denoted by $g_{cyl}  =  dt^2 + \, d\theta^2$, where $d\theta^2$ represents
the standard metric on $\mathbb{S}^{n-1}$).
%and therefore they are also conformal to the round metric on the sphere $\mathbb{S}^n$ minus two antipodal points. In fact this class of metrics automatically provides a family of solutions to our problem in the case where the singular set consists of just two points. 
%
%Furthermore these special metrics, which for $k=1$ are also known as Delaunay or Fowler metrics, will play a crucial role in the construction of the approximate solutions when the cardinality of the singular set $\Lambda$ is bigger than $2$. In fact, as it will be made clear in Section \ref{s:as}, we will use them as a model for our metrics nearby the singularities.

\medskip

\noindent Let us consider then on the standard cylinder $(\R \times \mathbb{S}^{n-1}, dt^2 + \, d\theta^2)$ a conformal metric $g$ of the form $g \, = \, v^{4k /(n-2k)} \, g_{cyl}$, where the conformal factor $v$ only depends on the $t$ variable, i.e., $v=v(t)$, and let us impose the condition $\sigma_k\, \left( \,g^{-1}A_g \, \right) \, = \, 2^{-k} \, {{n \choose k}} $ or equivalently 
\begin{eqnarray}
\label{constant sigma_k}
\mathcal{N}(v, g_{cyl}) = 0 \,\,.
\end{eqnarray}
It is easy to observe that, under the usual change of coordinates, $t
= -\log |x|$ and $\theta = x/|x|$, this corresponds to look for a metric on $\R^n
\setminus \{0\}$ which has constant positive $\sigma_k$-curvature
and which is radially symmetric. These metrics has been studied in \cite{chang-han-yang05} by Chang, Han and Yang and we refer the reader to their work for further details. Here we just recall the following

%To complete this remark, we also
%notice that our metric $g$ can be written in terms of the Euclidean
%metric $g_{\R^n}$ on $\R^n \setminus \{0\}$ as $g \, = \, u^{{4k}/{(n-2k)}} \, g_{\R^n}$,
%where $u = u (|x|) = |x|^{-{(n-2k)}/{2k}} \, v\,(-\log |x|)$. This
%correspondence will turn out to be useful when we will need to take
%advantage of the conformal equivariance \eqref{confequi} of the equation (\ref{constant
%sigma_k}), but for the moment it is preferable to work with a cylindrical backgroud.
%
%We present now a couple of propositions which collects some important results concerning the existence of Delaunay type metrics on $\mathbb{R}\times \mathbb{S}^{n-1}$
%as well as the structure of the linearization of the fully nonlinear operator $\mathcal{N}$ around the 
%Delaunay type metric.

\begin{pro}[Delaunay-type metrics]\label{collecting-Del}
Let $g_v$ be a a metric on $\mathbb{R}^n$ of the form $g_v \, = \, v^{4k /(n-2k)} \, g_{cyl}$, where $v$ is a smooth positive function only depending on the variable $t\in \mathbb{R}$. Let us define the quantity   
\begin{eqnarray*}\label{eq:energy}
H \, ( v  , \dot{v})  & := & \big[ \,\, v^2 \, - \,
\big(\tfrac{2k}{n-2k}\big)^2 \, \dot{v}^2 \,\, \big]^{k} \,\, - \,\,
v^{\frac{2kn}{n-2k}} 
\end{eqnarray*}
Then, if $H \,(v, \dot{v}) \,\equiv \, H_0  \in \big( \,0 \, \, , \, \frac{2k}{n-2k}\, \big(
\frac{n-2k}{n}\big)^{{n}/{2k}}\, \big)$, 
in correspondence of each $H_0$, there exists a unique solution $v$ to 
\begin{eqnarray}\label{eq:ode}
\big[ \,\, v^2 \, - \, \big(\tfrac{2k}{n-2k}\big)^2 \, \dot{v}^2
\,\, \big]^{k-1} \, \big[ \,\, v \, - \,
\big(\tfrac{2k}{n-2k}\big)^2 \, \ddot{v} \,\, \big] & = &
\tfrac{n}{n-2k} \,\, v^{\frac{2kn}{n-2k} - 1} \,\, .
\end{eqnarray}
satisfying the conditions $\dot{v}(0)=0$, and $\ddot{v}(0)>0$. This family of solutions gives rise to a family of complete and periodic
metrics on $\R\times \mathbb{S}^{n-1}$ satisfying 
\begin{equation}
\label{del-solution}
\sigma_k\, \left( \,B_{g_v} \, \right) \, = \, 2^{-k} \, \hbox{${n \choose k}$}  \,\,\,\,\, in \,\, \mathbb{R}\times \mathbb{S}^{n-1}
\end{equation}
This solution is
periodic and it is such that $0 < v(t) < 1$ for all $t \in \R$.
In the following we will index the conformal factors and the metrics in this family by means of the parameter $\eta \, : = \, v(0)^{{2k}/({n-2k})} 
$ which represents the neck-size. Notice that $0 < \eta < \big( \frac{n-2k}{n}\big)^{{1}/{2k}}$ and that the period of $v_{D, \eta}$ will be denoted by $T_\eta$.
\end{pro}

As anticipated in the introduction, the first step in our strategy amounts to build approximate solutions on the connected sum of a finite number of Delaunay type solutions with possibly different neck-size parameter $(D_{\eta_1}, g_{\eta_1}) , \ldots , (D_{\eta_N}, g_{\eta_N})$. To this end, we need to modify the original metrics in a neighborhood of the points that we are going to excise, obtaining a new metric in the so called {\em neck region}. In the scalar curvature case ($k=1$), a clever choice turns out to be the (space-like) \emph {Schwarzschild} metric. This is a complete scalar flat metric conformal to the cylindrical metric $g_{cyl}$ on $\mathbb{R}\times\mathbb{S}^{n-1}$. The explicit formula is given by
$$
g \,\,:= \,\, \cosh\left(\tfrac{n-2}{2}t\right)^{\f{4}{n-2}}g_{cyl} \, .
$$
In a similar way, it is easy to construct a complete conformal metric on $\mathbb{R}\times\mathbb{S}^{n-1}$ with zero $\sigma_{k}$--curvature, for all $2\leq2k<n$. We have
\begin{pro}[Schwarzchild metrics]
\label{scharz}
Let $g_{v}$ be a metric on $\mathbb{R}\times\mathbb{S}^{n-1}$ of the form $g_{v}=v^{4k/(n-2k)}g_{cyl}$, where $v$ is a positive smooth function depending only on $t\in\mathbb{R}$. Let us define the quantity $$h(t):=v^{2}(t)-\big(\tfrac{2k}{n-2k}\big)^{2}\dot{v}^{2}(t).$$ Then, if $h_{0}:=h(0)>0$,  the family of positive solutions to the equation
$$
\sigma_{k}(B_{g_{v}})=0\quad\mbox{in }\mathbb{R}\times\mathbb{S}^{n-1}\\
$$
is given by $v(t)=\sqrt{h_{0}}\cosh\left(\tfrac{n-2k}{2k}t-c\right),$ $c\in\mathbb{R}$. In the following, the solution with $c=0$ will be denoted by $v_\Sigma$.
\end{pro}
For sake of completeness and for future convenience, we recover from \cite{cat-maz} the following formula for the linearized $\sigma_k$-operator about the Schwarzschild type metric.
\begin{lem}\label{lin-scharz}
The linearized $\sigma_k$-operator about the $\sigma_k$-Schwarzschild metric 
\bea
\mathbb{L}^{0}(v_{\Sigma}, g_{cyl})[w]:=\left.\frac{d}{ds}\right|_{s=0}  \sigma_k \left( B_{g_s} \right) \, ,
\eea
where $g_s = g_{v+ sw}$, is given by
\begin{eqnarray}
\label{1st expression D} \mathbb{L}^{0}(v_{\Sig},g_{cyl}) [w] = -C_{n,k}\, v_{\Sig}\,h_{\Sig}^{k-1} \big[ \, \partial^{2}_{t}+\tfrac{n-k}{k(n-1)}\D_{\t}-\big(\tfrac{n-2k}{2k}\big)^{2}\big]\, w \,,
\end{eqnarray}
where $ C_{n,k} \,\, = \,\,  {n-1 \choose k-1} \,\, \left( \tfrac{n-2k}{4k}
\right)^{k-1}$.
\end{lem}
Incidentally, we note that the computation (see \cite{cat-maz}) leading to \eqref{1st expression D} also shows that 
\begin{eqnarray}
\label{sigmaj}
\sigma_{k-1-j} (B_0) & = & \left(\tfrac{n-2k}{4k} \right)^{k-1-j} \, h_{\Sig}^{k-1-j}\,\,\tfrac{1+j}{k}\,\hbox{${n\choose k-1-j}$} \, .
\end{eqnarray}
From this it follows that the $\sigma_{k}$--Schwarzschild metric $g_{\Sig}$ belongs to $\overline{\Gamma}_{k}^{+}\cap \Gamma_{k-1}^{+}$, for $2\leq 2k<n$.

\section{Approximate solutions}
\label{s:as}

In this section we first describe the construction of the connected sum of a finite number of Delaunay-type solutions $D_{\eta_1}, \ldots, D_{\eta_N}$ and then we define on this new manifold a family of metrics which will represent the approximate solutions to our problem. 

\medskip

Since the whole construction is local, we restrict ourself to the connected sum of two Delaunay type solutions $(D_{\etu},g_{1})$ and $(D_{\etd},g_{2})$. In the following we will denote by $M_{\e}:=D_{\etu}\sharp_{\e} D_{\etd}$ the manifold obtained by excising two geodesic balls of radius $\e\in(0,1)$ centered at $p_{1}\in D_{\etu}$ and $p_{2}\in D_{\etd}$ and identifying the two left over boundaries. The manifold $D_{\eta_1}$ and $D_{\eta_2}$ are endowed with the metrics 
\bea
g_1 \,\, = \,\,v_{D,{\etu}}^{\frac{4k}{n-2k}}(dr_1^2 + g_{\mathbb{S}^{n-1}}) \quad & \hbox{and}& \quad
g_2	\,\, = \,\,v_{D,{\etd}}^{\frac{4k}{n-2k}}(dr_2^2 + g_{\mathbb{S}^{n-1}}) \,,
\eea
respectively. Starting from $g_1$ and $g_2$, Êwe will define on $M_{\e}$ a new metric $g_{\e}$ which agrees with the old ones outside the balls of radius one around $p_1$ and $p_2$ and which is modeled on (a scaled version of) the $\sigma_{k}$--Schwarzschild metric in the neck region.

\medskip

To describe the construction, we consider the diffeomorphisms given by the Êexponential maps
$$\exp_{p_{i}}: B(O_{p_{i}},1)\subset T_{p_{i}}D_{\eti}\longrightarrow B(p_{i},1)\subset D_{\eta_i}, \quad i=1,2.$$ Next, to fix the notations, we identify the tangent spaces $T_{p_{i}}D_{\eti}$ with $\R^{n}$. It is well known that this identification yields normal coordinates centered at the points $p_{i}$, namely
$$x: B(p_{1},1)\longrightarrow \R^{n}\quad\quad \mbox{and}\quad\quad y: B(p_{2},1)\longrightarrow \R^{n}.$$
We introduce now asymptotic cylindrical coordinates on the punctured ball $B^{*}(0,1)=x\left(B^{}(p_{1},1) \setminus \{p_1\}\right)$ setting $t:=\log\e-\log|x|$ and $\t:=x/|x|$. In this way we have the diffeomorphism $B^{*}(0,1)\simeq(\log\e,+\infty)\times\mathbb{S}^{n-1}$. Analogously, we consider the diffeomorphism $y\left(B^{}(p_{2},1) \setminus \{ p_2\}\right)=B^{*}(0,1)\simeq(-\infty,-\log\e)\times\mathbb{S}^{n-1}$, this time setting $t:=-\log\e+\log|y|$ and $\t:=y/|y|$. 

\medskip

In order to define the differential structure of $M_{\e}$, we excise a geodesic ball $B(p_{i},\e)$ from $D_{\eti}$, obtaining an annular region $A(p_{i},1,\e):=B(p_{i},1)\setminus B(p_{i},\e)$, $i=1,2$. The asymptotic cylindrical coordinates introduced above can be used to define a natural coordinate system on the neck region
$$
(t,\t):\,\left[A(p_{1},1,\e)\sqcup A(p_{2},1,\e)\right]/\sim\,\,\longrightarrow (\log\e,-\log\e)\times\mathbb{S}^{n-1}=:N_{\e},
$$
where $\sim$ denotes the relation of equivalence which identifies the boundaries of $B(p_{1},\e)$ and $B(p_{2},\e)$, namely
$$q_{1}\sim q_{2}\,\,\Longleftrightarrow\,\, x/|x|(q_{1})=y/|y|(q_{2})\quad\mbox{and}\quad |x|(q_{1})=\e=|y|(q_{2}).
$$
Clearly, in this coordinates, the two identified boundaries correspond now to the set $\{0\}\times\mathbb{S}^{n-1}$. To complete the definition of the differential structure of $M_{\e}$ it is sufficient to consider the old coordinate charts on $D_{\eti}\setminus B(p_{i},1)$, $i=1,2$. 

\medskip

We are now ready to define on $M_{\e}$ the approximate solution metric $g_{\e}$. First of all, we define $g_{\e}$ to be equal to the $g_{i}$ on $D_{\eti}\setminus B(p_{i},1)$, $i=1,2$. To define $g_{\e}$ in the neck region, we start by observing that the choice of the normal coordinate system allows us to expand the two metric $g_{1}$ and $g_{2}$ around $p_{1}$ and $p_{2}$ respectively as 
$$
g_{1}=\left[\d_{\a\b}+\bigo{|x|^{2}}\right]\,dx^{\a}\otimes dx^{\b}\,\quad\mbox{and}\,\quad g_{2}=\left[\d_{\a\b}+\bigo{|y|^{2}}\right]\,dy^{\a}\otimes dy^{\b}.
$$
Recalling thatÊ
the metrics $g_1$ and $g_2$ are \emph{locally conformally flat} and using the $(t,\t)$-coordinates introduced above, we can write
\bea
 g_{1}&=&u_{1}^{\frac{4k}{n-2k}}(1+c_{1})(dt^2 + d\theta^2), \,\,\hbox{ with } u_1(t):= \e^{\frac{n-2k}{2k}} e^{-\frac{n-2k}{2k} t}\\
 Ê g_{2}&=&u_{2}^{\frac{4k}{n-2k}}(1+c_{2})(dt^2 + d\theta^2), \,\,\hbox{ with } u_2(t):= \e^{\frac{n-2k}{2k}} e^{\frac{n-2k}{2k} t} Ê
 \eea
Now, we fix as background metric on $M_\e$ the following
$$
\bar{g}:=
\begin{cases}
g_{i}\quad\quad&\hbox{on }D_{\eti}\setminus B(p_{i},1)\\
(1 + c) (dt^2 + d\theta^2)\quad\quad&\hbox{on }A(p_{1},1,\e)\sqcup A(p_{2},1,\e)]/\sim
\end{cases}
$$
where
$$
c \,\,:=\,\, \eta \, c_1 + (1-\eta) \, c_2,
$$
with $\eta$ a smooth and non decreasing cut off function such that Ê
$\eta: (\log\e, -\log\e)
\rightarrow [0,1]$ and
identically equal to $1$ in $(\log\ep, -1]$ and $0$ in
$[1,-\log\ep)$.
Subsequentely, we consider another non increasing smooth  function $\chi :
(\log\ep, -\log\ep) \rightarrow [0,1]$ which is identically equal to $1$ in $(\log\ep, -\log \ep
-1]$ and which satisfies $\lim_{t\rightarrow -\log\ep} \chi = 0$. Using these cut-off functions, we can now define a new conformal factor
\begin{eqnarray}
\label{def-approx-conf}
u_{\e}& := &
\begin{cases}
1\quad\quad&\hbox{on }D_{\eta_i}\setminus B(p_{i},1)\\
\chi(t) \, u_{1} Ê+ \chi(-t) \,
u_{2} \quad\quad&\hbox{on }A(p_{1},1,\e)\sqcup A(p_{2},1,\e)]/\sim
\end{cases}
\end{eqnarray}
Finally we define on $M_\e$ the family of approximate solution metrics $g_\e$, by setting
\begin{equation}
\label{approx-metric}
g_\e\,\, :=\,\, u_\e^{\frac{4k}{n-2k}}\bg \,\, .
\end{equation}
To conclude this section, we observe that with this definition we immediately have that for every $m \in \mathbb{N}$ the approximate solution metrics converge to $g_i$ on the compact subset of $D_{\eta_i} \setminus \{p_i\}$ with respect to the $C^m$-topology when the parameter $\e$ tends to $0$, for $i=1,2$. For these reasons, we expect that the size of the term $\mathcal{N} (u_\e, \bg)$, which represents the fail of $u_\e$ from being an exact solution, will become smaller and smaller when $\e \rightarrow 0$. Finally, we notice that adapting the proof of \cite[Lemma 3.2]{cat-maz} it is straightforward to show that for $\e$ sufficiently small $g_\e$ lies in $\Gamma^{+}_{k-1}$.

\section{Analysis of the linearized operator about the Delaunay-type metrics}
\label{analisi-del}
In this section we discuss some boundary value problems for the linearized operator introduced in \eqref{linope} about Delaunay-type metrics.
%(see the next \eqref{linear-del-expression}). 
This local analysis, will find its application in Section \ref{global-linear-analysis}. 

\medskip

We start by recovering from \cite{mn} the expression for the linearized operator about a Delaunay-type metric $g_{D, \eta} = v_{D,\eta}^{4k/(n-2k)} \, g_{cyl}$. We set
%To do that we introduce the auxiliary quantities $h \, = \, h \,(v_{D,\eta}, \dot v_{D,\eta})$ and $F \,= \,F \,(v_{D,\eta}, H)$ by setting
\begin{eqnarray}
\label{h-del-F}
h \,\, := \,\,  v_{D, \eta}^2 \, - \, 
\big(\tfrac{2k}{n-2k}\big)^2 \, \dot{v}_{D,\eta}^2  \quad \quad \hbox{and} \quad \quad  F \,\, :=  \,\, {v_{D,\eta}^{\frac{2kn}{n-2k}}} \big/ \big( \, {H \,
+ \, v_{D,\eta}^{\frac{2kn}{n-2k}}} \, \big) Ê\,\, 
\end{eqnarray}
and we recall that that $h(t)\,>\,0$ for any $t\in \mathbb{R}$ (see \cite{chang-han-yang05} and \cite{mn}). With these definitions at hand, we can state the following
\begin{lem}[Linearization about the Delaunay-type metrics]
\label{Linear-Del}
The linearized operator about the Dealunay-type solution $v_{D, \eta}$ is given by
\begin{eqnarray}
\label{linear-del-expression} 
\mathbb{L} \, (v_{D,\eta}\, , g_{cyl})\,[w] Ê& = & Ê - \,\, C_{n,k} \,\,
v_{D,\eta} \,\, Êh^{\frac{k-1}{2}} Ê\,\,\left\{\,
\partial_t^2 \, + \, {a_\eta} \, \Delta_{\theta} \, - \, p_\eta \,
\right\} \,\,[h^{\frac{k-1}{2}}
\, w] \,\, ,
\end{eqnarray}
where $\Delta_\theta$ is the Laplace-Beltrami operator for standard round metric on $g_{\mathbb{S}^{n-1}}$ and the coefficients $a_\eta$ and $p_\eta$ are given by
\begin{eqnarray}
\label{coefficients D a} a_\eta & := & \tfrac{n-k}{k(n-1)} \,\, + \,\,
\tfrac{n(k-1)}{k(n-1)} \,\,
F \,\, , \\
\label{coefficients D p} p_\eta & := &  \left(\tfrac{n-2k}{2k}\right)^2
\,\, + \,\, \tfrac{n\,(nk\,+\,n\,-\,2k)\,(k\,-\,1)}{2k^2} \, F \,\,
- \,\,
\tfrac{n^2\,(k^2-\,1)}{4k^2} \, F^2 \\
& & \quad \quad \quad \quad \, - \,\, \tfrac{n\,(2kn \,- \,n \,+
\,2k )}{4k} \, v_{D,\eta}^{\frac{4k}{n-2k}} \, F^{\frac{k-1}{k}} \,\, + \,\,
\tfrac{n^2\,k\,(k\,-\,1)}{4k^2} \, v_{D,\eta}^{\frac{4k}{n-2k}} \,
F^{\frac{2k-1}{k}} \,\, . \nonumber
\end{eqnarray}
and the constant $C_{n,k}$ is defined by $C_{n,k} \, := \, \hbox{${n-1 \choose k-1}$} \,\, \left( \tfrac{n-2k}{4k}
\right)^{k-1}$. For notational convenience we also define the conjugate linearized operator by
\begin{eqnarray}
\label{conj}
\mathcal{L}_\eta & := & \partial_t^2 \, + \, {a_\eta} \, \Delta_{\theta} \, - \, p_\eta \,\, .
\end{eqnarray}
Moreover, we have that there exists a positive constant $c=c(n,k)>0$ such that for every $j \geq n+1$ and every admissible Delaunay parameter $\eta$
\begin{eqnarray}
\label{coercivity}
a_{\eta} \,\lambda_j \, + \,p_{\eta}  & \geq & c\,\,,\end{eqnarray}
where the positive real numbers $\lambda_j$, $j \in \mathbb{N}$, denote the eigenvalues (counted with multiplicity) of $\Delta_{\theta}$, i.e., $-\Delta_{\theta} \phi_j \, = \, \lambda_j \, \phi_j$. 
\end{lem}
As a consequence of the last inequality, we will obtain the coercivity of the conjugate linearized operator $\mathcal{L}_\eta$ along the high frequencies (i.e., for $j \geq n+1$).\\
%We advice the reader that in this Section the cylindrical coordinates will be denoted by $(t,\theta)\in \mathbb{R}\times \mathbb{S}^{n-1}$

\subsection{Jacobi fields}

\label{jacobi fields}

\noindent Using the conformal equivariance of the equation we introduce new families of solutions which are variations of the standard Delaunay solution with neck-size parameter $\eta$. The infinitesimal generators of these variations will provide us with natural elements sitting in the kernel of the linearized operator $\mathbb{L}(v_{D, \eta}, g_{cyl})$ about $v_{D,\eta}$, namely the Jacobi fields.

%Thus, starting from the one parameter family $\eta \mapsto v_{D,\eta} (t)$ of solutions to equation \eqref{del-solution} (see Proposition \ref{collecting-Del}), we want to use the geometric properties of our equation to present a wider set of solutions. 
\medskip

The first remark is that since the equation \eqref{eq:ode} is autonomous, then the solutions are translation invariant (with respect to the $t$ variable). In particular, for $\tau >0 $, the functions $v_{D,\eta,\tau} (t)$, defined by
\begin{eqnarray}
v_{D, \eta, \tau}(t) & := & v_{D,\eta} (\, t + \log(\tau+1)) \,\, ,
\end{eqnarray}
are still solution to \eqref{eq:ode}. To find other possible families of solutions it is convenient to use the conformal equivariance of equation \eqref{del-solution}. First notice that, writing $t = - \log |x|$ and $\theta = x/|x|$, with $x \in \R^n \setminus \{ 0\}$, the cylindrical metric and the Euclidean one are related by $g_{cyl} = |x|^{-2} g_{\R^n}$ on $\R^n \setminus \{0\}$. As a consequence of \eqref{confequi} we get
\begin{eqnarray}
\mathcal{N} \,( v , g_{cyl}\, ) & = & |x|^{n} \, \mathcal{N} \, (\, |x|^{-\frac{n-2k}{2k}} v , g_{\R^n}) \,\, .
\end{eqnarray}
Hence, if $v(t, \theta)$ solves $\mathcal{N}(v,g_{cyl}) = 0$ on $\R \times \mathbb{S}^{n-1}$, then the function $u (x)$, defined on $\R^n \setminus \{ 0\} $ by
\begin{eqnarray}
\label{correspondence}
u(x) & := & |x|^{-\frac{n-2k}{2k}} \, v (-\log |x| \, , \, x/|x|) \,\, ,
\end{eqnarray}
is a solution to $\mathcal{N} Ê(u, g_{\R^n}) = 0$ on $\R^n \setminus \{ 0\}$. In particular, the Delaunay solutions $v_{D,\eta} (t)$ defined on the cylinder correspond to the radial solutions of the latter equation $u_{D,\eta} (|x|) := |x|^{-(n-2k)/2k} \, v_{D,\eta} (-\log |x|)$ with a pole in the origin. Since the equation satisfied by $u_{D,\eta}$ is clearly translation invariant (due to the fact that the background metric is $g_{\R^{n}}$), we have that, for $b \in \R^n$, the $n$-parameter family of functions $u_{D,\eta,b} (x)$, defined by 
\begin{eqnarray}
u_{D,\eta, b}(x) & := & u_{D,\eta} (|x-b|) \,\,\, = \,\,\, |x-b|^{-\frac{n-2k}{2k}} v_{D,\eta} (-\log|x-b|) \,\, ,
\end{eqnarray}
still satisfies $\mathcal{N}(u_{D,\eta,b} \,, g_{\R^n}) = 0$. These functions present a singularity at $b \in \R^n$ and they are radial with respect to $b \in \R^n$. These new solutions $u_{D,\eta,b} (x)$ defined on $\R^n \setminus \{ b\}$ correspond via \eqref{correspondence} to the solutions $v_{D,\eta,b}(t,\theta)$ Êof the equation \eqref{constant sigma_k}
defined on $\R \times \mathbb{S}^{n-1} \setminus \{(-\log|b|\, , b/|b| )\}$ by
\begin{eqnarray}
v_{D,\eta,b} (t, \theta) & := & |\theta - b e^t|^{-\frac{n-2k}{2k}} v_{D,\eta} (t - \log|\theta - b e^t |) \,\,.
\end{eqnarray}
The last family of solutions comes in the following way. First observe that the function $t \mapsto \bar v_{D,\eta}(t) := v_{D,\eta} (-t)$ is still a solution to \eqref{constant sigma_k} on $\R \times \mathbb{S}^{n-1}$. This corresponds to the fact that on $\R \setminus \{ 0\}$ the Kelvin transform of $u_{D,\eta}$, namely the function 
\begin{eqnarray} Ê
\bar u_{D,\eta} (|x|) & := & |x|^{-\frac{n-2k}{k}}u_{D,\eta} \big( \big| x/|x|^2 \big|\big) \,\,\,\, Ê= \,\,\,\, |x|^{-\frac{n-2k}{2k}} v_{D,\eta} \big( - \log \big| x/|x|^2 \big|\big) \,\,, 
\end{eqnarray}
satisfies $\mathcal{N} \, (\bar u_{D,\eta} \, g_{\R^n}) \, = \, 0$. Now we translate $\bar u_{D,\eta}$ by a vector $a \in \R^n$, obtaining an $n$-parameter family of functions $\bar u_{D,\eta,a}(x) := \bar u_{D,\eta} (|x-a|) = |x-a|^{-(n-2k)/2k} v_{D,\eta} (\log |x-a|)$ and finally we take the Kelvin transforms of 
the $\bar u_{D,\eta,a}$'s obtaining, for $a \in \R^n$ the new family of solutions on $\R^n \setminus \{ 0\}$
\begin{eqnarray}
u_{D,\eta,a} (x) & := & \big| x - a|x|^2 \big|^{-\frac{n-2k}{2k}} v_{D,\eta} \big( -2 \log |x| + \log \big| x - a|x|^2 Ê\big| \,\, \big) Ê\,\, .
\end{eqnarray}
These solutions are no longer radial and present a singularity at the origin. For $a \in \R^n$, they correspond on $\R \times S^{n-1} \setminus \{ (\log |a| \, , a/|a| \, ) \}$ to the solutions
\begin{eqnarray} 
v_{D,\eta,a} (t, \theta) & := & |\theta - a e^{-t}|^{-\frac{n-2k}{2k}} v_{D,\eta} (t + \log |\theta - a e^{-t}|) \,\, .
\end{eqnarray}
In the remaining part of this section we will use all these families of solutions to define some special elements in the kernel of the linearized operator around the Delaunay solutions $v_{D,\eta}$. First of all, we recall that if $\lambda \mapsto v_{D,\eta ,\lambda}$ is a variation of $v_{D,\eta}$ such that for every admissible value of the parameter $\lambda$ 
\begin{eqnarray*}
\mathcal{N} \, (v_{D,\eta , \lambda} \, , g_{cyl}) \, \, = \,\, 0 Ê\quad & \hbox{and} & \quad v_{D, \eta , 0} (t) \,\, = \,\, v_{D, \eta} (t) \,\, ,
\end{eqnarray*}
then it is straightforward to see that 
\begin{equation*}
0\,\,\, = \,\,\, \left. \frac{\partial}{\partial\lambda} \right|_{\lambda = 0} \mathcal{N} \, (v_{D,\eta , \lambda} \, , g_{cyl}) \,\,\, = \,\,\, \mathbb{L} (v_{D,\eta} \, , g_{cyl}) \,\, \left. \frac{\partial}{\partial \lambda} \right|_{\lambda = 0} Êv_{D,\e,\lambda} \,\, ,
\end{equation*}
where $\mathbb{L}(v_{D,\eta} \,, g_{cyl})$ represents the linearized operator around the Delaunay solution $v_{D,\eta}$. The functions $\left. {\partial_\lambda} \right|_{\lambda = 0} Êv_{D,\eta,\lambda}$ are the so called Jacobi fields and they clearly belong to the kernel of $\mathbb{L}(v_{D,\eta} \,, g_{cyl})$. Applying this reasoning to the family of solutions $\alpha \mapsto v_{D, \eta + \alpha}$ and $\tau\mapsto v_{D, \eta, \tau}$, it is natural to define the quantities
\begin{eqnarray}
\label{psio}
\Psi^{0,-}_\eta (t) \,\,:= \,\, \left. \frac{\partial}{\partial \alpha} \right|_{\alpha = 0} Êv_{D,\eta + \alpha}(t) \quad & \hbox{and} & \quad \Psi^{0,+}_\eta (t) Ê\,\, := \,\, \left. \frac{\partial}{\partial \tau} \right|_{\tau = 0} Êv_{D,\eta, \tau} (t) \,\, = \,\, \dot{v}_{D,\eta} (t) \,\, .
\end{eqnarray}
In analogy with that, we use the other two families $b \mapsto v_{D,\eta, b}$ and $a \mapsto v_{D,\eta,a}$ to define, for $j=1,\ldots , n$, the Jacobi fields
\begin{eqnarray}
\label{psij-}
\Psi^{j,-}_{\eta} (t, \theta) & := & Ê\left. \frac{\partial}{\partial b^j} \right|_{b = 0} Êv_{D,\eta, b} (t, \theta) \,\, = \,\, \big[ \tfrac{n-2k}{2k} v_{D,\eta} (t) Ê\, + \, \dot{v}_{D,\eta} (t) \, \big] \, e^t \, \cdot \, \phi_j(\theta) \,\,, \quad Ê\\
\label{psij}
\Psi^{j,+}_{\eta} (t, \theta) & := & Ê\left. \frac{\partial}{\partial a^j} \right|_{a = 0} Êv_{D,\eta, a} (t, \theta) \,\, = \,\, \big[ \tfrac{n-2k}{2k} v_{D,\eta} (t) Ê\, - \, \dot{v}_{D,\eta} (t) \, \big] \, e^{-t} \,\cdot \, \phi_j (\theta) \,\,,
\end{eqnarray} 
where the $\phi_j$'s are the $n$ eigenfunction of the Laplacian on $S^{n-1}$ with eigenvalue $n-1$, namely $-\Delta_{\theta} \phi_j \, = \, (n-1) \, \phi_j$, for $j=1,\ldots,n$.

\subsection{A linear problem on the cylinder $\R \times \mathbb{S}^{n-1}$}

\label{linearcyl}

In this subsection we want to study the problem
\bea
\mathbb{L}(v_{D,\eta}, g_{cyl}) \, [w] & = & f \quad \hbox{in} \,\, \R \times \mathbb{S}^{n-1}\,\, .
\eea
Following \cite{mpu} we observe that the natural functional setting for this problem is given by weighted H\"older or Sobolev spaces. Both choices are essentially equivalent. However, in our argument we will use H\"older spaces. For a fixed weight parameter $\d \in \R$ and $m \in \mathbb N$ we define the space
\begin{eqnarray*}
C^m_\d (D_\eta) & := & \{ u \in C^m(D_\eta) \,\, : \,\, \nor {u}{C^{m}_\d} 
%\,:=\, \nor{ (\cosh t)^{-\d} u}{C^m} \,
<  +\infty    \} \,\, ,
\end{eqnarray*}
where the weighted norm is defined by
\begin{eqnarray*}
\nor{u}{C^m_\d (D_\eta)} & := & \sup_{\,\, \R \times \mathbb{S}^{n-1}}   \hbox{$\sum_{j=1}^m$} (\cosh t)^{-\d} \,|\nabla^j u | \, (t, \theta) \,\, .
\end{eqnarray*}
We point out that $|\,\cdot\,|$ and $\nabla$ are respectively the norm and the Levi-Civita connection of the cylindrical metric $g_{cyl}$. In the same way we define for $\d \in \R$, $m \in \mathbb{N}$ and $\b \in (0,1)$ the weighted H\"older seminorm by
\begin{eqnarray}
\label{semiholder-weight}
[\,u\,]_{C^{m,\b}_\d (D_\eta)} & := & \sup_{ t \,\in \,\R }  \,\, (\cosh t)^{-\d}\, [\,u\,]_{C^{m,\b}(\,(t-1, t+1) \times \mathbb{S}^{n-1}\,)} \,\, .
\end{eqnarray}
The weighted H\"older spaces are then given by
\begin{eqnarray}
\label{holder-weight}
C^{m,\b}_\d (D_\eta) & := & \big\{ u \in C^{m,\b}(D_\eta) \,\, : \,\, \nor {u}{C^{m, \b}_\d}  \, := \, \nor {u}{C^{m}_\d} + \, [\,u\,]_{C^{m,\b}_\d } \,
%\,:=\, \nor{ (\cosh t)^{-\d} u}{C^m} \,
<  +\infty   \, \big\} \,\, .
\end{eqnarray}
Following the analysis in \cite{pacard}, one immediately find that 
\bea
\mathbb{L}(v_{D, \eta}, g_{cyl}) \,\,: \,\, C^{2,\b}_\d (D_\eta) & \longrightarrow & C^{0,\b}_\d (D_\eta)
\eea
is Fredholm, provided $\d \notin I_\eta$, where $I_\eta := \{\pm\d_{j,\eta}  \, : \, j \in \mathbb{N} \, \}$ is the set of the indicial roots of the operator $\mathbb{L}(v_{D, \eta}, g_{cyl})$ at both $+\infty$ and $-\infty$. In general the indicial roots (for a precise definition see \cite{pacard}) do depend on the neck-size parameter $\eta$, but here it follows from the explicit knowledge of the Jacobi fields that $\d_{0,\eta}$ and $\d_{1,\eta}$ are independent of $\eta$. In particular the indicial root $\d_{0,\eta} = 0$ is related to the Jacobi fields $\Psi^{0,-}_\eta$ and $\Psi^{0, + }_{\eta}$, which are respectively linearly growing and bounded in $t$, whereas the indicial root $\d_{1,\eta} = 1$ has multiplicity $n$ and is related to the Jacobi fields $\Psi^{j,-}_\eta$ and $\Psi^{j ,+}_{\eta}$, $j=1, \ldots,n$, which are respectively exponentially growing with rate $e^t$ and exponentially decreasing with rate $e^{-t}$. We also point out that as a consequence of the inequality \eqref{coercivity} it 
can be deduced that for every admissible value of the Delaunay parameter $\eta$, the indicial root $\d_{2,\eta}$ verifies the inequality
\begin{eqnarray}
\label{2nd root}
\bar \d (n,k) \,\, \,:= \, \,\, \sqrt{\tfrac{2n(n-k)}{k(n-1)} + \big(\tfrac{n-2k}{2k}\big)^2} & \leq & \d_{2, \eta} \,\, .
\end{eqnarray}

We are now in the position to prove the following 
\begin{lem}
\label{inj.}
Let $1 < \d  $, then the operator
\bea
\mathbb{L}(v_{D, \eta}, g_{cyl}) \,\,: \,\, C^{2,\b}_{-\d} (D_\eta) & \longrightarrow & C^{0,\b}_{-\d} (D_\eta)
\eea
is injective.
\end{lem}
\begin{proof}
Since the functions which are involved in the conjugation \eqref{conj} are bounded and positive, it is not restrictive to prove the result for the conjugate operator
\bea
\mathcal{L}_\eta \,\,: \,\, C^{2,\b}_{-\d} (D_\eta) & \longrightarrow & C^{0,\b}_{-\d} (D_\eta) \,\, .
\eea
Performing a standard separation of variables and projecting the equation along the eigenfunctions of $\Delta_\theta$, we note that for the low frequencies $j= 0, \ldots , n$ the space of the general solutions to the homogeneous equation is spanned by the (conjugate) Jacobi fields $\Phi^{j,\pm}_\eta \, := \, h^{(k-1)/2} \Psi^{j,\pm}_\eta$. On the other hand, it is easy to check that, for $\d>1$,  no one of these functions belongs to the weighted space $ C^{2,\b}_{-\d} (D_\eta)$.
Thus, it is sufficient to test the injectivity only for the high frequencies, $j \geq n+1$. Hence, suppose to have a function $\Phi$ such that
$$
\mathcal{L}_\eta \, \Phi \,\,\, = \,\,\, 0  \quad \quad \hbox{and} \quad \quad \Phi \, (t, \theta) \,\, = \,\, \hbox{$\sum_{j \geq n+1}$} \,\,\Phi^j(t) \, \phi_j\,(\theta)\,\, .
$$
Since $\Phi \in  C^{2,\b}_{-\d} (D_\eta)$, we have that
$$
|\Phi|\, (t ,\theta) \,\, \leq \,\, C \cdot (\cosh t)^{-\d}
$$
for some fixed $C>0$. On the other hand, the maximum principle (which holds when $\mathcal{L}_\eta$ acts on the high frequencies, see \cite{mn}) gives 
$$
|\Phi|\, (T ,\theta) \,\, \leq \,\, C \cdot (\cosh T)^{-\d}
$$
for every $T \in \R$. Letting $T \rightarrow + \infty$ we deduce that $\Phi\equiv 0$ and the proof is complete. 
\end{proof}
Using the fact that $\mathcal{L}_\eta$ is formally selfadjoint, it is standard to deduce (see \cite{mpu}) that 
\bea
\mathbb{L}(v_{D, \eta}, g_{cyl}) \,\,: \,\, C^{2,\b}_\d (D_\eta) & \longrightarrow & C^{0,\b}_\d (D_\eta)
\eea
is surjective for $\d >1$, $\d \notin I_\eta$. Following \cite{mpu} we are going to improve these first issue by showing that the surjectivity can be obtained on a smaller space. To do that it is convenient to set
\bea
W(D_{\eta,R}) & := &  {\rm span} \, \{ \, \chi_R \,\Psi^{j, \pm}_\eta  \,\, : \,\, j = 0, \ldots, n \, \} \,\, ,
\eea 
where $\chi_R$ is a non decreasing smooth cut-off function which is identically equal to $1$ for $t \geq R$ and which vanish for $t \leq R-1$. A simple adaptation of the ODE argument used in \cite[Proposition 2.7]{mpu}, gives us the following
\begin{lem}
Let $1 < \d < \bar\d (n,k) $, then the operator
\bea
\mathbb{L}(v_{D, \eta}, g_{cyl}) \,\,: \,\, C^{2,\b}_{-\d}  (D_\eta) \, \oplus \, W(D_{\eta,R}) & \longrightarrow & C^{0,\b}_{-\d} (D_\eta)
\eea
is surjective.
\end{lem}

\subsection{A linear Dirichlet problem on the half cylinder $\R^+ \times \mathbb{S}^{n-1}$}
In this subsection we study the Dirichlet problem 
\begin{eqnarray}
\label{DIRICHLET}
 \left\{
  \begin{split}
\mathbb{L}(v_{D, \eta}, g_{cyl})  \, [w  ] \,\, = \,\, f\,\, & \quad \,\, \text{in}
\quad (R \,, +\infty) \times \mathbb{S}^{n-1}\,\,,
\\
w \,\,  = \,\, 0 \,\,  & \quad \,\,\text{on}  \quad \{ R\} \times \mathbb{S}^{n-1}
\,\, ,
  \end{split}
\right.
\end{eqnarray}
for which will prove a well posedness result 
 in the next Proposition \ref{APER}. As it will be apparent from the proof, this result heavily relies on a proper choice of the value of $R$. Loosely speaking, the correct choice of $R$ has to compensate the lack of maximum principle for the linear operator $\mathbb{L}(v_{D, \eta}, g_{cyl})$. The same kind of problem will show up also in the next subsection \ref{sezione-cilindro-finito}.
%In \eqref{DIRICHLET}, $R$ is, according to .......,  a real number of the form $R\,=\,m\tfrac{T_\eta}{2} \,+\,\tilde{r}$. 
%R \notin \{ m \,T_\eta/2 \,\, : \,\, m \in \mathbb{N}\, \}
For future convenience we set $D_{\eta,R} := (R \,, +\infty) \times \mathbb{S}^{n-1}$. Again, since $h$ and $v_{D, \eta}$ are bounded and periodic we study without loss of generality the 
conjugate problem
\begin{eqnarray}
 \left\{
  \begin{split}
\mathcal{L}_\eta  \, z \,\, = \,\, y\,\, & \quad \,\, \text{in}
\quad D_{\eta,R}\,\,,
\\
z \,\,  = \,\, 0 \,\,  & \quad \,\,\text{on}  \quad \partial D_{\eta,R}\,\, ,
\label{linear dirichlet}
\end{split}
\right.
\end{eqnarray}
where $z = h^{(k-1)/2} w$ and $y=-C_{n,k}^{-1} v_{D,\eta}^{-1} h^{-(k-1)/2} f$. We consider now the usual eigenfunction decomposition
\begin{eqnarray*}
y \,(t,\theta) \,\,\, = \,\,\, \sum_{j=0}^{\infty} \, y^j(t)
\,\, \phi_j  (\theta) \quad\quad & \text{and} & \quad \quad z \,
(t,\theta) \,\,\, = \,\,\, \sum_{j=1}^{\infty} \, z^j (t) \,\,
\phi_j  (\theta) \,\, ,
\end{eqnarray*}
where the $\phi_j$'s indicate the eigenfunctions of the
Laplace-Beltrami operator on $(S^{n-1} ,  g_{S^{n-1}})$ which
satisfy the identities $-\Delta_\theta \, \phi_j = 
\lambda_j \,\, \phi_j\,$, with $j\in \mathbb{N}$. We also recall
that the spectrum of $\Delta_\theta$ is given by $\{\, m \, (n-2+m)
\,\,\, : \,\, \, m \in \mathbb{N} \, \}$ and that in particular the first nonzero
eigenvalue is $n-1$, with multiplicity $n$.

\medskip

\noindent In the spirit of \cite{mp}, it is convenient to treat separately the
high frequencies, i.e.,  $j \, \geq \, n+1$, and the low
frequencies, namely $j\,=\, 0, \ldots, n$. Basically, this
distinction is motivated by the fact that, depending on the size of
the $\lambda_j$, the quantity
%\big[ \, \tfrac{n-k}{k(n-1)} \,\, +
%\,\, \tfrac{n(k-1)}{k(n-1)} \, F_{\e,R} \, \big]
$ \, a_{\eta}\,\, \lambda_j \, + \, p_{\eta}  $ presents a change of
sign and this has a clear influence on the analytical properties of
our operators.

\medskip

\noindent \textbf{High frequencies: $j \, \geq \, n+1$.} We consider
the projection of $z$ and $y$ along the high frequencies
\begin{eqnarray*}
\bar{y}\, (t,\theta) \,\,\,\, := \,\,\,\, \sum_{j =
n+1}^{\infty}\, y^j(t)\,\phi_j(\theta) \quad \quad & \hbox{and} &
\quad \quad \bar{z}\, (t,\theta) \,\,\,\, := \,\,\,\, \sum_{j=
n+1}^{\infty}\, z^j(t)\,\phi_j(\theta) \,\, ,
\end{eqnarray*}
and for $T> R$ we consider the projected and truncated linear Dirichlet problem
\begin{eqnarray}
\label{linear problem DT}
 \left\{
  \begin{split}
\mathcal{L}_{\eta} \, \bar{z}  \,\, = \,\, \bar{y} \,\, & \quad
\,\, \text{in} \quad D^T_{\eta,R}\,\,,
\\
\bar{z} \,\,  = \,\, 0 \,\,  & \quad \,\,\text{on}  \quad \partial
D^T_{\eta,R} \,\, ,
  \end{split}
\right.
\end{eqnarray}
where $D^T_{\eta,R} := (R, T) \times \mathbb{S}^{n-1}$. 
In the high frequencies regime the linear problem \eqref{linear problem DT} has a clear variational structure. Indeed it is easy to see
that critical points of the Euler-Lagrange functional
\begin{eqnarray}
\label{energia-alte-frequenze}
E_T \, (\bar z) & := & \int_{ R}^T \int_{\mathbb{S}^{n-1}}\,\left( |\,\partial
_t\bar{z}\,|^2 \,\, + \,\,
a_{\eta} \, |\, \nabla_{\t} \, \bar{z} \, |_\theta^2 \,\, + \,\,
p_{\eta} \, \bar{z}^2 \,\, + \,\, 2 \, \bar{y} \, \bar{z} \,\,\,\,
\right)dt \, d\theta
\end{eqnarray}
are weak solutions of $\eqref{linear problem DT}$ (here $d\theta$ represent the volume element of the round metric $g_{S^{n-1}}$ on the $(n-1)$-dimensional sphere).
% hence by
%standard elliptic regularity theory are also smooth ones.
On the other hand, since $j\geq n+1$, we have by  \cite[Lemma 5.3]{mn} that
$\, a_{\eta}\, \lambda_j \, + \, p_{\eta} \, > \, 0$ in $D_{\eta,R}$. This
implies that the functional $E_T$ is coercive on 
$$
\big[\,H^1_0 \,
(D_{\eta,R}^T) \, \big]^\perp \, := \, \left\{ \, u \, \in \, H^1_0(D_{\eta,R}^T)\,\,\,  \left| \,\,\, \int_{\mathbb{S}^{n-1}} u( \, \cdot \, ,\theta) \,\, \phi_j(\theta) \,\,\, d\theta \,\ = \,\, 0 \,\, ,\quad \,\, j=0, \ldots, n  \, \right.\right\}\,\,,
$$ 
hence it is bounded
from below. Furthermore it is easy to check that the functional
$E_T$ is weakly lower-semicontinuous on $\big[\,H^1_0\, (D_{\eta,R}^T)\,
\big]^\perp$. Thus, using the direct method of calculus of
variations, we infer the existence of a minimizer $\bar z_T$ of $E_T$, which provides a (weak) solution of $\eqref{linear problem
DT}$. The standard elliptic theory yields the expected regularity
issues for $\bar{z}_T$ in terms of the regularity of $\bar{y}$.

\medskip
Moreover, as a particular case of \cite[Proposition 6.4]{mn}, in the high frequencies regime, there holds the following
\begin{lem}
\label{APET}
Let $|\,\d \,| \, < \, \bar \d (n,k)
%\,\sqrt{\tfrac{2n(n-k)}{k(n-1)} + \big(\tfrac{n-2k}{2k}\big)^2} 
$, then there exists a positive constant $C = C(n, k, \d)>0$ such that if $\bar z_T \in C^{2,\b}_{-\d} (D_{\eta,R}^T)$ and $y \in C^{0, \b}_{-\d} (D_{\eta,R}^T)$ verify \eqref{linear problem DT}, then we have
\be
\label{apeT}
\nor{\,\bar z_T}{C^{2,\b}_{-\d} (D_{\eta,R}^T)} & \leq & C \,\,\, \nor{\,\bar y\,}{C^{0,\b}_{-\d} (D_{\eta,R}^T)} \,\, ,\quad \quad 
\ee
for every $T>R$.
\end{lem}
%\begin{proof}
%The proof relies on a barrier argument. In particular, let us introduce, for $\d$ such that $|\, \d \,|\, < \, \bar \d (n,k)$,
%\begin{eqnarray*}
%\bar{v}\,\,\,\,:=\,\,\,\,e^{\d t}\sum_{j=n+1}^{\infty}w^{j}\phi_{j}(\theta)\,\,\,\,=\,\,\,\,e^{\d t}\bar{w}(\theta).
%\end{eqnarray*}
%There holds
%\begin{eqnarray}
%\label{barrier1}
%&\disp \mathcal{L}_{\eta} \, \bar{v}\,\,=\,\,\d^2e^{\d t}\bar{w} - [\sum_{j=n+1}^{\infty}(\lambda_j\a_\eta + p_\eta)w^{j}\phi_{j}(\theta)]e^{\d t}\\
%&\disp \le \,\,\d^2e^{\d t}\bar{w} - \bar{\d}^2 (n,k)e^{\d t}\bar{w}\le 
%(\d ^2 - \bar{\d}^2(n,k))e^{\d t}\bar{w},
%\end{eqnarray}
%where we have used the fact that, for $j\ge n+1$, 

%
%......
%......
%......
%\end{proof}
Using the fact that the estimate is independent of $T$, it is easy to obtain a solution $\bar z$ to \eqref{linear dirichlet} by letting $T \rightarrow +\infty$. Moreover it is clear that $\bar z$ verifies the estimate
\be
\label{ape}
\nor{\,\bar z}{C^{2,\b}_{-\d} (D_{\eta,R})} & \leq & C \,\,\, \nor{\,\bar y\,}{C^{0,\b}_{-\d} (D_{\eta,R})} \,\, ,\quad \quad 
\ee
with the same constant $C$ as in Proposition \ref{APET}.

\medskip

\noindent{\bf Low frequencies:} $j=0,\ldots,n$. Here we start by
considering the projection of our original problem (\ref{linear
dirichlet}) along the eigenfunction $\phi_0$, obtaining
\begin{eqnarray}
\label{ODE 0}
 \left\{
  \begin{split}
\,\,\mathcal{L}^{\,\,0}_{\eta} \,\, z^{0}  \,\, = \,\, y^{0} \,\, &
\quad \,\, \text{in} \quad (\, R \, , \, + \infty \,)\,\,,
\\
z^{0} ( R) \, = \,\, 0 \,\,.&
  \end{split}
\right.
\end{eqnarray}
where $\mathcal{L}^0_\eta \, := \, \partial_t^2 \, - \, p_\eta$. As it is evident, in this case the potential has a wrong sign, thus we are forced to use a different
approach in order to provide existence. We suppose that the right
hand side is at least continuous and we extend it to the whole $\R$
(with a small abuse of notations, we still denote this extension by
$y^{0}$). Next, following \cite{mp}, we consider, for any $T>R$, the auxiliary backward Cauchy problem
\begin{eqnarray}
\label{Backward Cauchy 0}
 \left\{
  \begin{split}
\mathcal{L}^{\,\,0}_{\eta} \,\, z  \,\, = \,\, y^{\,0} \,\, & \quad
\,\, \text{in} \quad (\,- \infty \, , \, T \,)\,\,,
\\
z \,(T) \,\, = \,\, 0 \,\,,
\\
\dot{z} \,(T) \,\, = \,\, 0 \,\,.
  \end{split}
\right.
\end{eqnarray}
Using the Cauchy-Lipschitz Theorem, we infer the existence of a
unique solution $z_T^0$ to (\ref{Backward Cauchy 0}). As we are
going to show, the weighted norms of these solutions admit a bound
which is uniform in $T$. This will allow us to produce
a solution to the problem (\ref{ODE 0}) with the wrong boundary
data, just by taking the limit of the $z_T^0$'s for $T \rightarrow
+\infty$. As a final step we will correct these boundary data by
adding a suitable multiple of the (conjugated) Jacobi field $\Phi_{\eta}^{0,+} := \, h_{}^{{(k-1)}/{2}}  \,\, \Psi^{0,+}_\eta  =  \, h_{}^{{(k-1)}/{2}}  \,\, \dot{v}_{D,\eta} $,
which lies by definition in the kernel of $\mathcal{L}_{\eta}^{\,
0}$. 
\begin{lem}
\label{APET0}
Let $0 < \d $ and $R\,=\,\tilde{m}\,T_\eta +\tilde{r}$ with $\tilde{m}\,\in\mathbb{N}$ and $\tilde{r}\,\in \,\mathbb{R}$ sufficiently small. Then, there exists a positive constant $C = C(n, k, \d)>0$ such that if $ z^0_T \in C^{2,\b}_{-\d} (R,T)$ and $y^0 \in C^{0, \b}_{-\d} (R,T)$ verify \eqref{Backward Cauchy 0}, then we have
\be
\label{apeT0}
\nor{\, z^0_T}{C^{2,\b}_{-\d} (R,T)} & \leq & C \,\,\, \nor{\, y^0\,}{C^{0,\b}_{-\d} ((R,T)} \,\, ,\quad \quad 
\ee
for every $T>R$.
\end{lem}
\begin{proof} We only prove the weighted $C^0$-estimate, since the weighted $C^{2,\b}$-estimate will follow by standard scaling arguments. We want to establish the $T$-uniform bound
\bea
\nor{\, z^0_T}{C^{0}_{-\d} (R,T)} & \leq & C \,\,\, \nor{\, y^0\,}{C^{0}_{-\d} ((R,T)} \,\, .\quad \quad 
\eea
We argue by contradiction. If the statement does not
hold, then it is
possible to find a sequence of triples $(T_i,
{z}^{0}_{T_i}, {y}^{0}_i)$ such that
\begin{itemize}
\item \quad$\mathcal{L}^{0}_{\eta} \, {z}^{0}_{T_i}  =  {y^{\,0}_i}$ \,\,
in $(R , T_i)$  \quad \quad and \quad\quad ${z}^{0}_{T_i}
(T_i) \,= \,0 \, = \, \dot{z}^{ 0}_{T_i} (T_i)$ \quad \quad  for
every $i \in \mathbb{N}$\,\,,
\item $\quad \nor{\,{z}^{0}_{T_i} \,}{\mathcal{C}^{\,0}_{-\d} ( R , T_i )} \, = \,
1$ \quad  for every $i \in \mathbb{N}$\,\,,
\item \quad $\nor{\,y^{0}_i\,}{\mathcal{C}^{\,0}_{-\d} ( R , T_i)}
\longrightarrow \,\, 0$ \quad  as  $i \rightarrow +\infty$\,\,.
\end{itemize}
From the second point we infer the existence of a point $t_i \in
( R , T_i)$ such that
\begin{eqnarray*}
\begin{split}
\sup_{t \in ( R , T_i)} \,\,e^{\d t}\,\, |\, z^{\,
0}_{T_i}|\,(t) \,\,\,& = & e^{\d t_i}\,\, |\, z^{\,
0}_{T_i}|\,(t_i)  \,\,\,& = & 1   \,\,.
\end{split}
\end{eqnarray*}
In fact on the half cylinder the weighting function $\cosh t $ can be replaced by $e^t$ in the definition of the weighted norms. This yields an equivalent norm and simplify the computations of this subsection. It is now convenient to set, for every $i \in \mathbb{N}$,
\begin{eqnarray*}
z_i \, (t) \,\,\, := \,\,\, e^{\d t_i} \, {z}^{\, 0}_{T_i}\,
(t+t_i) \quad & \hbox{and} & \quad y_i \, (t) \,\,\, := \,\,\,
e^{\d t_i} \, {y^{\, 0}_i} \, (t+t_i) \,\, ,
\end{eqnarray*}
where the point $t$ varies now in $ ( \, -\log R-t_i \, , \,
T_i-t_i \,)$, for every $i \in \mathbb{N}$. From these definitions it follows that
\begin{itemize}
\item \quad$\mathcal{L}^{\,0}_{\eta} \, {z}^{}_{i}  =  {y^{}_i}$ \,\,
in \, $( R - t_i, T_i- t_i)$  \quad  and \quad ${z}^{}_{i} (T_i
- t_i) \,= \,0 \, = \, \dot{z}^{\, 0}_{i} (T_i - t_i)$ \quad for
every $i \in \mathbb{N}$\,,
\item \quad
$\sup_{t \in ( R - t_i \, , \,T_i - t_i )} \,\,e^{\d
t}\,\, |\, z^{}_{i}|\,(t) \,\,\, = \,\,\, |\, z^{}_{i}|\,(0)
\,\,\, = \,\,\, 1$  \quad  for every $i \in \mathbb{N}$\,\,,
\item \quad $ \sup_{t \in (- R - t_i \, , \,T_i - t_i )} \,\,e^{\d
t}\,\, |\, y^{}_{i}|\,(t) \,\, \longrightarrow \,\, 0$ \quad  as
$i \rightarrow +\infty$\,\,.
\end{itemize}
We are now ready to let $i \rightarrow +\infty$ and study the
different limit situations in order to get a contradiction. As a
first step we remark that (up to a subsequence) the intervals
$(R - t_i \, , \, T_i - t_i )$ converge to an interval
$(\beta^- \, , \, \beta^+)$ which is nonempty. In fact, since $ R - t_i \, \leq \, 0$ and $T_i -t_i \, \geq \, 0$, we have
immediately that $\beta^- \in \R^- \cup \{ - \infty \}$ and $\beta^+
\in \R^+ \cup \{ +\infty \}$. Moreover, we claim that $\beta^+$ is
strictly positive. In fact, if it would not be the case, then we
would have that up to a subsequence $T_i - t_i \, \rightarrow \, 0$.
Since $|\,z_i  |\, (0)  \, = \, 1$ and $z_i \, (T_i - t_i) \, = \,
0 \, = \, \dot{z}_i \, (T_i - t_i)$ for every $i \in \mathbb{N}$,
the quantities $|\, \partial_t  {z}_i \,|$'s must explode in the
intervals $(\,T_i-t_i-1 \,,\, T_i-t_i \, )$, as $i \rightarrow
+\infty$. On the other hand, from the hypothesis on $z_i$ and $y_i$
it follows easily that
\begin{eqnarray*}
|\,\partial^2_t\, {{z}}_i |\, (t) & \leq & C \, e^{- \d \,
(T_i - t_i)} \,\, ,
\end{eqnarray*}
on the interval $(\,T_i-t_i-1 \,,\, T_i-t_i \, )$. Since we are
supposing that $T_i-t_i \rightarrow 0$, this inequality tells us
that the second derivatives of ${z}_i$ are uniformly bounded as $i
\rightarrow + \infty$. The fact that $\partial_t {z}_i \, (T_i-t_i)
\, = \, 0$ for every $i \in \mathbb{N}$ implies that the first
derivatives of ${z}_i$ also admit a uniform bound on $(\,T_i-t_i-1
\,,\, T_i-t_i \, )$ as $i \rightarrow +\infty$, which is a
contradiction. Hence we have that $(\beta^-, \, \beta^+)$ is always
nonempty.

\medskip

\noindent The equation satisfied by the $z_i$'s implies that  there exists a function $w_\infty$ such that $w_i
\rightarrow w_\infty$ in $\mathcal{C}^{\, 1}_{loc} (\,\beta^-,\, \beta^+)$. In particular, the function $z_\infty$ verifies the homogeneous equation
\begin{eqnarray}
\label{limit equation 0} \mathcal{L}^{\, 0}_\eta \, z_\infty & = &
0 \quad \,\, \,\, \,\, \hbox{in} \,\, (\,\beta^-,\, \beta^+ ) \,\,,
\end{eqnarray}
in the sense of distributions. As a consequence $z_\infty$ can be written as a linear combination of the Jacobi fields $\Phi_\eta^{0,-}$ and $\Phi_\eta^{0,+}$, namely, there exists $A, B \in \R$ such that
$$
z_\infty \,\, = \,\, A \, \Phi_\eta^{0,-}  + \, B \Phi_\eta^{0,+} \,\, .
$$
Moreover,
the hypothesis on $|\,z_i |\,(0)$ implies at once that
$ |\,z_\infty |\, (0) \,\, =\,\, 1$. Thus $z_\infty$ is non trivial. When $\beta^{+}  <  +\infty$, then the Cauchy data for the limit problem are given by $z_\infty \,(\beta^{+}) \, = \, 0 \, = \, \dot{z}_\infty  \,(\beta^+) $, thus $z_\infty \equiv 0$ and we have a contradiction. If $\beta^+  =  + \infty$, then the decay prescription $|\, z_\infty |(t) \,
\leq \, e^{- \d \, t}$ with $\d >0$ implies that both the constants $A$ and $B$ must be zero, contradicting the non triviality of $z_\infty$.
\end{proof}
\noindent Since the estimate \eqref{apeT0} is independent of the
parameter $T >  R$, we let $T \rightarrow +\infty$ and we obtain
a function $\hat{z}^{\,0}$ which verifies the identity
\begin{eqnarray*}
\mathcal{L}^{\,\,0}_{\eta} \,\, \hat{z}^{\,0}  \,\, = \,\, y^{\,0}
\,\, \quad \hbox{in} \,\, (R,+\infty) \,\, 
\end{eqnarray*}
together with the $T$-uniform estimate
\begin{eqnarray}
\label{0T-uniform} \nor{\,\hat{z}^{0}
}{\mathcal{C}^{2,\b}_{-\d}( R \,,\,+\infty)} & \leq & {C}
\,\,\, \nor{\,{y}^{0}}{\mathcal{C}^{ 0, \b}_{-\d}(R\,,\,+\infty)} \,\, ,
\end{eqnarray}
where $\d> 0$ and $C$ is the same constant as in Lemma \ref{APET0}. 

\medskip

\noindent The next step amounts to correct the function
$\hat{z}^{\,0}$ to a solution of the problem (\ref{ODE 0}). This will be done by adding an element in the kernel of $\mathcal{L}^{\,
0}_{\eta}$ to the function $\hat{w}_0$ in order to fulfill the
homogeneous boundary condition at $t = R$. Here we decide to choose the (conjugated) Jacobi field $\Phi_{\eta}^{0,+} \, := \,
h_{}^{(k-1)/2}\dot{v}_{D,\eta}$. We notice that this correction is no longer a
function in $ C^{\,2, \beta}_{-\d}( R \,,+\infty)
$, with $\d>0$, since it is just bounded at
$+\infty$. With these considerations, we are now ready to set
\begin{eqnarray}
\label{def z^0} 
z^{0} (t) & := & \hat{z}^{\,0}(t) \, - \,
\frac{\hat{z}^{\,0}( R)}{\Phi^{0,+}_{\eta}( R)} \,\,
\Phi^{0,+}_{\eta}(t) \,\,.
\end{eqnarray}
It is now immediate to check that this yields a solution to
(\ref{ODE 0}). We point out that the definition of $w^{\,0}$ makes sense since $R\,=\,\tilde{m}T_\eta \,+\,\tilde{r}$ and thus $\Phi^{0,+}_{\eta}( R)\,\neq\,0$.
% \notin \{ m \,T_\eta/2 \,\, : \,\, m \in \mathbb{N}\, \}$ in order to insure that the factor $\Phi_\eta^{0,+} \, (R)$ is different from $0$.
 From the definition of $z^0$ it follows at once that its component along the Jacobi field $\Phi^{0,+}_\eta$ is bounded by $(C/\Phi^{0,+}_\eta)(R) \, \,  \nor{\,{y}^{0}}{\mathcal{C}^{ 0, \b}_{-\d}(R\,,\,+\infty)}$ with $C$ as in Proposition \ref{APET0}, hence for $\d>0$ the
solution $z^{\, 0}$ to \eqref{ODE 0} is unique in the space
$$
 C^{\,2,\beta}_{-\d}( R \, , +\infty ) \,
\oplus \,\hbox{span} \, \{ \, \Phi^{0,+}_\eta \,\} \,.
$$
\noindent Now we are ready to treat the projection of (\ref{linear dirichlet}) along the eigenfunction $\phi_j$, with $j=1,\ldots,n$
\begin{eqnarray}
\label{ODE J}
 \left\{
  \begin{split}
\mathcal{L}^{\,\,j}_{\eta} \,\, z^{j}  \,\, = \,\, y^j \,\, & \quad
\,\, \text{in} \quad (\, R \, , \, + \infty \,)\,\,,
\\
w^j \,( R) \,\, = \,\, 0 \,\, ,
  \end{split}
\right.
\end{eqnarray}
where $\mathcal{L}_\eta^{\,j} \, := \, \partial_t^2 \,- \lambda_j a_\eta \, - \, p_\eta$. Proceeding in the same manner as in the case $j=0$ we deduce that for $\d>1$ there exists a unique solution $\hat z^j$ to this problem in the space 
$$
 C^{\,2,\beta}_{-\d}( R \, , +\infty ) \,
\oplus \,\hbox{span} \, \{ \, \Phi^{j,+}_\eta \,\} \,
$$
which can be written as 
$$
z^j \,\, = \,\, \hat{z}^j \, + \, \frac{ \hat z^j (R)}{ \big [ \, \tfrac{n-2k}{2k} v_{D,\eta} (R) Ê\, - \, \dot{v}_{D,\eta} (R)  \, \big] } \,\, \Phi^{j,+}_{\eta} \,\, , \quad \quad j=1, \ldots n \,\, .
$$
Note that this definition makes sense since, thanks to $h(t)\,>\,0$ for any $t\in\mathbb{R}$ (see \eqref{h-del-F}),  there holds that 
$ \big[\tfrac{n-2k}{2k} v_{D,\eta} (R) Ê\, - \, \dot{v}_{D,\eta} (R)  \, \big]\,\neq\,0$.
% the fact that $R \notin \{ m \,T_\eta/2 \,\, : \,\, m \in \mathbb{N}\, \}$ implies that the definition makes sense.
Moreover we have as in the previous case that $\hat z^j$ verifies the estimate
\begin{eqnarray}
\label{JT-uniform} \nor{\,\hat{z}^{j}
}{C^{2,\b}_{-\d}( R \,,\,+\infty)} & \leq & {C}
\,\,\, \nor{\,{y}^{j}}{C^{ 0, \b}_{-\d}(R\,,\,+\infty)} \,\, ,
\end{eqnarray}
where now $\d >1$ and $C$ is a positive constant only depending on $n,k$ and $\d$ and the component of $z^j$ along the (conjugate) Jacobi fields $\Phi_\eta^{j,+} := h^{(k-1)/2} \Psi^{j,+}_\eta$ is bounded by $\big( \,C \, / \, [ \, \tfrac{n-2k}{2k} v_{D,\eta} (R) Ê\, - \, \dot{v}_{D,\eta} (R)  \, \big]     \, \big) \, \nor{\,{y}^{j}}{C^{ 0, \b}_{-\d}(R\,,\,+\infty)} $.

\medskip

To summarize all the result of this subsection, we define the finite dimensional function space
$$
\mathcal{W}^+ (D_{\eta,R}) \,\, := \,\, {\rm span} \, \big\{  \Psi^{j,+}_\eta \,\, : \,\, j=0, \ldots,n   \big\}
$$
and for a function $u = \sum_{j=0}^n \, a^+_j \, \Psi^{j,+}_\eta  \in  \mathcal{W}^+ (D_{\eta,R})$ we simply set 
\be
\label{normaW}
\nor{\,u \,}{\mathcal{W}^+ (D_{\eta,R})} \,\, := \,\, \hbox{$\sum_{j=0}^n$} \, |\,a_j^+| \,\, .
\ee
We thus have proved the following
\begin{pro}
\label{APER}
Let $1  <  \d  <  \bar \d (n,k) 
%\sqrt{\tfrac{2n(n-k)}{k(n-1)} + \big(\tfrac{n-2k}{2k}\big)^2}
$ and $R\,=\,\tilde{m}T_\eta \,+\,\tilde{r}$ as above, then for every $f \in C^{0,\b}_{-\d} (D_{\eta,R})$ there exists a unique solution 
$w \in C^{2,\b}_{-\d} (D_{\eta,R}) \oplus \mathcal{W}^+ (D_{\eta,R})$ to the problem 
\begin{eqnarray*}
%\label{DIRICHLET}
 \left\{
  \begin{split}
\mathbb{L}(v_{D, \eta}, g_{cyl})  \, [w  ] \,\, = \,\, f\,\, & \quad \,\, \text{in}
\quad (R \,, +\infty) \times \mathbb{S}^{n-1}\,\,,
\\
w \,\,  = \,\, 0 \,\,  & \quad \,\,\text{on}  \quad \{ R\} \times \mathbb{S}^{n-1}
\,\, .
  \end{split}
\right.
\end{eqnarray*}
Moreover we have that there exists a positive constant $C=C(n,k,\d,\eta)>0$ such that 
\begin{eqnarray*}
\label{aper+}
\nor{w}{C^{2,\b}_{-\d} (D_{\eta,R}) \oplus \mathcal{W}^+ (D_{\eta,R})} \,\,\, := \,\,\, \nor{w}{C^{2,\b}_{-\d} (D_{\eta,R})} + \,\nor{w}{\mathcal{W}^+ (D_{\eta,R})}  & \leq & C \,\,\, \nor{f}{C^{0,\b}_{-\d} (D_{\eta,R})} \,\, .
\end{eqnarray*}
\end{pro}
The same analysis can be reproduced on a domain of the form $D_{\eta,-R} := (-\infty,-R) \times \mathbb{S}^{n-1}$, in order to solve the problem
\begin{eqnarray}
\label{DIRICHLET-}
 \left\{
  \begin{split}
\mathbb{L}(v_{D, \eta}, g_{cyl})  \, [w ] \,\, = \,\, f\,\, & \quad \,\, \text{in}
\quad (-\infty , \, -R) \times \mathbb{S}^{n-1}\,\,,
\\
w \,\,  = \,\, 0 \,\,  & \quad \,\,\text{on}  \quad \{- R\} \times \mathbb{S}^{n-1}
\,\, ,
  \end{split}
\right.
\end{eqnarray}
where $R$ is a real number of the form $R\,=\,\tilde{m}T_\eta \,+\,\tilde{r}$, as in Proposition \ref{APER}. In this situation we use the finite dimensional function space
$$
\mathcal{W}^- (D_{\eta,-R}) \,\, := \,\, {\rm span} \, \big\{  \Psi^{j,-}_\eta \,\, : \,\, j=0, \ldots,n   \big\}
$$
for the corrections along the low frequencies. Obviously, for a function $u = \sum_{j=0}^n \, a^-_j \, \Psi^{j,-}_\eta  \in  \mathcal{W}^+ (D_{\eta,-R})$ the norm can be defined by 
$$
\nor{\,u \,}{\mathcal{W}^- (D_{\eta,-R})} \,\, := \,\, \hbox{$\sum_{j=0}^n$} \, |\,a_j^-| \,\, .
$$
In analogy with the previous proposition, it is straightforward to obtain the following
\begin{pro}
\label{APER-}
Let $1  <  \d  <  \bar \d (n,k)$ and $R\,=\,\tilde{m}\,T_\eta +\tilde{r}$ as above,
%with $\tilde{m}\,\in\mathbb{N}$ and $\tilde{r}\,\in \,\mathbb{R}$ sufficiently small.
%\sqrt{\tfrac{2n(n-k)}{k(n-1)} + \big(\tfrac{n-2k}{2k}\big)^2}
then for every $f \in C^{0,\b}_{-\d} (D_{\eta,-R})$ there exists a unique solution 
$w \in C^{2,\b}_{-\d} (D_{\eta,-R}) \oplus \mathcal{W}^- (D_{\eta,-R})$ to the problem \eqref{DIRICHLET-}. Moreover we have that there exists a positive constant $C=C(n,k,\d,\eta)>0$ such that 
\begin{eqnarray*}
\label{aper-}
\nor{w}{C^{2,\b}_{-\d} (D_{\eta,-R}) \oplus \mathcal{W}^- (D_{\eta,-R})} \,\,\, := \,\,\, \nor{w}{C^{2,\b}_{-\d} (D_{\eta,-R})} + \,\nor{w}{\mathcal{W}^- (D_{\eta,-R})}  & \leq & C \,\,\, \nor{f}{C^{0,\b}_{-\d} (D_{\eta,-R})} \,\, .
\end{eqnarray*}
\end{pro}
To conclude we observe that as a consequence of the analysis 
of this subsection we can also solve the same problems with nonzero boundary condition, namely for every $1<\d< \bar\d (n,k)
%\sqrt{\tfrac{2n(n-k)}{k(n-1)} + \big(\tfrac{n-2k}{2k}\big)^2}
$, every $f \in {C}^{0,\b}_{-\d} (D_\eta)$ and every boundary datum $v^\pm \in C^{2,\b}(\partial D_{\eta, \pm R})$, there exists a unique function $w^\pm \in C^{2,\b}_{-\d} (D_{\eta,\pm R}) \oplus \mathcal{W}^\pm (D_{\eta, \pm R})$ verifying 
\begin{eqnarray*}
\label{DIRICHLET PM}
 \left\{
  \begin{split}
\mathbb{L}(v_{D, \eta}, g_{cyl})  \, [\,w^\pm] \,\, = \,\, f\,\, & \quad \,\, \text{in}
\quad D_{\eta,\pm R}\,\,,
\\
w^\pm \,\,  = \,\, v^\pm \,\,  & \quad \,\,\text{on}  \quad \partial  D_{\eta,\pm R}\,\, ,
  \end{split}
\right.
\end{eqnarray*}
together with the estimate 
\begin{eqnarray*}
\label{apepm}
\nor{\,w^\pm}{C^{2,\b}_{-\d} (D_{\eta,\pm R}) \oplus \mathcal{W}^-\pm(D_{\eta,\pm R})}  & \leq & C \,\,\, \big[ \,\, \nor{f}{C^{0,\b}_{-\d} (D_{\eta, \pm R})}  \,\, + \,\,  \nor{\, v^{\pm}}{C^{2,\b}(\partial D_{\eta, \pm R})}    \, \big] \,\, ,
\end{eqnarray*}
for some positive constant $C= C(n,k, \d,\eta,R)>0$. We thus have proved 
\begin{pro}\label{iso-cylinder}
Let $1 < \d < \bar\d $ and $R\,=\,\tilde{m}\,T_\eta +\tilde{r}$ as above, then the operator
\bea
\mathbb{L}(v_{D, \eta}, g_{cyl}) \, \otimes \, \partial^\pm \,\,: \,\, C^{2,\b}_{-\d}  (D_{\eta,\pm R}) \, \oplus \, \mathcal{W}^\pm (D_{\eta, \pm R}) & \longrightarrow & C^{0,\b}_{-\d} (D_{\eta, \pm R}) \, \times \, C^{2, \b} (\partial D_{\eta, \pm R})
\eea
is an isomorphism, where $\partial^\pm  :  w \, \mapsto \, w_{|\partial D_{\eta, \pm R}}$.
\end{pro}

%%%%%%%%%%%%%%%%%%%%%%%%%%%%%%%%%%%%%%%%%%%%%%%%%%%%%%%%%%%%%%%%%%%%
\subsection{A linear Dirichlet problem on a finite cylinder $(-R,R) \times \mathbb{S}^{n-1}$}
\label{sezione-cilindro-finito}
In this subsection we study the Dirichlet problem
\begin{eqnarray}
\label{DIRICHLET-cilindro-finito}
 \left\{
  \begin{split}
\mathbb{L}(v_{D, \eta}, g_{cyl})  \, [w  ] \,\, = \,\, f\,\, & \quad \,\, \text{in}
\quad (-R \,, R ) \times \mathbb{S}^{n-1}\,\,,
\\
w \,\,  = \,\, 0 \,\,  & \quad \,\,\text{on}  \quad \{ R, -R\} \times \mathbb{S}^{n-1}
\,\, ,
  \end{split}
\right.
\end{eqnarray}
for which we are going to prove the following 
\begin{pro}
\label{APER-cil-fin}
Let $R\,=\,\tilde{m}\,T_\eta +\tilde{r}$ with $\tilde{m}\,\in\mathbb{N}$ sufficiently large and $\tilde{r}\,\in \,\mathbb{R}$ sufficiently small.
% in such a way that $\Psi^{0,\pm}_{\eta}(\pm R) \neq 0$.
Then, for any $\,f\,\in\, C^{0,\beta}((-R,R)\times\mathbb{S}^{n-1})$ there exists a unique solution
$w\,\in \,C^{2,\beta}((-R,R)\times\mathbb{S}^{n-1})$ to
\eqref{DIRICHLET-cilindro-finito}. Moreover, there exists a positive constant $C\,=\,C(n,k)$ such that  
\begin{equation}
\label{stima-dominio-compatto}
\nor{w}{C^{2,\beta}((-R,R)\times\mathbb{S}^{n-1})}\,\,\le\,\,\,C\,\,\,\nor{f}{C^{0,\beta}((-R,R)\times\mathbb{S}^{n-1})}.
\end{equation}
\end{pro}
\begin{proof} 
Thanks to the compactness of the domain, we use the Fredholm alternative to prove the well posedness of \eqref{DIRICHLET-cilindro-finito}. Thus, the existence and uniqueness of solutions to \eqref{DIRICHLET-cilindro-finito} follows from the fact that the homogeneous problem
\begin{eqnarray*}
\label{dir-unique}
 \left\{
  \begin{split}
\mathbb{L}(v_{D, \eta}, g_{cyl})  \, [w  ] \,\, = \,\, 0\,\, & \quad \,\, \text{in}
\quad (-R \,, R ) \times \mathbb{S}^{n-1}\,\,,
\\
w \,\,  = \,\, 0 \,\,  & \quad \,\,\text{on}  \quad \{ \pm R\} \times \mathbb{S}^{n-1}
\,\, ,
  \end{split}
\right.
\end{eqnarray*}
only admits the trivial solution $w \equiv 0$. This is equivalent to say that $z\,\equiv\,0$ is the unique solution of the conjugate problem 
\begin{eqnarray}
 \left\{
  \begin{split}
\mathcal{L}_\eta  \, z \,\, = \,\,0 \,\, & \quad \,\, \text{in}
\quad (-R,R)\times \mathbb{S}^{n-1}\,\,,
\\
z \,\,  = \,\, 0 \,\,  & \quad \,\,\text{on}  \quad \left\{\pm R\right\}\times \mathbb{S}^{n-1}\,\, ,
\label{linear dirichlet-coniugato}
\end{split}
\right.
\end{eqnarray}
where $z\,:=\,h^{(k-1)/2}w$.
To show that $z=0$ uniquely solves \eqref{linear dirichlet-coniugato} we project the equation along the eigenfunctions $\phi_j$ of the Laplace Beltrami operator on the sphere
 $(\mathbb{S}^{n-1},g_{\mathbb{S}^{n-1}})$. Then, as we already did in the proof of Proposition
\ref{APER}, we treat separately the high frequencies ($j\ge n+1$) and the low frequencies ($j=0,\ldots, n$). Thus, we decompose $z$ as
\begin{eqnarray}
\label{decomposizione-z}
z\, (t,\theta) \,\,\,\, = \,\,\,\, \sum_{j =
0}^{n}\, z^j(t)\,\phi_j(\theta) \,+ \,  \sum_{j=
n+1}^{\infty}\, z^j(t)\,\phi_j(\theta) \,\, .
\end{eqnarray}
Now, since in the high frequencies regime the (weak) solutions to \eqref{linear dirichlet-coniugato} can be obtained as critical points of the coercive and weakly lower semicontinuous energy defined in \eqref{energia-alte-frequenze},
%(to this purpose, see the proof of Proposition \ref{APER})
it turns out that the second sum in \eqref{decomposizione-z} is identically equal to zero. 

\medskip

To obtain the same result for the low frequencies, we start with the case $j=1, \ldots n$ and
%(for which, by the way, the maximum principle fails), 
we note that the Fourier coefficients $z^j$ can be written as a linear combination of $\Phi^{j,+}_{\eta}\,\phi_j(\theta)$ and $\Phi^{j,-}_{\eta}\,\phi_j(\theta)$, where the conjugated Jacobi fields $\Phi^{j,\pm }_{\eta}$ are defined as $\Phi^{j,\pm }_{\eta}\,:=\,h^{(k-1)/2} \Psi^{j,\pm}_\eta$. Hence, there exist real numbers $A_j$ and $B_j $ such that 
\begin{equation*}
\label{wlowfreq1}
z^{j}(t)\,=\,h^{(k-1)/2}(t)\,\left\{ A_j\big[ \tfrac{n-2k}{2k} v_{D,\eta} (t) Ê\, + \, \dot{v}_{D,\eta} (t) \, \big] \, e^t Ê\\
+B_j \big[ \tfrac{n-2k}{2k} v_{D,\eta} (t) Ê\, - \, \dot{v}_{D,\eta} (t) \, \big]\,e^{-t}\right\}.
\end{equation*} 
We are going to show that the homogeneous Dirichlet boundary conditions together with our choice of $R$ imply that all the $z^{j}$ must vanish for every $j=1,\ldots,n$. From the null boundary condition we deduce that 
%(recall \eqref{h-del-F}) 
%\begin{eqnarray}
%\label{A+B}
%&\disp (A + B)^2(\tfrac{n-2k}{2k})^2\Big[( v_{D,\eta} (R) Ê\, + \, \tfrac{2k}{n-2k}\dot{v}_{D,\eta} (R) )^2e^{2R} \\
%& \disp + ( v_{D,\eta} (R) Ê\, - \, \tfrac{2k}{n-2k}\dot{v}_{D,\eta} (R))^2e^{-2R} + 2h(R)\Big] \,\,=\,\,0.\nonumber
%\end{eqnarray}
%Now, since $h(R)\,>0$, equation \eqref{A+B} implies that $A\,=\,-B$. 
%Analogously, by computing $w^j(R)\,-\,w^j(-R)$, we obtain
\begin{eqnarray*}
\label{condizioneAB}
\left\{
  \begin{split}
(A_j + B_j)\,\, \Big[\tfrac{n-2k}{2k}\, +\,\frac{\dot{v}_{D,\eta} (R)}{v_{D,\eta} (R)}\tanh{(R)} \Big] \,&=&\,0,\\
(A_j - B_j)\, \, \Big[\tfrac{n-2k}{2k}\tanh{(R)}\, +\,\frac{\dot{v}_{D,\eta} (R)}{v_{D,\eta} (R)} \Big] \,&=&\,0,
\end{split}
\right.
\end{eqnarray*}
where we have used the fact that $v_{D,\eta}(\cdot)$ is an even function. Now, using the fact that $R$ is of the form
$R\,=\,\tilde{m}T_\eta + \tilde{r}$, it is sufficient to choose $\tilde m$ large enough and $\tilde r$ small enough in order to insure that 
%to be small enough, we obtain $A\,=\,B\,=\,0$. In fact, 
%since $\dot{v}_{D,\eta} (\tilde{m}T_\eta)\,=\,0$ for $\tilde{m}\,
%\in\,\mathbb{N}$, for any given $n$ and $k$, we can find a small $\tilde{r}\,=\,\tilde{r}(n,k)$ for which $\vert\dot{v}_{D,\eta} (R)\vert$ is small and 
  \begin{eqnarray*}
\left\{
  \begin{split}
\tfrac{n-2k}{2k}\, +\,\frac{\dot{v}_{D,\eta} (R)}{v_{D,\eta} (R)}\tanh{(R)}  \,&\neq &\,0,\\
\tfrac{n-2k}{2k}\tanh{(R)}\, +\,\frac{\dot{v}_{D,\eta} (R)}{v_{D,\eta} (R)}  \,&\neq &\,0.
\end{split}
\right.
\end{eqnarray*}
Hence, $A_j =0 =B_j$, for every $j=1, \ldots,n$.

\medskip

%
%  Thus, since $v_{D,\eta}\,>\,0$ for any $t\in \mathbb{R}$, from \eqref{condizioneAB}, we deduce that $A\,=\,B\,=\,0$.  To conclude, 
 
It remains to prove that also $z^0\,\equiv\,0$. Now, $z^0$ has this form %(recall \eqref{psio})
$$
z^0(t) \,\, = h^{(k-1)/2}(t)\,\, \left\{A_0 \, \Psi_\eta^{0,-}(t)  + \, B_0\, \Psi_\eta^{0,+}(t)\right\} \,\,, 
%\,=\,h^{(k-1)/2}(t)\,\, \left\{A_0\Psi_\eta^{0,-}(t) + B_0\,\dot{v}_{D,\eta} (t)\right\},
$$
for suitable $A_0$ and $B_0$ in $\R$. From the homogeneous boundary conditions combined with the fact that $\Psi^{0,+}_\eta$ is odd and $\Psi_\eta^{0,-}$ is even, we deduce at once that if $R$ is chosen as in the statement both $A_0$ and $B_{0}$ must vanish.

\medskip

Hence, the problem \eqref{DIRICHLET-cilindro-finito} is well posed. On the other hand the a priori estimate \eqref{stima-dominio-compatto} is a direct consequence of the standard elliptic regularity theory.
%we infer that
%%should be chosen in such a way that $z^0$ satifies the Dirichlet boundary condition. Thus, we aim at proving that, for $R\,=\,\tilde{m} T_\eta + \tilde{r}$ the only possible choice for $A_0$ and $B_0$ is $A_0\,=\,B_0\,=\,0$. This is indeed the case. In fact, from the relation $0\,=\,z^0(R) \,+\,z^0(-R)$ one obtains
%$$ 
%A_0\,\Big[\,\,\Psi_\eta^{0,-}(R) \,\,+ \,\,\Psi_\eta^{0,-}(-R)\,\,\Big]\,=\,0.
%$$
%Now, since $\Psi_\eta^{0,-}$ is an even function, the relation above, when $R\,=\,\tilde{m}T_\eta + \tilde{r}$, implies that $A_0 \,=\,0$ (by hypothesis there holds that $\Psi_{\eta}^{0,-}(R)\,\neq\,0$). Thus, from the relation $0\,=\,z^0(R) - z^0(-R)\,=\,2 B_0\,\dot{v}_{D,\eta} (R)$, it immediately follows that also $B_0\,=\,0$. Thus, the wellposedness of problem \eqref{APER-cil-fin} follows from the Fredholm alternative Theorem. Finally, the a priori estimate \eqref{stima-dominio-compatto} is a direct consequence of the standard elliptic regularity theory. 
\end{proof}

%%%%%%%%%%%%%%%%%%%%%%%%%%%%%%%%%%%%%%
\section{Global linear analysis}
\label{global-linear-analysis}

The aim of this section is to provide existence, uniqueness and {\em a priori} estimates for solutions to the linear problem
\be\label{eq:global}
\mathbb{L}(u_{\e} , \bg)\,[w] &=& f \hs \hbox{in} \,\,\, M_\e.
\ee
For sake of simplicity we just consider the case where $M_\e$ is the connected sum of two Delaunay type solution $M_\e = D_{\eta_1} \sharp_\e D_{\eta_2}$. All the arguments can be trivially adapted to the general case $M_{\e} = D_{\eta_1} \sharp_\e \ldots \sharp_\e D_{\eta_N}$.

\medskip

To introduce the correct functional framework for the global linear analysis on $M_\e$ we define
\begin{eqnarray}
\label{def-norma-holder}
\nor{u}{C^{m,\beta}_{\d,\gamma}(M_\e)}\,\,\,:=\,\,\,
%\hbox{$\sum_{i=1}^2$} \,\,
\nor{u}{C^{m,\beta}_\d(D_{\etu} \setminus B(p_1,1))}  + \, 
\nor{u}{C^{m,\beta}_\d(D_{\etd} \setminus B(p_2,1))}
+  \, \nor{u}{C^{m,\beta}_{\gamma}(N_\e)} \,\,,
\end{eqnarray}
where the first two norms are defined as in Subsection \ref{linearcyl} and the norm in the neck region $N_\e$ 
is defined by
\begin{eqnarray}
\label{norma-neck}
\nor{u}{C^{m,\beta}_{\gamma}(N_\e)}  & := & 
\sup_{N_\e} Ê\,\,\, \hbox{$\sum_{j=1}^m$} (\e\cosh t)^{\gamma}
Ê\,|\nabla^j u | \, (t, \theta) 
\,\, \\
& &+  \,\, \sup_{ t \,\in \,(\log \e, -\log \e) } Ê\,\, (\e\cosh t)^{\gamma}\, [\,u\,]_{C^{m,\b}(\,(t-1, t+1) \times \mathbb{S}^{n-1}\,)} \,\, . \nonumber
\end{eqnarray} 
Note that, again, $|\,\cdot\,|$ and $\nabla$ are respectively the norm and the Levi-Civita connection of the cylindrical metric $g_{cyl}$ whereas $[\,\cdot\,]$ stands for the usual H\"older seminorm.
As a consequence, we introduce the following weighted H\"older space
\begin{eqnarray*}
C^{m,\beta}_{\d,\gamma} (M_\e) & := & \{ u \in C^{m,\beta}_{loc}(M_\e) \,\, : \,\, \nor {u}{C^{m,\beta}_{\d,\gamma}(M_\e)} 
%\,:=\, \nor{ (\cosh t)^{-\d} u}{C^m} \,
< Ê+\infty Ê Ê\} \,\, .\end{eqnarray*}

\subsection{Global and uniform {\em a priori} estimates on $M_\e$}

We recover from \cite{cat-maz} the following removable singularities result
\begin{lem}[Removable singularities]\label{removable_lem} Let $g=(1+b)^{\frac{4k}{n-2k}}g_{\mathbb{R}^{n}}$ be a conformally flat metric defined on a geodesic ball $B(p,1)$ verifying the equation
$$
\sigma_{k}(g^{-1}A_{g})\,\, = \,\, 2^{-k}\hbox{${n \choose k}$} \,\, .
$$ Suppose $\bar{w}$ satisfies in the sense of distributions
$$
\mathbb{L}(1+b ,g_{\mathbb{R}^n})[\bar{w}] \,\, = \,\, 0\quad\hbox{on}\quad B^{*}(p,1)$$ 
with $|\bar{w}(q)|\leq C |dist_{g}(q,p)|^{-\mu}$ for any $q\in B^{*}(p,1)$ for some positive constant $C>0$ and for some weight parameter $0<\mu<n-2$. Then $\bar{w}$ is a bounded smooth function on $B(p,1)$ and satisfies the equation above on the entire ball. 
\end{lem}
We are now in the position to prove the following $\e$-uniform {\em a priori} estimate for solutions to \eqref{eq:global}.
\begin{pro}
\label{stime:local2} 
Suppose that $1<\d< \bar \d \,(n,k)
%\sqrt{\frac{2n(n-k)}{k(n-1)} + \big(\frac{n-2k}{2k}\big)^2}
$, Ê$ 0 < \gamma < {(n-2k)}/{k}$ and let $w\in C^{2,\beta}_{-\d,\gamma-\frac{n-2k}{2k}} (M_\e)$ and 
$f\in C^{0,\beta}_{-\d,\gamma-{(n-2k)}} (M_\e)$ be two functions satisfying 
\begin{eqnarray*}
\mathbb{L}(u_{\e},\bg)\,[w] &= & f \hs  \hbox{in} \,\,\, M_\e \,\, .
\end{eqnarray*}
Then there exist $C=C(n,k,\b,\gamma, \d)>0$ and $\e_{0}=\e_{0}(n,k,\gamma , \d)$ such that, for every $\e\in(0,\e_{0}]$, we have
\be
\label{stima-globale-1}
\Vert w\Vert_{C^{2,\b}_{-\d,\gamma-\frac{n-2k}{2k}}} \,\, \leq \,\, ÊC \, \,\Vert f\Vert_{C^{\,0, \b}_{ -\d, \gamma - (n-2k)}}\,.
\ee
\end{pro}
\begin{proof} 
Before starting the proof, we advise the reader that we will prove \eqref{stima-globale-1} using a different, albeit equivalent, norm. 
With a little abuse of notation, we introduce the norm
\begin{eqnarray*}
\label{norma-equivalente}
\nor{u}{C^{m,\beta}_{-\d,\gamma}(M_\e)}\,\,\,:=\,\,\,
%\hbox{$\sum_{i=1}^2$} \,\,
\max\left\{\nor{u}{C^{m,\beta}_{-\d}(D_{\etu} \setminus B(p_1,1))}, \, 
\nor{u}{C^{m,\beta}_{-\d}(D_{\etd} \setminus B(p_2,1))},
  \, \nor{u}{C^{m,\beta}_{\gamma}(N_\e)}\right\} \,\,,
\end{eqnarray*} 
which is clearly equivalent to \eqref{def-norma-holder}. We will just provide the uniform weighted $C^0$--bound
\begin{eqnarray*}
\label{C0bound} 
\Vert w\Vert_{C^{0}_{-\d,\gamma - \frac{n-2k}{2k}}} \,\, \leq  \,\,ÊC \, \,\Vert f\Vert_{C^{\,0}_{- \d, \gamma - (n-2k)}}\,\,,
\end{eqnarray*}
since the uniform weighted $C^{2,\b}$--bound will follows easily from standard scaling argument, see \cite{pacard}. To prove the above estimate
%\eqref{C0bound}
 we argue by contradiction. Suppose that there exists a sequence $(\e_{i}, w_{i}, f_{i})$, $i \in \mathbb{N}$, such that
\begin{itemize}
\item $\e_{i} \, \longrightarrow Ê\, 0, \quad \mbox{ as } i\rightarrow + \infty,$
\item $\Vert w_{i}\Vert_{C^{0}_{-\d,\gamma-\frac{n-2k}{2k}}} \equiv 1 ,\quad \,\, i \in \mathbb{N}$, 
\item $ \Vert f_{i}\Vert_{C^{\,0}_{-\d, \gamma Ê- (n-2k)}} \longrightarrow 0 ,\quad \mbox{ as } i\rightarrow + \infty$
\end{itemize}
and 
\bea
\mathbb{L}(u_{\e_{i}},\bg)\,[w_{i}] &=& f_{i} \hs \hbox{in} \,\, M_\e.
\eea
Now, up to a not relabelled subsequence of $i$, one has to face with the following two cases:
\begin{enumerate}
\item[\bf{1.}] \,\,$\nor{w_i}{C^{0}_{-\d}(D_{\eta_{j}}\setminus B(p_j,1))} = 1, \,\,\, \hbox{ for any }Êi\in \mathbb{N}$ \,\, \hbox{and for $j=1$ or $2$}.
\item[\bf{2.}] \,\,$\nor{w_i}{C^{0}_{\gamma}(N_\e)} =1 \,\,\,\hbox{for any } i\in \mathbb{N}$.
\end{enumerate}
The second case has been fully analysed in
\cite[Proposition 4.4, case 2. and case 3.]{cat-maz} to which we refer for the 
details. 
Concerning the first one, we note that there is no loss of generality
in restricting the analysis only to $D_{\etu}$ (recall that on $D_{\etu}$ we use coordinates $r_1$ and $\t$, according to Section \ref{s:as}). Secondly, it is natural to split the case 1 into two subcases. The first subcase appears when there exists $M\,>\,0$ and
a subsequence of points $q_i=(r_i, \theta_i)$'s such that
 $r_i\in [-M,M]
\times \mathbb{S}^{n-1}$ 
and $(\cosh r_i)^\d|w_i| (r_i, \theta_i)\ge 1/2$. The second subcase is when the points $q_i$'s leave every compact set of $D_{\eta_1}$. 
%all but finitely many $q_i = (r_i,\t_i)$'s are such that $M < |r_i|$ and again $(\cosh r_i)^\d|w_i| (r_i, \theta_i)\ge 1/2$. 
However, it follows from \eqref{apepm} that this second subcase can always be reduced to the first one. The assumption $q_i\in [-M,M \,]
\times \mathbb{S}^{n-1}$ implies that, up to a not relabelled subsequence,
there holds that $q_i \rightarrow q_\infty$, $w_i\rightarrow w_\infty$ in 
$C^{2}_{loc}(D_{\etu}\setminus \left\{p_1\right\})$ and $f_i\rightarrow 0$ in Ê
$C^{0}_{loc}(D_{\etu}\setminus \left\{p_1\right\})$. Thus, the function $w_\infty$ is nontrivial, since $|w_\infty| (q_\infty) \geq (\cosh M)^{-\d} / 2$, and solves in the sense of distributions the limit problem 
\begin{eqnarray*}
\label{eq:limit1}
\mathbb{L}(1,g_1) \,\, [\,  u_1^{-1} \, w_{\infty}] \,\, = \,\, 0\quad\quad \hbox{in}\,\,\,D_{\eta_1}\setminus\{p_{1}\}  \quad & \,\,\hbox{with} & \quad \nor{\, w_\infty}{C^0_{-\d} (D_{\eta_1} \setminus B(p_1,1)\,)} \,\, = \,\, 1 \,\, ,
\end{eqnarray*}
where $u_1$ is the function used in the construction of the approximate solutions (see Section \ref{s:as}), which we have set to be equal to $1$ in $D_{\etu} \setminus B(p_1,1)$. If we show that the limit problem is solved by $ u^{-1}_1 \, w_\infty$ on the whole $D_{\etu}$, then, using Lemma \ref{inj.} (which is in force thanks to the fact that $1<\d$), we will reach the desired contradiction. To remove the singularity at $p_1$, we observe that on $B(p_1,1)$ we can always write  
$$
g_{1}=(1+b_{1})^{\frac{4k}{n-2k}}g_{\R^{n}},
$$ 
with $b_{1}(q)=\bigo{|dist_{g_{1}}(q,p_{1})|^{2}}$. Hence, thanks to the conformal equivariance property \eqref{confequilin}, we have that the limit equation can be rewritten as 
$$
\mathbb{L} \, \big( \,(1+b_{1}) \, , g_{\R^n} \big) \,\, [(1+b_{1})u_{1}^{-1}w_{\infty}] \,\, = \,\, 0\quad\hbox{on}\quad B(p_{1},1) \setminus \{ p_1\} \,\, .
$$ 
Recalling that $g_{1}$ solves the $\sigma_{k}$--Yamabe equation and that 
$$
|(1+b_{1})u_{1}^{-1}w_{\infty}|(q) \,\, \leq \,\, C \,\,|dist_{g_{1}}(q,p_{1})|^{-\g} \,\, ,
$$
we can apply Lemma \ref{removable_lem} to obtain that $u_1^{-1} w_{\infty}$ extends through $p_{1}$ to a nontrivial smooth solution of 
\begin{eqnarray*}\label{contraddiction}
\mathbb{L}\, ( 1 \,, g_1) \,\, [u_{1}^{-1}w_{\infty}] \,\, = \,\, 0\quad\hbox{in}\quad D_{\etu} \,\, .
\end{eqnarray*}
This completes the proof.
\end{proof}
As a consequence of the {\em a priori} estimates, we obtain the following
\begin{cor} 
\label{INJ}
Suppose that $1<\d< \bar \d \,(n,k)$
%\sqrt{\frac{2n(n-k)}{k(n-1)} + \big(\frac{n-2k}{2k}\big)^2}
andÊ$ \,\,0 < \gamma < {(n-2k)}/{k}$, then there exists a positive real number $\e_{0}(n,k,\gamma , \d) >0 $ such that, for every $\e\in(0,\e_{0}]$ the operator
\begin{eqnarray*}
\mathbb{L}(u_{\e},\bg) \,\, : \,\, C^{2,\beta}_{-\d,\gamma-\frac{n-2k}{2k}} (M_\e)   &  \longrightarrow & \,\, C^{0,\beta}_{-\d,\gamma-{(n-2k)}} (M_\e)
\end{eqnarray*}
is injective.
\end{cor}
The next task is to provide surjectivity for $\mathbb{L}(u_\e, \bg)$. Unfortunately, with this choice of the weight parameter, which will turn out to be suitable for the nonliner analysis, the surjectivity cannot be recovered. The duality theory would suggest the use of the spaces with weight parameter $\d$ insted of $-\d$, but as it is remarked in \cite{mpu}, these spaces are definitely too large. In particular they contain functions which may grow too fast at $\pm \infty$ and even worst which are not {\em integrable}, in the sense which is made precise below. To overcome this difficulty, one can take advantage of the Fredholm character of the operator (which is actually the case, when $\d$ is not an indicial root) by considering a finite dimensional extension of the domain,  the so called {\em deficiency space}. Of course, there are several different options for the choice of such a space (for example a different approach is contained in \cite{mpu}). We start by introducing the following spaces
\begin{eqnarray*}
W(D_{\etu,\Ru})  := \,  {\rm span} \, \{ \, \chi_{\Ru '} \,\Psi^{j, \pm}_{\eta_1}  \, : \, j = 0, \ldots, n \, \} & \hbox{and} & 
W(D_{\etu,-\Ru})  := \,  {\rm span} \, \{ \, \chi_{-\Ru'} \,\Psi^{j, \pm}_{\eta_1}  \, : \, j = 0, \ldots, n \, \} \\
W(D_{\etd,\Rd})  := \,  {\rm span} \, \{ \, \chi_{\Rd'} \,\Psi^{j, \pm}_{\eta_2}  \, : \, j = 0, \ldots, n \, \}  & \hbox{and} & 
W(D_{\etd,-\Rd})  := \,  {\rm span} \, \{ \, \chi_{-\Rd'} \,\Psi^{j, \pm}_{\eta_2} \, : \, j = 0, \ldots, n \, \}  \\
\mathcal{W}^{+}(D_{\etu,+ \Ru})  := \,  {\rm span} \, \{ \, \chi_{+ \Ru'} \,\Psi^{j, +}_{\eta_1}  \, : \, j = 0, \ldots, n \, \}  & \hbox{and} & 
\mathcal{W}^+(D_{\etd,+ \Rd})  := \,  {\rm span} \, \{ \, \chi_{+ \Rd'} \,\Psi^{j, +}_{\eta_2}  \, : \, j = 0, \ldots, n \, \} 
\\
\mathcal{W}^{-}(D_{\etu,- \Ru})  := \,  {\rm span} \, \{ \, \chi_{- \Ru'} \,\Psi^{j, -}_{\eta_1}  \, : \, j = 0, \ldots, n \, \}  & \hbox{and} & 
\mathcal{W}^-(D_{\etd,- \Rd})  := \,  {\rm span} \, \{ \, \chi_{- \Rd'} \,\Psi^{j, -}_{\eta_2}  \, : \, j = 0, \ldots, n \, \} 
\end{eqnarray*}
where $\chi_{\Ru'}$ is a non decreasing smooth cut-off function defined on $D_{\eta_1}$ which is identically equal to $1$ for $t \geq \Ru'$ and which vanish for $t \leq \Ru'-1$ (with $\Ru'-1 >\Ru$. The cut-off function $\chi_{-\Ru '}$ is defined by the relationship $\chi_{-\Ru'}(r) \, = \, \chi_{\Ru'}(-r)$, and $\chi_{\Rd'}$ and $\chi_{-\Rd'}$ are defined on $D_{\eta_2}$ in an analogous fashion. Moreover, the parameters $\Ru'$ and $\Rd'$ are chosen in such a way that $\Ru' -1>\Ru$ and $\Rd' -1 > \Rd$, in order to apply the analysis of the previous section. 

\medskip

We observe that all the functions in these spaces are {\em integrable} at $\pm \infty$ in the sense that they are linear combinations of infinitesimal generators of families of conformal variations of the original Delaunay type solutions $v_{D, \eta_1}$ and $v_{D, \eta_2}$. This fact has an important geometrical meaning which will be made clear later and which will be used in the nonlinear framework to justify the choice of a correction $w$ with components lying in these spaces. To continue, we define the full deficiency space $\overline{W}(M_\e)$ by
\begin{eqnarray*}
\overline W(M_\e) & := & W(D_{\etu,\Ru}) \oplus W(D_{\etu,-\Ru}) \oplus W(D_{\etd,\Rd}) \oplus W(D_{\etd,-\Rd})  
%& = & \mathcal{W}^+(D_{1,R}) \oplus \mathcal{W}^-(D_{1,-R}) \oplus\mathcal{W}^+(D_{2,S}) \oplus\mathcal{W}^-(D_{2,-S}) \,\,\, =: \,\,\, \overline{\mathcal{W}} (M_\e)
\,\,. 
\end{eqnarray*}
We notice that since all of these spaces are finite dimensional, we can choose any norm on them. To be definite we always consider the norm given by the sum of the absolute value of the components.

\medskip

The importance of these spaces for the linear analysis is clarified by the following Proposition, which can be deduced combining Corollary \ref{INJ} with \cite[Theroem 12.4.2]{pacard},
\begin{pro}
\label{pac nostrum}
Suppose that $1<\d< \bar \d \,(n,k)$
%\sqrt{\frac{2n(n-k)}{k(n-1)} + \big(\frac{n-2k}{2k}\big)^2}
andÊ$ \,\,0 < \gamma < {(n-2k)}/{k}$, then there exists a positive real number $\e_{0}(n,k,\gamma , \d) >0 $ such that, for every $\e\in(0,\e_{0}]$ the operator
\begin{eqnarray*}
\mathbb{L}(u_{\e},\bg) \,\, : \,\, C^{2,\beta}_{-\d,\gamma-\frac{n-2k}{2k}} (M_\e)  \, \oplus \, \overline{{W}} (M_\e)  &  \longrightarrow & \,\, C^{0,\beta}_{-\d,\gamma-{(n-2k)}} (M_\e)
\end{eqnarray*}
is surjective and 
$$
 {\rm dim \,\, Ker}\, \,\mathbb{L}\, (u_\e, \bg) \, \,\, =  \,\, \tfrac{1}{2} \,\, {\rm dim} \, \overline{W} (M_\e) \,\, = \,\, 4\,(n+1) \,\, .
$$
\end{pro}
Our {\em deficiency space} is defined by
\begin{eqnarray}
\label{calibro 9}
\mathcal{W}(M_\e) & := & \mathcal{W}^+(D_{\etu,\Ru}) \oplus \mathcal{W}^-(D_{\etu,-\Ru}) \oplus \mathcal{W}^+(D_{\etd,\Rd}) \oplus \mathcal{W}^-(D_{\etd,-\Rd}) \,\, .
\end{eqnarray}
Incidentally we note that in \cite{mpu}, due to the use of a different functional framework, the {\em deficiency space} is chosen to be  
\begin{eqnarray*}
\label{calibro 45}
W(M_\e) & := & W(D_{\etu,\Ru}) \oplus W(D_{\etd,\Rd})   \,\,.
\end{eqnarray*}
As expected, ${\rm dim} \, {W} (M_\e) \, = \, 4\,(n+1) \, = \, {\rm dim} \, \mathcal{W} (M_\e)$.
%
%, as expected, ${\rm dim} \, {W} (M_\e) \, = \, 4\,(n+1) \, = \, {\rm dim} \, \mathcal{W} (M_\e)$ and that the norm is defined exactly as in \eqref{normaW}

%

%

%The first one consists in transplanting the method used in \cite{mpu} and here we calibrate the deficiency space to be
%\begin{eqnarray}
%\label{calibro 45}
%W(M_\e) & := & W(D_{\etu,\Ru}) \oplus W(D_{\etd,\Rd})   \,\, 
%\end{eqnarray}
%whereas the second one relies in a somewhat different and more explicit construction and here the suitable choice of the deficiency space will be given by
%\begin{eqnarray}
%\label{calibro 9}
%\mathcal{W}(M_\e) & := & \mathcal{W}^+(D_{\etu,\Ru}) \oplus \mathcal{W}^-(D_{\etu,-\Ru}) \oplus \mathcal{W}^+(D_{\etd,\Rd}) \oplus \mathcal{W}^-(D_{\etd,-\Rd}) \,\, .
%\end{eqnarray}
%Incidentally we note that, as expected, ${\rm dim} \, {W} (M_\e) \, = \, 4\,(n+1) \, = \, {\rm dim} \, \mathcal{W} (M_\e)$ and that the norm is defined exactly as in \eqref{normaW}

\medskip

The remaining part of this section will be devoted to prove the core of our linear analysis, namely the following 
\begin{pro}
Suppose that $1<\d< \bar \d \,(n,k)$
%\sqrt{\frac{2n(n-k)}{k(n-1)} + \big(\frac{n-2k}{2k}\big)^2}
andÊ$ \,\,0 < \gamma < {(n-2k)}/{k}$, then there exists a positive real number $\e_{0}(n,k,\gamma , \d) >0 $ such that, for every $\e\in(0,\e_{0}]$ the operator
\begin{eqnarray*}\label{isomorfismo-calibrato}
\mathbb{L}(u_{\e},\bg) \,\, : \,\, C^{2,\beta}_{-\d,\gamma-\frac{n-2k}{2k}} (M_\e)  \, \oplus \, \mathcal{{W}} (M_\e)  &  \longrightarrow & \,\, C^{0,\beta}_{-\d,\gamma-{(n-2k)}} (M_\e)
\end{eqnarray*}
is an isomorphism. Moreover the following $\e$-uniform a priori estimate is satisfied
\begin{eqnarray}
\label{stima-globale}
\nor{w}{C^{2,\b}_{-\d , \g - \frac{n-2k}{2k}}(M_\e)  } \,\, + \,\, \nor{w}{\mathcal{W}(M_\e)}   & \leq & C \,\,\, \nor{ \, \mathbb{L}\, (u_\e, \bg) \, [w] \, }{C^{0,\b}_{-\d, \g - (n-2k)}(M_\e) } \,\, ,
\end{eqnarray}
where the positive constant $C$ only depends on $n, k, \b , \g$ and $\d$.
\end{pro}

\begin{proof}
First of all, Proposition \ref{pac nostrum} furnishes for any $f\in C^{0,\beta}_{-\d,\gamma-{(n-2k)}} (M_\e)$ the existence of a solution $u$ to \eqref{eq:global} in the class $C^{2,\beta}_{-\d,\gamma-\frac{n-2k}{2k}} (M_\e)  \, \oplus \, \overline{{W}} (M_\e)$. We canonically decompose $\overline{W}$ as $\overline{W}\,=\,\mathcal{W} \oplus
 \mathcal{W}^\perp$, where $\mathcal{W}^\perp$ represent the orthogonal complement of $\mathcal{W}$ in $\overline W$. Consequentely, $u$ admits the decomposition 
 \bea 
 u \,\,= \,\,\hat{u} + u^{\top} + u^{\perp},
 \eea
where $\hat{u} + u^{\top}\,\in\,C^{2,\beta}_{-\d,\gamma-\frac{n-2k}{2k}} (M_\e)  \, \oplus \, \mathcal{{W}} (M_\e)$, whereas $u^\perp \,\in\, \mathcal{W}^\perp$. The aim is thus to find a correction $z\in  C^{2,\beta}_{-\d,\gamma-\frac{n-2k}{2k}} (M_\e)  \, \oplus \, \mathcal{{W}} (M_\e)$ such that, the function $w$ defined by 
\begin{equation*}
\label{def-w}
w\,\,:=\,\,\hat{u} + u^{\top} + z,
\end{equation*}
is a solution to \eqref{eq:global}. It is clear that, as soon as we are able to produce such a correction, then the surjectivity in the smaller space $C^{2,\beta}_{-\d,\gamma-\frac{n-2k}{2k}} (M_\e)  \, \oplus \, \mathcal{{W}} (M_\e)$ will be achieved. Using the linearity of the problem and the fact that 
$u$ is already a solution, one can easily deduce that the function $z$ must satisfy
%we are looking for should be a $ C^{2,\beta}_{-\d,\gamma-\frac{n-2k}{2k}} (M_\e)  \, \oplus \, \mathcal{{W}} (M_\e)$ solution of the following problem
\begin{eqnarray*}
\label{correction-z}
\mathbb{L}(u_{\e} , \bg)\,[z] &=& \mathbb{L}(u_{\e} , \bg)\,[u^\perp] \hs \hbox{in} \,\,\, M_\e.
\end{eqnarray*}
A remarkable feature of problem above
%\eqref{correction-z}
is that the right hand side is supported by construction only in $D_{\etu,\pm \Ru}\, \cup \, D_{\etd,\pm \Rd}$. More precisely, recalling that $\Ru'>\Ru$ and $\Rd'>\Rd $, we have that  
$ \mathbb{L}(u_{\e} , \bg)\,[u^\perp]$ is actually supported in the four annuli $[\Ru'-1, \Ru' ]\times \mathbb{S}^{n-1}, \, [- \Ru', -\Ru' + 1]\times \mathbb{S}^{n-1},   
[\Rd'-1, \Rd']\times \mathbb{S}^{n-1}$ and $[- \Rd', -\Rd'+ 1]\times \mathbb{S}^{n-1}.$ As a consequence, we are led to solve the following kind of problem
\begin{eqnarray}
\label{dirichlet-correction1}
 \left\{
  \begin{split}
\mathbb{L}(u_\e, \bg)  \, [\,v_{1,+}] \,\,& = &\,\, \mathbb{L}(u_\e, \bg)  \, [\,u^\perp]  \,\, \text{in}
\quad D_{\etu,\Ru}\,\,,
\\
v_{1,+} \,\, & = & \,\, 0 \quad \,\quad \quad  \,\,  \quad \,\,\text{on}  \quad \partial  D_{\etu,\Ru},
  \end{split}
\right.
\end{eqnarray}
with $v_{1,+}\in C^{2,\beta}_{-\d} (D_{\etu,\Ru})  \, \oplus \, \mathcal{{W}^{+}} (D_{\etu,R})$. Analogous problems should be posed, respectively, on $D_{\etu,-\Ru}$,  $D_{\etd,\Rd}$, $D_{\etd,-\Rd}$,
leading to the construction of $v_{1,-}, v_{2,+}, v_{2,-}$. Problem \eqref{dirichlet-correction1} is, modulo the use of the conformal equivariance property \eqref{confequilin}, of the same kind of the Dirichlet problem 
\eqref{DIRICHLET}.
%\footnote{Note that Proposition \ref{APER} uses a different deficiency space, without the cut off function.}
Thus, Proposition \ref{APER} applies giving the well posedness of
\eqref{dirichlet-correction1}. Once we have obtained these local solutions, we define the candidate correction $z$ as \begin{eqnarray}
\label{ansatz-z}
 z\,\, := \,\,
 \left\{
  \begin{split}
v_{1,+}  + \bar{v}_{1,+} \quad\text{in}
\quad & D_{\etu,\Ru}&\,\,
\\
v_{1,-}  + \bar{v}_{1,-}\quad \text{in}
\quad  & D_{\etu,-\Ru}&\,\,
\\
v_{C} \quad\quad \quad    \text{in} \quad & C_\e:=M_\e\setminus (D_{\etu,\Ru} \cup D_{\etu,-\Ru} \cup D_{\etd,\Rd} \cup D_{\etd,-\Rd})&
\\
v_{2,+}  + \bar{v}_{2,+}\quad \text{in}
\quad  &D_{\etd,\Rd}&\,\,
\\
v_{2,-}  + \bar{v}_{2,-}\quad \text{in}
\quad  &D_{\etd,-\Rd},&\,\,
  \end{split}
\right.
\end{eqnarray}
where 
$\bar{v}_{1,+}, \bar{v}_{1,-}, \bar{v}_{2,+}, \bar{v}_{2,-}, v_{C}$ are the solutions of the following problems
%\begin{eqnarray}
%\label{dirichlet-correction2}
% \left\{
%  \begin{split}
%\mathbb{L}(u_\e, \bg)  \, [\,\bar{v}_{1,+}]  \, & =  0 \hs
% \text{in}
%\quad D_{1,R}\,\,,
%\\
%\bar{v}_{1,+} \, =  \psi_{1,+} \hs  &\quad \text{on}  \quad \partial  D_{1,R},
%  \end{split}
%\right.
%\end{eqnarray}
\be
\label{dirichlet-correction2}
\left\{
\begin{split}
\mathbb{L}(u_{\e},\bg)\,[\bar{v}_{1,+}]&= 0 \hs D_{\etu,\Ru}\\
\bar{v}_{1,+} &= \psi_{1,+} \,\,\,\,\,\,\,\,\partial D_{\etu,\Ru}
\end{split}
\right.
\ee
(analogous problems for $\bar{v}_{1,-}, \bar{v}_{2,+}, \bar{v}_{2,-}$)
and 
\begin{eqnarray}
\label{dirichlet-correction3}
 \left\{
  \begin{split}
\mathbb{L}(u_\e, \bg)  \, [\,v_C] \,\, = \,\, 0 \,\,&\quad
  \text{in}
\quad C_\e\,\,,\\
 v_C \,\,  =  \,\, \psi_{1,+} \,\,  &\quad \,\,\text{on}  \quad \partial  D_{\etu,\Ru}
\\
v_C \,\,  =  \,\, \psi_{1,-} \,\,  &\quad \,\,\text{on}  \quad \partial  D_{\etu,-\Ru}
\\
v_C \,\,  =  \,\, \psi_{2,+} \,\, & \quad \,\,\text{on}  \quad \partial  D_{\etd,\Rd}
\\
v_C \,\,  =  \,\, \psi_{2,-} \,\,  &\quad \,\,\text{on}  \quad \partial  D_{\etd,-\Rd}.
  \end{split}
\right.
\end{eqnarray}
The Dirichlet data for the problems \eqref{dirichlet-correction2} and
\eqref{dirichlet-correction3} must be chosen in a proper way. Namely, the choice of $\bar{\psi}\,:=(\,\psi_{1,+}, \psi_{1,-}, \psi_{2,+}, \psi_{2,-})$
is dictated by the requirement that the function $z$ has the correct regularity for being a solution on the whole $M_\e$. In fact $z$ is certainly continuous for any choice of $\bar\psi$, but there may be a lack $C^1$-regularity across the interface $\partial D_{\eta_{1},\Ru}$. To avoid this situation, we impose the following $C^1$-matching condition
\be
\label{C1-matching}
%\left\{
\begin{split}
\partial_{\ru}(v_{1,\pm} + \bar{v}_{1,\pm})_{|\ru=\pm \Ru} &= \partial_{\ru} v_{{C}_{\left|\ru=\pm \Ru\right.}} \\
\partial_{\rd}(v_{2,\pm} + \bar{v}_{2,\pm})_{|\rd=\pm \Rd} &= \partial_{\rd} v_{{C}_{|\rd=\pm \Rd}}. 
\end{split}
%\right.
\ee
To show that the {\em ansatz} above actually yields a good definition for $z$, we adopt the following strategy. First of all, we show that problems \eqref{dirichlet-correction2} and \eqref{dirichlet-correction3} have a unique solution with $\e$-uniform a priori estimate for generic Dirichlet data. Secondly, studying the behavior of the so called {\em{Dirichlet to Neumann map}}, we will calibrate the choice of the boundary data in such a way that conditions \eqref{C1-matching} are satified.

\medskip

As we did for \eqref{dirichlet-correction1}, we note that, modulo the use of the conformal equivariance property, \eqref{dirichlet-correction2} is of the same type as \eqref{DIRICHLET PM}. Thus, Proposition \ref{iso-cylinder} applies giving, for any Dirichlet datum 
$\psi_{1,+}\in C^{2,\b}(\partial D_{\etu,  \Ru})$, a unique 
$\bar{v}_{1,+}\in C^{2,\b}_{-\d} (D_{\etu, \Ru}) \oplus \mathcal{W}^+ (D_{\etu, \Ru})$ such that
\be
\label{stimavbar}
\nor{\,\bar{v}_{1,+}}{C^{2,\b}_{-\d} (D_{\etu,\Ru}) \oplus \mathcal{W}^+(D_{\etu,\Ru})}  & \leq & C \,\,\, \nor{\, \psi_{1,+}}{C^{2,\b}(\partial D_{\etu, \Ru})}. \,\,
\ee
 Of course, the same holds for $\bar{v}_{1,-}, \bar{v}_{2,+}$, and $\bar{v}_{2,-}$. 
 
\medskip
 
To solve problem \eqref{dirichlet-correction3} we take advantage of the fact that for any fixed $\e$, the domain 
$C_\e$ is compact and thus we can apply the Fredholm alternative. Hence, we have existence and uniqueness for problem \eqref{dirichlet-correction3}, provided
\begin{equation*}
\label{dirichlet-correction3bis}
\left\{
\begin{split}
\mathbb{L}(u_{\e},\bg)\,[v]&= 0 \,\,\,\,\,\,\,\, C_\e\\
v &= 0 \,\,\,\,\,\,\,\,\partial C_\e
\end{split}
\right.
\end{equation*}
only admits the trivial solution. To prove this, we are going to show that there exists a positive constant $B>0$ independent of $\e$ such that the {\em a priori} bound
\be
\label{stimatris}
\Vert v\Vert_{C^{2,\b}_{\gamma-\frac{n-2k}{2k}}(C_\e)} \,\, \leq \,\, ÊB \, \,\Vert f\Vert_{C^{\,0, \b}_{\gamma - (n-2k)}(C_\e)}\,
\ee
is in force for solutions to
\begin{equation*}
\label{dirichlet-correction3tris}
\left\{
\begin{split}
\mathbb{L}(u_{\e},\bg)\,[v]&= f \,\,\,\,\,\,\,\, C_\e\\
v &= 0 \,\,\,\,\,\,\,\,\partial C_\e
\end{split}
\right.
\end{equation*}
%of the form
%\be\label{stimatris}
%\Vert v\Vert_{C^{2,\b}_{\gamma-\frac{n-2k}{2k}}(C_\e)} \,\, \leq \,\, ÊC \, \,\Vert f\Vert_{C^{\,0, \b}_{\gamma - (n-2k)}(C_\e)}\,.
%\ee
%Clearly, once we have proved \eqref{stimatris}, the uniqueness
%of problem \eqref{dirichlet-correction3bis} easily follows (and thus the existence for problem \eqref{dirichlet-correction3}). 
To prove \eqref{stimatris}, we use a blow-up argument similar to the one used in the proof of Proposition \ref{stime:local2}. The only difference 
%with respect to the proof of Proposition \ref{stime:local2} 
is in the treatment of the 
case ${\bf 1}$. In particular, following the argument used in the above mentioned proof and one ends up with a function $v_\infty$ such that $u_{1}^{-1}\,v_\infty$ is a non trivial smooth solution of the following boundary value problem
\begin{eqnarray*}
\label{prova1stimatris}
 \left\{
  \begin{split}
\mathbb{L}(1\,, g_{1})  \, [ u_{1}^{-1}v_{\infty}] \,\, = \,\, 0\,\, & \quad \,\, \text{in}
\quad (- R_1 , \, R_1) \times \mathbb{S}^{n-1}=:C_{0,1}\,\,,
\\
v \,\,  = \,\, 0 \,\,  & \quad \,\,\text{on}  \quad \partial C_{0,1}
\,\,
  \end{split}
\right.
\end{eqnarray*}
%Note, that the analogous \eqref{contraddiction} was posed on the whole
%$D_1$. Thus, we can not directly use the local analysis developped in the previous section to reach the desired contradiction. However, we can argue as follow.
Thus, thanks to the conformal equivariance property 
\eqref{confequilin} and 
%
%to \eqref{linear-del-expression},
%we have that 
%\eqref{prova1stimatris} could be rewritten as
%\begin{align}
%\mathbb{L}\, ( v_{D_{\etu},\etu} \,, g_{cyl}) \,\, [v_{D_{\etu},\etu}\, u_{1}^{-1}\,w_{\infty}] \,\, = - \,\, C_{n,k} \,\,
%v_{D_{\etu},\etu} \,\, Êh^{\tfrac{k-1}{2}} Ê\,\,\mathcal{L}_{\etu}\,\,[h^{\tfrac{k-1}{2}}\,\,v_{D_{\etu},\etu}\, u_{1}^{-1}\,w_{\infty}]=\, 0\quad\hbox{in}\quad C_{0,1} \,\, .
%\end{align}
%%Secondly, using that 
%%$$
%%\mathcal{L}_\eta \rightarrow \partial_{t}^{2} + \tfrac{n-k}{k(n-1)}\Delta_\theta -\left(\tfrac{n-2k}{2k}\right)^2=:\,\mathcal{L}_0\,\,\,\, \hbox{ when } \eta\rightarrow 0
%%$$
%%on the compact sets of $C_{0,1}$ and since $\mathcal{L}_0$ is coercive along all the frequencies, we have that, upon choosing the Delaunay parameter $\eta$ sufficiently small, 
%%the operator $\mathcal{L}_\eta$ becomes coercive too.
%Now,
Proposition \ref{APER-cil-fin}, we infer that $v_\infty\,\equiv\, 0$, which is a contradiction. Thus, there exists a unique $v_C$ solving \eqref{dirichlet-correction3} and such that
\be
\label{stimavc}
\Vert v_C\Vert_{C^{2,\b}_{\gamma-\frac{n-2k}{2k}}(C_\e)} \,\, \leq \,\, ÊB \, \,\Vert \bar{\psi}\Vert_{C^{\,2, \b}
%_{\gamma - \frac{n-2k}{2k}}
(\partial C_\e)}\,.
\ee
Having at hand the functions $\bar{v}_{1,+}, \bar{v}_{1,-}, \bar{v}_{2,+}, \bar{v}_{2,-}$ and $v_C$, we can introduce the Dirichlet to Neumann maps for problems \eqref{dirichlet-correction2} and \eqref{dirichlet-correction3}. For the exterior problem \eqref{dirichlet-correction2}, we define the mapping $T:\,C^{2,\b}(\partial C_\e)\longrightarrow C^{1,\b}(\partial C_\e)$, whose action is given by
\begin{equation*}\label{DNT}
T:\, \bar{\psi}:=(\psi_{1,+}, \psi_{1,-}, \psi_{2,+}, \psi_{2,-}) \longmapsto (-\partial_{\ru}\bar{v}_{1,+},\partial_{{\ru}}\bar{v}_{{1,-}};-\partial_{\rd}\bar{v}_{2,+},\partial_{\rd}\bar{v}_{2,-})_{\left|_{\ru=\pm \Ru; \rd=\pm \Rd}\right.}.
\end{equation*}
We notice that the action of $T$ is diagonal in the sense 
that the Dirichlet datum defined on a connected component of the boundary $\partial C_\e$ is mapped to a Neumann datum defined on the same connected component. 

\medskip

On the other hand, for the interior problem \eqref{dirichlet-correction3} we define the mapping $ S_\e:\,C^{2,\b}(\partial C_\e)\longrightarrow C^{1,\b}(\partial C_\e)$ which acts in the following way
\begin{equation*}\label{DNS}
S_\e:\, \bar{\psi} \longmapsto (-\partial_{\ru} v_C, \partial_{\ru} v_C;-\partial_{\rd} v_C, \partial_{\rd} v_C)_{\left|_{\ru=\pm \Ru; \rd=\pm \Rd}\right.}.
\end{equation*} 
In terms of the operators $T$ and $S_\e$ the $C^1$-matching condition can be rewritten as
\begin{eqnarray*}\label{C1matchingbis}
\left[(T-S_\e)[\bar{\psi}]\right]_{1,\pm}\,&=&\,\pm\partial_{\ru} v_{1,\pm} \,\,\,\,\hbox{ on } \partial D_{\etu,\pm \Ru}\\
\left[(T-S_\e)[\bar{\psi}]\right]_{2,\pm}\,&=&\,\pm\partial_{\rd} v_{2,\pm} \,\,\,\,\hbox{ on } \partial D_{\etd,\pm \Rd}.\nonumber
\end{eqnarray*}
Thus, the {\em ansatz} for $z$ actually yields a well defined correction if we prove that the above pseudodifferential problems 
%\eqref{C1matchingbis}
is solvable. To this end, we are going to show that, up to choose the parameter $\e$ small enough, the map 
$$T-S_\e: C^{2,\b}(\partial C_\e)\longrightarrow C^{1,\b}(\partial C_\e) $$
is invertible. First of all, we notice that $T$ and $S_\e$ are well defined and, thanks to the {\em a priori} estimates on the solutions to problems \eqref{dirichlet-correction2} and \eqref{dirichlet-correction3}, they are also $\e$-uniformly bounded. We prove now the following
%
%Moreover, by a standard density argument, we can extend the definition of $T$ and $S_\e$ to operators (defined in the same manner) acting between $H^{1}(\partial C_\e)$ and $ L^{2}(\partial C_\e)$. Considering these extensions, we can prove the following 
\begin{lem}
\label{convS}
As $\e\rightarrow 0$ the operator $S_\e$ converges in norm to the operator $S_0$ defined as
\begin{eqnarray*}
\quad\quad\quad  S_0\,: \,C^{2,\b}(\partial C_{0,1})\times C^{2,\b}(\partial C_{0,2})  &\longrightarrow& C^{1,\b}(\partial C_{0,1}) \times C^{1,\b}(\partial C_{0,2})\\
\left((\psi_{1,+},\psi_{1,-}),(\psi_{2,+},\psi_{2,-}) \right) &\longmapsto & \big( (-\partial_{\ru} v_{C,1},\partial_{\ru} v_{C,1})_{\left|_{\ru=\pm \Ru}\right.},(-\partial_{\rd} v_{C,2},\partial_{\rd} v_{C,2})_{\left|_{\rd=\pm \Rd}\right.} \big),\nonumber
\end{eqnarray*}
where the function $v_{C,1}$ is the unique solution of
\begin{eqnarray}\label{problimit}
 \left\{
  \begin{split}
\mathbb{L}( 1 , g_1)  \, [ u_{1}^{-1} v_{C,1}] \,\, = \,\, 0\,\, & \quad \,\,\,\, \hbox{in} \,\,\,\,\,C_{0,1}\,\,,
\\
v_{C,1} \,\,  = \,\, \psi_{1,+} \,\,  & \quad \,\,\,\,\hbox{on}  \,\, \left\{	\Ru\right\}\times \mathbb{S}^{n-1}\\
v_{C,1} \,\,  = \,\, \psi_{1,-} \,\,  & \quad \,\,\,\hbox{ on}  \,\, \left\{- \Ru\right\}\times \mathbb{S}^{n-1}
  \end{split}
\right.
\end{eqnarray}
Of course, an analogous problem characterise $v_{C,2}$.
\end{lem}
\begin{proof} The proof is by contradiction. Taking advantage of the uniform {\em a priori} bound \eqref{stimatris}, we can suppose that if the statement is not true, then there exist a sequence $(\e_j, \bar \psi_j)$ such that $\e_j \rightarrow 0$ and, for every $j \in \mathbb{N}$, $\nor{\bar \psi_j}{C^{2,\beta}(\partial C_{\e_j})} = 1$ and 
\be
\label{contrS}
\nor{ (S_{\e_j} - S_0) \, (\bar\psi_j)}{C^{1,\beta}(\partial C_\e)} \,\, \geq \,\, 1/2 \,\, .
\ee
Let $v_{C}^{j}$ be the unique solution to problem \eqref{dirichlet-correction3} with $\bar \psi_j$ as Dirichlet boundary datum. Up to choose a subsequence, we may suppose that the boundary data $\bar\psi_j$ converge to some $\bar\psi_\infty$ in $C^2(\partial C_{0,1} \cup \partial C_{0,2})$. From the uniform {\em a priori} estimates \eqref{stimatris} combined with the fact that $g_{\e_j}\rightarrow g_i$ in $C^2$ on the compact subsets of $D_{\eti}\setminus \left\{ p_i\right\}$, we deduce that, up to a subsequence, also the functions $v_C^j$ converge to some functions $v_{\infty,i}$ on the compact subsets of $C_{0,i}\setminus \left\{p_i\right\}$ with respect to the $C^2$-topology, $i=1,2$.
Making use of the conformal equivariance property combined with the removable singularities Lemma \ref{removable_lem} (which is in force since $\gamma\,<\,(n-2k)/k$ and $2<2 k\le n $), it is not difficult to see that $v_{\infty,i}$ can be extended through $p_i$ to a smooth solution of problem \eqref{problimit} on the whole $C_{0,i}$, $i=1,2$. Since it is evident that problem \eqref{problimit} has a unique solution, the functions $v_{\infty,1}$ and $v_{\infty,2}$ must coincide with $v_{C,1}$ and $v_{C,2}$ respectively. As a consequence, their normal derivative at the boundary must agree. This contradicts \eqref{contrS}.
%
%
%such that $v_{C,1}$ solves the limit problem \eqref{problimit}. Incidentally, note that since the estimate for problem \eqref{dirichlet-correction3} passes to the limit, it turns out that problem \eqref{problimit} has a unique solution.
%Now, since the $C^2$ convergence of $v_{C}^\e$ to $v_0$ implies, in particular, the
% $L^{2}$ convergence of $(\partial_{\ru} v_{C}^\e, \partial_{\ru} v_{C}^\e; \partial_{\rd} v_{C}^\e, \partial_{\rd} v_{C}^\e)_{\left|_{\ru=\pm \Ru; \rd=\pm \Rd}\right.}$, the convergence of $S_\e$ to $S_0$ follows.  
\end{proof}
In force of the lemma above, the invertibility of $T-S_\e$ is now a consequence of the invertibility of the limit pseudodifferential operator $T-S_0$. Now, since the spectrum of the limit operator $T-S_0$ is discrete 
%(as one can easily compute)
%and since the set of the linear invertible operator is open, in order to prove the solvability of the pseudodifferential problems \eqref{C1matchingbis}, 
it is sufficient to prove the injectivity of $(T- S_0)$. We reason by contradiction and we assume to have $\bar{\psi}=(\psi_{1,+}, \psi_{2,-}, \psi_{2,+}, \psi_{2,-})\neq 0$ for which $(T-S_0)[\bar{\psi}]=0$. The existence of such a boundary datum $\bar{\psi}$ implies the existence of a smooth function $\tilde v$ 
%\footnote{Actually, there exists also the analogous counterpart of $v_\infty$ defined on $D_2$ and solving a similar problem.} 
solving  
\begin{eqnarray*}
\label{problimittotal}
\mathbb{L}(1 , g_1)  \, [ u_{1}^{-1} \tilde v] \,\, = \,\, 0\,\, & \quad \,\, \hbox{on} \,\,\,\,\,D_{\eta_1}\,\,, 
 \end{eqnarray*}
and such that
 \begin{eqnarray*}
 \label{vincolo}
 \left\{
  \begin{split}
\tilde v (R_1,\theta)\,\, & = & \psi_{1,+}(\theta) \,\,\, \\
\tilde{v}(-\Ru,\theta)\,\, &  = &\psi_{1,+}(\theta)\, .
  \end{split}
\right.
\end{eqnarray*}
Moreover, $\tilde{v}$ has the following form
\begin{eqnarray*}
\label{strutturavinfty}
 \tilde{v} \,\, = \,\, 
 \left\{
  \begin{split}
 \bar{v}_{1,+} \quad\text{in}
\quad & D_{\etu,\Ru}&\,\,
\\
v_{C,1}\,\,\,\,\,\,\text{in} \quad & C_{0,1}&
\\
\bar{v}_{1,-}\quad \text{in}
\quad  &D_{\etu,-\Ru} \,,&\,\,
  \end{split}
\right.
\end{eqnarray*}
where we recall that $\bar{v}_{1,+}$ and $\bar{v}_{1,-}$ are solutions to problems of the type \eqref{dirichlet-correction2} and belong to  $C^{2,\beta}_{-\d} (D_{\etu,\Ru})  \, \oplus \, \mathcal{{W}^{+}} (D_{\etu,\Ru})$ and to $C^{2,\beta}_{-\d} (D_{\etu,-\Ru})  \, \oplus \, \mathcal{{W}^{-}} (D_{\etu,-\Ru})$, respectively.
Of course $\tilde v$ has the corresponding features on $D_{\eta_2}$, but since the situation is somehow symmetric we just focus on $D_{\eta_1}$. To reach the desired contradiction, we aim to show that $\tilde v\equiv 0$. 

\medskip

We perform the usual separation of variable, projecting the equation along the eigenfunction $\phi_j$'s of the Laplace-Beltrami operator on $(S^{n-1},g_{S^{n-1}})$, having in mind that the high frequencies ($j\ge n+1 $) and the low frequencies ($j=0,\ldots,n $) will be treated separately 
\begin{eqnarray*}
\tilde v\, (r_1,\theta) \,\,\,\, = \,\,\,\, \sum_{j =
0}^{n}\, \tilde v^j(\ru)\,\phi_j(\theta) \,+ \,  \sum_{j=
n+1}^{\infty}\, \tilde v^j(\ru)\,\phi_j(\theta) \,\, .
\end{eqnarray*}
The high frequencies are easier to treat. In fact the deficiency space components are not present in this regime. Thus the $\tilde v^j$ are exponentially decreasing for $|r_1| \rightarrow + \infty$. Hence, the maximum principle forces them to be identically zero.
%
%In particular, since in the high frequencies regime the maximum principle holds on every compact subset of $D_{1}$ (see the proof of Proposition \ref{APER} for the details), one immediately gets that 
%$\tilde{v}_\infty(\ru,\theta)\,:=\, \hbox{$\sum_{j= n+1}^\infty$}v^j_{\infty}(\ru)\,\phi_j(\theta) \equiv 0$ for any $\ru$ and $\theta$ and thus that 
%$\tilde{\psi}_{1,\pm}(\theta)\,:=\,\hbox{$\sum_{j= n+1}^\infty$}\psi^j_{1,\pm}(\theta)\,\phi_j(\theta) \equiv 0$.

\medskip

To obtain the same result for the low frequencies, we note that the functions $\tilde v^{j}$ for $j=1,\ldots,n$ can be written as a linear combination of 
$\Phi_{\eta_1}^{j,+}(\ru,\theta) $ and $\Phi_{\eta_1}^{j,-}(\ru,\theta)$. Namely, for any $j=1,\ldots,n$ there exist real numbers $A_j, B_j$ such that 
\begin{equation*}\label{lowfreq1}
\tilde v^{j} (\ru)\,:=\, A_j \, \big[ \tfrac{n-2k}{2k} v_{D,\etu} (\ru) Ê\, + \, \dot{v}_{D,\etu} (\ru) \, \big] \, e^{\ru} Ê\\
+B_j \, \big[ \tfrac{n-2k}{2k} v_{D,\etu} (\ru) Ê\, - \, \dot{v}_{D,\etu} (\ru) \, \big]\,e^{-\ru}.
\end{equation*}   
Now, since $\tilde v^{j}\equiv \bar{v}^{j}_{1,+}$ for $r_1>R_1$ and $\bar{v}_{1,+}$ has a component decaying like $e^{-\d r_1}$ with $\d>1$ and the other one  decaying like $e^{-\ru}$ as $r_1 \rightarrow +\infty$, there holds that necessarily $A_j \,=\,0$. The same argument used when $r_1<-R_1$ shows that also $B_j\,=\,0$. This implies that
$\psi_{1,\pm}^{j}\,=\, 0$ for $j=1,\ldots,n$. 

\medskip

The last case is when $j=0$. As before, the function $\tilde v^{0}$ is a linear combination of the two Jacobi fields $\Phi_{\etu}^{0,-}$ and $\Phi_{\etu}^{0,+}$, namely, there exist real numbers $A_0$ and $B_0$ such that
$$
v_\infty^{0}(\ru) \,\, = \,\, A_0 \, \Phi_{\etu}^{0,+}(\ru)  + \, B_0\Phi_{\etu}^{0,-}(\ru) \,\, .
$$
Comparing the asymptotic behavior at $\pm \infty$ of the expression above with the one prescribed by the constraints $\tilde v^0 = \bar v^0_{1,+}$ for $r_1 > R_1$ and $\tilde v^0 = v^0_{1,-}$ for $r_1 < \Ru$, we get $A_0=0=B_0$. As a consequence $\tilde v^0 \equiv 0$ and $\psi_\pm^0=0$.
%As in the previous case, the triviality of $\tilde v^0$ follows from the compariso between the asymptotic behavior of $\bar v^0_{1,+}$
%since for $\ru\rightarrow +\infty$ $v_\infty^{0}\equiv\bar{v}_{1,+}^{0}$ and $\bar{v}_{1,+}^{0}$ has a bounded component (the one that lies in the deficiency space) and a component which decays like $e^{-\d \ru}$, we readily get that $B_0\equiv 0$. With a similar argument, we get that also $A_0\equiv 0$.
% The same argument, permits to conclude also in the case $j=0$. \\

\medskip

The conclusion is that $T-S_0$ is injective, hence invertible and for $\e$ sufficiently small also $T-S_\e$ is invertible. As already observed, this implies that the correction $z$ is well defined and thus the operator $\mathbb{L}(u_{\e},\bg)$ is surjective on $C^{2,\beta}_{-\d,\gamma-\frac{n-2k}{2k}} (M_\e)  \, \oplus \, \mathcal{{W}} (M_\e)$.

%can define the vector of the matchinDirichlet data in \eqref{dirichlet-correction2} and \eqref{dirichlet-correction3} as 
%\begin{eqnarray}
%\label{psicalibrato}
% \bar{\psi}:=
% \left\{
%  \begin{split}
% \bar{\psi}_{1,\pm}\,=\,\left[(T-S_\e)^{-1}(-\partial_{\ru} v_{1,\pm},0)\right]_{\left|_{1,\pm }\right.} \quad\text{in}
%\quad & \partial D_{\etu,\pm \Ru}&\,\,
%\\
%\bar{\psi}_{2,\pm}\,=\,\left[(T-S_\e)^{-1}(0,-\partial_{\rd} v_{2,\pm})\right]_{\left|_{2,\pm }\right.}  \quad \text{in}
%\quad  &\partial D_{\etd,\pm \Rd}.&\,\,
%  \end{split}
%\right.
%\end{eqnarray}

\medskip

To conclude the proof of our statement, we need to establish the {\em a priori} estimate \eqref{stima-globale}, which obviously implies injectivity. First of all we notice that thanks to \eqref{apepm},
the solution $w$ verifies, for some positive constant $C>0$ independent of $\e$, the bound
$$
\nor{w}{C^{2,\b}_{-\d}(D_{\etu,\Ru})\oplus \mathcal{W}^+(D_{\etu,\Ru})} \,\, \leq \,\, C \,\, \big[ \,
\nor{f}{C^{0,\b}_{-\d,\gamma-{(n-2k)}}(M_\e)} + \, \nor{w_{|\partial C_\e}}{C^{2,\b}(\partial C_\e)}
 \, \big]
$$
with analogous estimates on the other connected components of $M_\e \setminus C_\e$. Moreover, the same argument used to prove \eqref{stimavc} implies that for $\e$ sufficiently small
\begin{equation*}
\Vert w \Vert_{C^{2,\b}_{\gamma-\frac{n-2k}{2k}}(C_\e)} \,\, \leq \,\, ÊC \, \, \big[ \,\, \nor{f}{C^{0,\b}_{-\d,\gamma-{(n-2k)}}(M_\e)} \,\, + \,\, \Vert   w_{|\partial C_\e}
\Vert_{C^{\,2, \b} 
%_{\gamma - \frac{n-2k}{2k}}
(\partial C_\e)} \,\, \big] \,,
\end{equation*}
for some $C>0$ independent of $\e$. Finally, it follows from standard interior Schauder estimates that the norm of $w$ on $\partial C_\e$ is uniformly bounded by the norm of $f$ in a small neighborhood of $\partial C_\e$. Combining this remark with the last two estimates, it is now easy to obtain \eqref{stima-globale}. This completes the proof. 
\end{proof}

\section{Nonlinear analysis}
\label{s:nonlinear}

In this last section we are going to correct the approximate solutions $u_\e$ to exact solutions, provided the parameter $\e$ is small enough. Again, for sake of simplicity, we limit ourself to the case $M_\e = D_{\eta_1} \sharp_\e D_{\eta_2}$, without discussing the minor changes needed for the general case. According to the linear analysis, it is natural to look for a correction lying in $ C^{2,\b}_{-\d , \g - \frac{n-2k}{2k}}(M_\e) \oplus \mathcal{W}(M_\e)$. Recalling that $\mathcal{W}(M_\e)$ is finite dimensional and identifying a function in this space with its coordinates with respect to the Jacobi fields basis, it is immediate to obtain the following isomorphism of Banach spaces
\begin{eqnarray*}
\label{ident}
\mathscr{I} \,\, : \,\,
 C^{2,\b}_{-\d , \g - \frac{n-2k}{2k}}(M_\e) \oplus \mathcal{W}(M_\e)
 & \longrightarrow & C^{2,\b}_{-\d , \g - \frac{n-2k}{2k}}(M_\e) \times\mathbb{R}^{n+1} \times\mathbb{R}^{n+1} \times\mathbb{R}^{n+1} \times\mathbb{R}^{n+1} \\
w\,=\,\hat{w} + \tilde{a}^{i,+}_{j} \Psi_{\eta_i}^{j,+} + \tilde{a}^{i,-}_{j}\Psi_{\eta_i}^{j,-} & \longmapsto & (\hat w \, , {\boa}^{1,+} , {\boa}^{1,-} , {\boa }^{2,+} , {\boa}^{2,-})
\end{eqnarray*}
where, for $i=1,2$ and $j=0,\ldots,n$, 
\bea 
\tilde{a}_{j}^{i,+}\,\,:=\,\,\chi_{R'_{i}} \, a_{j}^{i,+} & \,\,\hbox{and}\,\,&\tilde{a}_{j}^{i,-}\,:=\,\chi_{-R'_{i}} \, a_{j}^{i,-} \quad \hbox{and}\quad \boa^{i,\pm}\,:=\,\big({a}^{i,\pm}_{0}, \ldots ,{a}^{i,\pm}_{n}\big) \, .
\eea 
%\boa^{i,\pm}\,:=\,\left\{\tilde{a}^{i,\pm}_{j}\right\}_{j=0}^n$, for $i=1,2$
%$w$ in the form   
%\be\label{forma-w}
%w\,=\,\hat{w} + \tilde{a}^{i,+}_{j}\Psi_{\eta_i}^{j,+} + \tilde{a}^{i,-}_{j}\Psi_{\eta_i}^{j,-}, 
%\ee
%For $i=1,2$ and $j=0,\ldots,n$, these are defined as,
%\bea
%&\disp \tilde{a}_{j}^{i,+}\,:=\,\chi_{R'_{i}}a_{i}^{j,+}\nonumber\\
%&\disp \tilde{a}_{j}^{i,-}\,:=\,\chi_{-R'_{i}}a_{i}^{j,-}.
%\eea 
To describe our perturbation, we denote by $u_\e(\boa^{1,+}, \boa^{1,-}, \boa^{2,+}, \boa^{2,-} , \, \cdot \,)$ the variation of $u_\e$ with parameters $\boa^{i, \pm}$, $i=1,2$, supported on $M_\e \setminus C_\e$, which is defined on $D_{\eta_1, R_1}$ as 
$$
|\theta - (\tilde a^{1,+}_1, \ldots , \tilde a^{1,+}_n) e^{-r{_1}}|^{-\frac{n-2k}{2k}} v_{D, \eta_1} \big( \, r_1 + \log |\theta - (\tilde a^{1,+}_1, \ldots , \tilde a^{1,+}_n) e^{-\ru}| + \log (1 + \tilde a_0^{1,+})  \,   \big) \,\, .
$$
The definition of $u_\e(\boa^{1,+}, \boa^{1,-}, \boa^{2,+}, \boa^{2,-} , \, \cdot \,)$ on the other connected component of $M_\e \setminus C_\e$ is analogous. We note {\em en passant} that the `straight' approximate solution $u_\e (\cdot)$ corresponds to $\boa^{i,\pm} = 0$, for $i=1,2$.

\medskip

%we are led to solve the following (slightly modifiyied, compare with \eqref{perturbation}) fully nonlinear problem 
%\begin{center}
To get exact solutions for our nonlinear problem, we are led to look for $(\hat{w},\boa^{1,+},\boa^{1,-},\boa^{2,+},\boa^{2,-})\in C^{2,\b}_{-\d , \g - \frac{n-2k}{2k}}(M_\e) \times\mathbb{R}^{n+1} \times\mathbb{R}^{n+1} \times\mathbb{R}^{n+1} \times\mathbb{R}^{n+1}$ such that 
\begin{eqnarray}
\label{nonlinear}
\mathcal{N}(u_\e(\boa^{1,+},\boa^{1,-},\boa^{2,+},\boa^{2,-},\, \cdot\,) + \hat{w}(\cdot) \, , \, \bar g \, ) & = & 0\,.
\end{eqnarray}

\medskip

A simple computation gives
\bea
\label{newlin}
\left.
\frac{d}{ds} \right|_{s=0} \mathcal{N}\, (u_{\e}(s\tboa^{1,+},s\tboa^{1,-},s\tboa^{2,+},s\tboa^{2,-},\cdot) Ê+ s\hat{w}(\cdot),\bg)\, =\,\mathbb{L}(u_\e(\cdot), \bg) [w] \, , 
\eea
where, according to \eqref{ident},
\be
\label{forma w}
w\,=\,\hat{w} + \tilde{a}^{i,+}_{j}\Psi_{\eta_i}^{j,+} + \tilde{a}^{i,-}_{j}\Psi_{\eta_i}^{j,-}.
\ee
This formula suggests that the perturbation of $u_\e$ to an exact solution will involve a decaying term (through $\hat{w}$) together with small conformal variations of the former `straight' approximate solution. Geometrically, these variations corresponds to translations along the Delaunay axis, changes of the Delaunay parameter $\eta$ and `bendings at infinity'.

\medskip

Using a Taylor expansion, we can rewrite \eqref{nonlinear} as
\bea
\label{equation}
 \disp 0 \,&=&\,\mathcal{N}(u_{\e}(\tboa^{1,+},\tboa^{1,-},\tboa^{2,+},\tboa^{2,-},\cdot) Ê+ \hat{w}(\cdot), \bar g) \nonumber\\
\disp \,  &=& \, \mathcal{N}(u_{\e}(\cdot), \bar g) \, + \, \mathbb{L}(u_{\e}(\cdot) , \bar g) [ w] \, + \, \mathcal{Q} (u_{\e}(\cdot) , \bar g) \,[w,w] \,\, ,
\eea
where $\mathcal{Q} (u_{\e}(\cdot) , \bar g) \,[w,w]$
is the quadratic remainder.
%\begin{eqnarray}
%\label{def-quadratic2}
%&\disp \mathcal{Q} (u_{\e}(\cdot) , \bar g) \,[w;w]\,=\,\mathcal{Q} (u_{\e}(\cdot) , \bar g) \,[(\hat{w},\tboa^{1,+},0,0,0);(\hat{w},\tboa^{1,+},0,0,0)]\nonumber\\
%&\disp\,=\,
%\int_{0}^1\Big[\mathbb{L}(u_{\e}(\tboa^{1,+},0,0,0,\cdot) +\tau
% \hat{w}(\cdot),\bg) -\mathbb{L}(u_{\e}(\tboa^{1,+},0,0,0,\cdot))\Big][\hat{w}]d\tau\nonumber\\
%&\disp \, +\, \Big[\mathbb{L}(u_{\e}(\tboa^{1,+},0,0,0,\cdot))-\mathbb{L}(u_\e(\cdot),\bg)\Big][\hat{w}]\nonumber\\
%&\disp \,+\,\int_{0}^1\Big[\mathbb{L}(u_{\e}(\tau\tboa^{1,+},0,0,0,\cdot))-\mathbb{L}(u_\e(\cdot),\bg) \Big][\tboa^{1,+}_{j}\Psi_{\eta_1}^{j,1}]d\tau.
%\end{eqnarray}
Thus, the fully nonlinear problem becomes equivalent to the following fixed point problem for $w=(\hat{w},\boa^{1,+},\boa^{1,-},\boa^{2,+},\boa^{2,-})\in C^{2,\b}_{-\d , \g - \frac{n-2k}{2k}}(M_\e) \times\mathbb{R}^{n+1} \times\mathbb{R}^{n+1} \times\mathbb{R}^{n+1} \times\mathbb{R}^{n+1}$
%\\
%\begin{center}
%Find  $w=(\hat{w},\boa^{1,+},\boa^{1,-},\boa^{2,+},\boa^{2,-})\in C^{2,\b}_{-\d , \g - \frac{n-2k}{2k}}(M_\e) \times\mathbb{R}^{n+1} \times\mathbb{R}^{n+1} \times\mathbb{R}^{n+1} \times\mathbb{R}^{n+1}$ solving
\begin{eqnarray}
\label{nonlinear-fixed-point}
w\,=\,\mathbb{L}(u_\e(\cdot),\bg)^{-1}\Big[\,  - \mathcal{N}(u_\e(\cdot),\bg) \, - \, \mathcal{Q}(u_\e(\cdot),\bg)[ \, w,w \,] \,\Big] \, .
\end{eqnarray}
%\end{center}
It is worth remarking that at first time one could have used a simple perturbation of the form $u_\e + w$ with $w$ as \eqref{forma w}. In fact, the first order expansion of $\mathcal{N} (u_\e+w,\bg)$ is also given by $\mathbb{L}(u_{\e}(\cdot) , \bar g) [ w]$ and thus the linear analysis of the previous sections is still in force. On the other hand the components of the correction $w$ along the Jacobi fields $\Psi^{0,\pm}_{\eta_i}$ are not necessarily infinitesimal with respect to $u_\e$ when $|r_i| \rightarrow +\infty$, $i=1,2$, and this may possibly affect the completeness of the final solution. To avoid this problem we had to deal with perturbations of the form \eqref{nonlinear}. Indeed the conformal equivariance of the $\sigma_k$-equation insures that the variations $u_\e(\boa^{1,+},\boa^{1,-},\boa^{2,+},\boa^{2,-},\, \cdot\,)$ 
are still complete exact solutions far away from the central body $C_\e$. Since the remaining part of the perturbation $\hat w$ is exponentially decaying at infinity, the completeness of the exact solutions is definitely guaranteed. 
 
\medskip

We denote by $\mathcal{P}$ the mapping $\mathcal{P}:C^{2,\b}_{-\d , \g - \frac{n-2k}{2k}}(M_\e) \times\mathbb{R}^{n+1} \times\mathbb{R}^{n+1} \times\mathbb{R}^{n+1} \times\mathbb{R}^{n+1}\rightarrow C^{2,\b}_{-\d , \g - \frac{n-2k}{2k}}(M_\e) \times\mathbb{R}^{n+1} \times\mathbb{R}^{n+1} \times\mathbb{R}^{n+1} \times\mathbb{R}^{n+1}$ that associates to any $w$ the right hand side of \eqref{nonlinear-fixed-point}. In what follows, we will find the fixed point $w$ as a limit of the sequence $\left\{w_i\right\}_{i\in \mathbb{N}}$ defined by means of the following Newton iteration scheme
\begin{eqnarray}\label{newton-scheme}
\begin{cases}
w_{0}\,\,\,\,\,\,:=\,0\\
w_{i+1}\,:=\,\mathcal{P}(w_i),\,\,\,i\in\mathbb{N}.
\end{cases}
\end{eqnarray} 
%In the next Proposition \ref{contraction} we shall prove that for a properly choice of $\e$, $\mathcal{P}$ is a contraction defined on a small  ball of $C^{2,\b}_{-\d , \g - \frac{n-2k}{2k}}(M_\e) \times\mathbb{R}^{n+1} \times\mathbb{R}^{n+1} \times\mathbb{R}^{n+1} \times\mathbb{R}^{n+1}$. \\
To this end, we need some preparatory Lemmata.
We state the following 
\begin{lem}
\label{error estimate}
There exists a positive constant $A=A(n,k)>0$ such that for every
$1<\d< \bar \d \,(n,k)$
andÊ$ \,\,0 < \gamma < {(n-2k)}/{k}$ the proper error is estimated as
\begin{eqnarray}
\label{proper-error}
\nor{\, \mathcal{N}( u_\e,\bg ) \,}{C^{0,\b}_{-\d,\gamma-(n-2k)}(M_\e)} & \leq & A \, \e^{\gamma\frac{n-2k}{n} } .
\end{eqnarray}
\end{lem}
 The proof of this result could be found in \cite{cat-maz} to which we refer for all the details. Incidentally, we notice that the proof substantially uses that
% the localization property of the error term 
% $\mathcal{N}( u_\e,\bg )$. More precisely, thanks to the construction of 
%the approximate solution $u_\e(\cdot)$, it turns out that 
$\mathcal{N}( u_\e,\bg )$ is concentrated only on the neck region $N_\e$. Thus, even if we have to deal with noncompact manifolds, the computations needed to estimate this term are exactly of the same type of the computations in \cite[Lemma $5.1$]{cat-maz}.

\medskip

In the following lemma we provide an estimate for the quadratc remainder outside the neck region. 
\begin{lem}
\label{lem-stima-quadratica-end-comp}
There exists a positive constant $C=C(n,k,\d,\gamma)$ such that for every $1<\d< \bar \d \,(n,k)$
andÊ$ \,\,0 < \gamma < {(n-2k)}/{k}$ there holds
\begin{equation}
\label{stima-quadratica-end-comp}
\nor{w}{C^{2,\b}_{-\d , \g - \frac{n-2k}{2k}}(M_\e) \oplus \mathcal{W}(M_\e)  } \le \rho \,\,\,\,\Longrightarrow \,\,\,\,\nor{\mathcal{Q}(u_\e,\bg)[ \, w,w \,] }{C^{0,\b}_{-\d, \g - (n-2k)}(M_\e\setminus N_\e)}\,\le\, C\rho^2.
\end{equation}
\end{lem}
\begin{proof}
Let us fix a positive $\rho$ for which 
$$\nor{w}{C^{2,\b}_{-\d , \g - \frac{n-2k}{2k}}(M_\e) \oplus\mathcal{W}(M_\e)  } \le \rho $$
holds. We analyze the size of $\mathcal{Q}(u_\e,\bg)[ \, w;w \,]$ according to the definition of the norm in $ C^{0,\b}_{-\d, \g - (n-2k)}(M_\e)$ in \eqref{def-norma-holder}.
In particular, we decompose $M_\e\setminus N_\e$ as 
$M_\e\setminus N_\e\,=\, D_{\etu,R_1}\cup D_{\etu,-R_{1}}\cup D_{\etd,R_{2}}\cup D_{\etd,-R_{2}}
\cup C_\e\setminus N_\e$ and we prove that \eqref{stima-quadratica-end-comp} holds on every component of the above decomposition. On this regard, let us notice that it will be sufficient to check \eqref{stima-quadratica-end-comp} only on $C_\e\setminus (N_\e \cap D_{\eta_1})$ and on $D_{\etu,R_1}$.\\
We start with the analysis on $D_{\etu,\Ru}$. We recall that in this region $\bg\,=\,\vu^{{4k}/({n-2k})}\,\gc$. Hence, from the computational point of view it is more convenient to work with the cylindrical metric as a background metric. To this end we set $z\,:=\,\vu \,w$, $\hat{z}\,:=\,\vu\,\hat{w}$ and $z^\top\,:=\,\vu w^\top$. Thus, by using the conformal equivariance property \eqref{confequilin} and that $u_\e\,\equiv\,1$ on $D_{\eta_1,\Ru}$ (see \eqref{def-approx-conf}), one has
\begin{equation}
\label{conf-quad}
\mathcal{Q}(u_\e(\cdot) , \bg) \,[w;w]\,=\,\vu^{-\frac{2kn}{n-2k}}\mathcal{Q} (\vu(\cdot) , \gc) \,[z;z], \,\,\,\,\,\,\,w\,=\,\vu^{-1}\,z \, .
\end{equation}
Since $\vu$ is uniformly bounded from above and from below, one can deduce the desired estimate from its analogous for $\mathcal{Q} (\vu(\cdot) , \gc) \,[z \,;z] $. This last term can be expanded on $D_{\etu,\Ru}$ as
\begin{eqnarray}
\label{def-quadratic2}
\mathcal{Q} (\vu(\cdot) , \gc) \,[z \,;z] & = & \mathcal{Q} (\vu(\cdot) , \gc) \,[(\hat{z},\tboa^{1,+},0,0,0) \, ; (\hat{z},\tboa^{1,+},0,0,0)]\nonumber \\
& = & \int_{0}^1\Big[\mathbb{L}(\vu(\tboa^{1,+},0,0,0,\cdot) +\tau
 \hat{z}(\cdot),\gc) -\mathbb{L}(\vu(\tboa^{1,+},0,0,0,\cdot),\gc)\Big] \, [\hat{z}] \,\, d\tau\nonumber\\
& & +\, \Big[\mathbb{L}(\vu(\tboa^{1,+},0,0,0,\cdot),\gc)-\mathbb{L}(\vu(\cdot),\gc)\Big] \, [\hat{z}]\nonumber\\
& & +\,\int_{0}^1\Big[\mathbb{L}(\vu(\tau\tboa^{1,+},0,0,0,\cdot),\gc)-\mathbb{L}(\vu(\cdot),\gc) \Big] \, [\tilde{a}^{1,+}_{j}\Psi_{\eta_1}^{j,1}] \,\, d\tau \\
&=: &  Q_1 + Q_2 + Q_3 \,. \nonumber
\end{eqnarray}
%
%\begin{eqnarray}
%\label{def-quadratic2}
%&\disp \mathcal{Q} (\vu(\cdot) , \gc) \,[z,z]\,=\,\mathcal{Q} (\vu(\cdot) , \gc) \,[(\hat{z},\tboa^{1,+},0,0,0);(\hat{z},\tboa^{1,+},0,0,0)]\nonumber\\
%&\disp\,=\,
%\int_{0}^1\Big[\mathbb{L}(\vu(\tboa^{1,+},0,0,0,\cdot) +\tau
% \hat{z}(\cdot),\gc) -\mathbb{L}(\vu(\tboa^{1,+},0,0,0,\cdot),\gc)\Big][\hat{z}]d\tau\nonumber\\
%&\disp \, +\, \Big[\mathbb{L}(\vu(\tboa^{1,+},0,0,0,\cdot),\gc)-\mathbb{L}(\vu(\cdot),\gc)\Big][\hat{z}]\nonumber\\
%&\disp \,+\,\int_{0}^1\Big[\mathbb{L}(\vu(\tau\tboa^{1,+},0,0,0,\cdot),\gc)-\mathbb{L}(\vu(\cdot),\gc) \Big][\tboa^{1,+}_{j}\Psi_{\eta_1}^{j,1}]d\tau\,=:\,Q_1 + Q_2 + Q_3.
%\end{eqnarray}
To proceed, we recall that the linearization of $\mathcal{N}(\cdot,\gc)$ around $\vu$ has the following general structure (see \eqref{eq} and \cite[Section 5]{mn})
\be
\label{general-linear}
\mathbb{L}(\vu,\gc)\,=\,\mathbb{L}^0(\vu,\gc) + c_{n,k}\vu^{\frac{2kn}{n-2k}-1},       
\ee
where $c_{n,k}$ is a computable positive constant and $\mathbb{L}^0(\vu,\gc)$ is a second order differential operator with smooth coefficients of the following form
\be
\label{forma-parte-princ-linear}
\mathbb{L}^0(\vu,\gc) \,=\,\sum_{\vert \alpha\vert \leq 2} P^{2k-1}_\alpha( \vu , \nabla \vu , \nabla^2 \vu) \, \partial^\alpha \,, 
\ee
where $\alpha$ is a multi-index and 
$P^{2k-1}_\alpha\,:\,\mathbb{R}\times\mathbb{R}^n\times \mathbb{R}^{\frac{n(n+1)}{2}}\rightarrow \mathbb{R}$ 
is an homogeneous polynomial of degree $2k -1$
\begin{equation*}
\label{polinomio}
P^{2k-1}_\alpha(x,y,z)\,:=\,\sum_{\beta_0 +\vert\beta_1\vert +\vert\beta_2\vert=2k -1}a_{\alpha,(\beta_0,\beta_1,\beta_2)}x^{\beta_0}y^{\beta_1}z^{\beta_2}.
\end{equation*}
As a consequence, setting ${\boldsymbol{h}}:=(h_0,h_1,h_2)\in \mathbb{R}\times \mathbb{R}^n\times \mathbb{R}^{\frac{n(n+1)}{2}}$ and expanding at first order $P^{2k-1}_{\alpha}$ one has
\begin{equation}
\label{espansioneP}
P^{2k-1}_\alpha(x+h_0,y+h_1,z+h_2) - P^{2k-1}_\alpha(x,y,z)\,=\,D P^{2k-1}_\alpha(x,y,z)\cdot {\boldsymbol{h}} + O(\vert {\boldsymbol{h}}\vert^2).
\end{equation}
We have now all the ingredients to obtain the estimate on $D_{\etu,R_1}$ for \eqref{def-quadratic2}. 
First of all, we recall that (see \eqref{semiholder-weight} and \eqref{holder-weight})
\begin{eqnarray}
\label{normaQD1}
 \nor{\mathcal{Q}(\vu,\gc)[ \, z;z \,] }{C^{0,\b}_{-\d}(D_{\etu,R_1})}\,&:=&\,\sup_{[R_1,+\infty) \times \mathbb{S}^{n-1}}(\cosh r_1)^{\d} \,|\mathcal{Q}(\vu,\gc)[ \, z;z \,]|  \nonumber\\
 &+& \sup_{ r_1 \ge R_1 + 1 }  \,\, (\cosh r_1)^{\d}\, [\,\mathcal{Q}(\vu,\gc)[ \, z;z \,]\,]_{C^{0,\b}(\,(r_1-1, r_1+1) \times \mathbb{S}^{n-1}\,)}.
\end{eqnarray}
We will estimate separately the two terms in \eqref{normaQD1}. 

\medskip

We start with the estimate of the weighted $C^0$ norm of $Q_1$. 
To this end, by applying \eqref{espansioneP} to the operator $\mathbb{L}^0$ and simply expanding at first order the remaining term in \eqref{general-linear} we may decompose $Q_1$ into $Q_1\,=\,{q}_{1,1} + {q}_{1,2}$ where, for any $(\ru,\theta)\,\in\,D_{\eta_1,\Ru}$,
\begin{eqnarray}
\label{espQ11}
q_{1,1} \,:=\,\int_{0}^1\sum_{\vert\alpha\vert\le 2}\big[ D P^{2k-1}_\alpha(\vu, \nabla \vu,\nabla^2 \vu)\cdot \tau{\boldsymbol{h}} + O(\vert {\tau\boldsymbol{h}}\vert^2) \big]\partial^\alpha \hat{z} \,\, d\tau
\end{eqnarray}
and 
\begin{eqnarray}
\label{espQ12}
q_{1,2}\,:=\,\int_0^1
\big( d_{n,k}\,\vu^{\frac{2kn}{n-2k} -2}\,\tau \hat z + O(\vert \tau \hat{z}\vert^2   \big)[\hat{z}] \,\, d\tau,\,\quad d_{n,k}\,:=\,c_{n,k} \, \big(\tfrac{2kn}{n-2k}-1 \big)
\end{eqnarray}
where the vector $\boldsymbol{h}$ appearing in $\eqref{espQ11}$ has components $( \hat z , \nabla\hat{z}, \nabla^2 \hat{z} )$. 
Thus, it is immediate to obtain
\begin{eqnarray}
\label{stimaquad1}
\nor{q_{1,1}}{C^0_{-\d}(M_\e\setminus C_\e)}\,\le\,C \, \nor{\hat{z}}{C^{2}_{-\d}(M_\e\setminus C_\e)}^2\,\,\ \hbox{ and } \,\,\,\nor{q_{1,2}}{C^0_{-\d}(M_\e\setminus C_\e)}\,\le\,C \, \nor{\hat{z}}{C^{2}_{-\d}(M_\e\setminus C_\e)}^2 \,\,
\end{eqnarray}
for some positive constant $C$, possibly depending on $n,k,\gamma$ and $\delta$.

\medskip
% for any $(r,\theta)\in [R_1,+\infty)\times \mathbb{S}^{n-1}$,
%\begin{eqnarray}\label{stimaquad1}
%& \disp \nor{Q_1}{C^0_{-\d}(M_\e\setminus C_\e)}\,\le\,(\cosh r)^{\d}\arrowvert\int_{0}^1 \Big[\mathbb{L}^0(\vu(\tboa^{1,+},0,0,0,\cdot) +\tau
% \hat{z}(\cdot),\gc) - \mathbb{L}^0(\vu(\tboa^{1,+},0,0,0,\cdot),\gc)\Big.\nonumber\\
%& \disp  + c_{n,k}(\vu(\tboa^{1,+},0,0,0,\cdot) +\tau
% \hat{z}(\cdot))^{\frac{2kn}{n-2k}-1}-c_{n,k}(\vu(\tboa^{1,+},0,0,0,\cdot))^{\frac{2kn}{n-2k}-1}\Big][z]d\tau \arrowvert\nonumber\\
% & \disp\le C\nor{\hat{z}}{C^{2,\b}_{-\d}(M_\e) }\nor{z}{C^{2,\b}_{-\d , \g - \frac{n-2k}{2k}}(M_\e) +\mathcal{W}(M_\e) }.
%\end{eqnarray}
  
Concerning $Q_2$, we preliminarly expand $\vu(\tboa^{1,+},0,0,0,\,\cdot \,)$ as (recall \eqref{psio}-\eqref{psij} and \eqref{normaW})
\be
\label{espansionue}
\vu(\tboa^{1,+},0,0,0,\,\cdot\,)\,=\,\vu(\cdot) + \Psi_{{\eta_{1}}}^{j,+}\tilde{a}_{j}^{1,+} + O(\|z^{\top}\|^2_{\mathcal{W(M_\e)}}).
\ee
Thus, splitting $Q_2$ into $Q_2\,=\, {q_{2,1}} + {q_{2,2}}$, where
\begin{eqnarray}
\label{espQ21}
q_{2,1} \,:=\, \big[ \, \mathbb{L}^0 \big( \vu(\cdot) + \Psi_{{\eta_{1}}}^{j,+}\tilde{a}_{j}^{1,+} + O(\|z^{\top}\|^2_{\mathcal{W(M_\e)}}) \, , \, \gc \, \big)  \, - \,  \mathbb{L}^0 \big( \, \vu(\cdot),\gc \, \big) \, \big] \, \, [z]
\end{eqnarray}
and 
\begin{eqnarray}
\label{espQ22}
q_{2,2}\,:=\,c_{n,k} \, \big[ \, (\vu(\cdot) + \Psi_{\eta_{1}}^{j,+}a_{j}^{1,+} + O(\|z^{\top}\|^2_{\mathcal{W(M_\e)}}))^{\frac{2kn}{n-2k}-1} - c_{n,k} \, \vu^{\frac{2kn}{n-2k}-1}(\cdot) \, \big] \,\, [z]
\end{eqnarray}
Thus, by using \eqref{espansioneP} in \eqref{espQ21} and expanding at first order in \eqref{espQ22} we get
\begin{eqnarray}
\label{stimaquad2}
&\disp \nor{q_{2,1}}{C^0_{-\d}(M_\e\setminus C_\e)}\,\le\,C \, 
\nor{z^\top}{\mathcal{W}(M_\e)}\nor{z}{C^{2,\b}_{-\d , \g - \frac{n-2k}{2k}}(M_\e) \oplus \mathcal{W}(M_\e)} \,\,\,\nonumber\\
&\disp \nor{q_{2,2}}{C^0_{-\d}(M_\e\setminus C_\e)}\,\le\,C \, 
\nor{z^\top}{\mathcal{W}(M_\e)}\nor{z}{C^{2,\b}_{-\d , \g - \frac{n-2k}{2k}}(M_\e) \oplus\mathcal{W}(M_\e)} \,\, ,
\end{eqnarray}
where the positive constant $C$ possibly depends on $n,k,\gamma$ and $\delta$.

\medskip

%\nor{Q_2}{C^0_{-\d}(M_\e\setminus C_\e)}\,\le\, (\cosh r)^\d\vert \big[\mathbb{L}^0(\vu(\tboa^{1,+},0,0,0,\cdot),\gc)  -\mathbb{L}^0(\vu(\cdot),\gc)\big][z]\vert
% \nonumber\\
% +  c_{n,k}\vert\big[(\vu(\tboa^{1,+},0,0,0,\cdot))^{\frac{2kn}{n-2k}-1} -
%   c_{n,k}(\vu(\cdot))^{\frac{2kn}{n-2k}-1}\big][z]\vert\,
%  \nonumber\\
% = \,(\cosh r)^\d\vert\big[\mathbb{L}^0( \vu(\cdot) + \Psi_{1}^{j,+}a_{j}^{1,+} + O(\|z^{\top}\|^2_{\mathcal{W(M_\e)}}),\gc)- \mathbb{L}^0(\vu(\cdot),\gc)\big][z]\vert\nonumber\\
%  + c_{n,k}\vert\big[(\vu(\cdot) + \Psi_{1}^{j,+}a_{j}^{1,+} + O(\|z^{\top}\|^2_{\mathcal{W(M_\e)}}))^{\frac{2kn}{n-2k}-1} - c_{n,k}(\vu(\cdot))^{\frac{2kn}{n-2k}-1}\big][z]\vert     
% \nonumber\\
%  \le\, C\nor{z^\top}{\mathcal{W}(M_\e)}\nor{z}{C^{2,\b}_{-\d , \g - \frac{n-2k}{2k}}(M_\e) +\mathcal{W}(M_\e)}.
%\end{eqnarray}

Now, we estimate $Q_3$. The estimate relies on the observation that $Q_3$ has compact support. To see this, let us show that 
\begin{equation*}
Q_3'\,:=\,\big[ \, \mathbb{L}(\vu(\tboa^{1,+},0,0,0,\cdot),\gc)-\mathbb{L}(\vu(\cdot),\gc) \, \big] \, [\tilde{a}^{1,+}_{j}\Psi_{\eta_1}^{j,+}]
\end{equation*}
has indeed compact support on $D_{\eta_1,\Ru}$. We may decompose
$Q_3'$ as
\begin{eqnarray*}
Q_3'\,&=&\,\big[ \, \mathbb{L}(\vu(\tboa^{1,+},0,0,0,\cdot),\gc)-\mathbb{L}(\vu(\boa^{1,+},0,0,0,\cdot),\gc) \big] \, [\tilde{a}^{1,+}_{j}\Psi_{\eta_1}^{j,+}]\\
 & & + \big[ \,  \mathbb{L}(\vu(\boa^{1,+},0,0,0,\cdot),\gc) - \mathbb{L}(\vu(\cdot),\gc) \, \big] \, [(\tilde{a}^{1,+}_{j}- a^{1,+}_{j})\Psi_{\eta_1}^{j,+}]\\ 
 && + \big[ \, \mathbb{L}(\vu(\boa^{1,+},0,0,0,\cdot),\gc) - \mathbb{L}(\vu(\cdot),\gc) \, \big] \, [ a^{1,+}_{j}\Psi_{\eta_1}^{j,+}]\\ 
 \,&=:&\,{q_{3,1}'} + {q_{3,2}'} + q_{3,3}' \,\, .
\end{eqnarray*}
Now, $q_{3,3}'\,\equiv\,0$ in $D_{\eta_1,\Ru}$. In fact, we observe that
\begin{eqnarray*}
{q_{3,3}'}\,=\,\int_{0}^1 D^2\mathcal{N}(\vu(s\boa^{1,+},0,0,0,\cdot),\gc) \, [ a^{1,+}_{j}\Psi_{\eta_1}^{j,+}, a^{1,+}_{i}\Psi_{\eta_1}^{i,+}]\,\,ds \, ,
\end{eqnarray*}
and that for any $(\ru,\theta)\in D_{\eta_1,\Ru}$  
$$
\mathcal{N}(\vu(\boa^{1,+},0,0,0,\ru,\theta),\gc)\,=\,0 \, .
$$
Now, since $\tilde{a}^{1,+}_j - a^{1,+}_j\,=\,a^{1,+}_j(\chi_{\Ru'}-1)$ has compact support, it turns out that also $q_{3,2}'$ is compactly supported. 
To see that the remaining term $q_{3,1}'$ has compact support, we first expand $\vu(\tboa^{1,+},0,0,0,\cdot)$ around $\vu(\boa^{1,+},0,0,0,\cdot)$ as 
\begin{equation*}
\vu(\tboa^{1,+},0,0,0,\cdot)\,=\,\vu(\boa^{1,+},0,0,0,\cdot) + 
%\widetilde{\Psi}_{D_{\eta_{1}}}^{j,+}
\tfrac{\partial\vu(\boa^{1,+},0,0,0,\cdot)}{\partial a_j^{1,+}}(\tilde{a}_j^{1,+}- a_j^{1,+}) \, + \,  O(\vert\tboa^{1,+}-\boa^{1,+}\vert ^2).
\end{equation*} 
As before, note that $\tfrac{\partial\vu(\boa^{1,+},0,0,0,\cdot)}{\partial a_j^{1,+}}(\tilde{a}_j^{1,+}- a_j^{1,+}) \, + \,  O(\vert\tboa^{1,+}-\boa^{1,+}\vert ^2)$ has compact support. Now, recalling \eqref{forma-parte-princ-linear} and \eqref{espansioneP} and expanding at first order the potential term in \eqref{general-linear}, it is not difficult to get
\begin{equation*}
\vert {q}_{3,1}' \vert\,\le\,C \, \nor{z^\top}{\mathcal{W}(M_\e)}
 \,\, \Big|\,   \tfrac{\partial\vu(\boa^{1,+},0,0,0,\cdot)}{\partial a_j^{1,+}}(\tilde{a}_j^{1,+}- a_j^{1,+}) + O(\vert\tboa^{1,+}-\boa^{1,+}\vert ^2) \, \Big| \, , 
\end{equation*} 
which clearly implies that $q_{3,1}'$ is compactly suported.

\medskip

We can now give the desired estimate for $Q_3$. In particular, thanks to the above computations, it is evident that we can equivalently estimate the $C^0$ norm of $Q_3$ instead of its $C^0_{-\d}$ norm. To obtain this estimate we reason as before. Using \eqref{espansionue}, \eqref{forma-parte-princ-linear} and \eqref{espansioneP} and expanding at first order the potential term in \eqref{general-linear} it is standard to get
\begin{equation}
\label{stimaquad3}
\sup_{(\ru,\theta)\,\in\,(\Ru,+\infty)\times \mathbb{S}^{n-1}}\vert Q_3 \vert (\ru,\theta) \,\le\,C \, \nor{z^\top}{\mathcal{W}(M_\e)} ^2.
\end{equation}
Thus, collecting \eqref{stimaquad1}, \eqref{stimaquad2}, 
\eqref{stimaquad3} and recalling that $\vu$ is uniformly bounded from below and from above, we get (see \eqref{conf-quad}) the weighted $C^0$ estimate for $\mathcal{Q}(u_\e,\bg)[ \, w;w \,]$ on $D_{\eta_1,\Ru}$, namely the following
\be
\label{stimaC0quad}
\sup_{[R_1,+\infty) \times \mathbb{S}^{n-1}}(\cosh r_1)^{\d} \,|\mathcal{Q}(u_\e,\bg)[ \, w;w \,]| \, (r_1, \theta)\,\le \,C\rho^2.
\ee
Now, we turn our attention to the estimate for the H\"older quotients. 
We will use two different strategies. In particular, for the terms $Q_1$ and $Q_2$, we will estimate directly their H\"older quotient. For $Q_3$ we will estimate its weighted $C^1$ norm. This is possible thanks to its particular structure. More precisely, it is possible to obtain a weighted $C^1$ estimate by relying, loosely speaking, on the regularity of the Jacobi fields and of the $\tboa^{1,+}\,:=\,\chi_{\Ru'}\boa^{1,+}$. 

\medskip

We start with the estimate of the term $Q_3$. By first using \eqref{espansionue} and then expanding at first order the coefficients of the linearized operator as in \eqref{espansioneP}, 
it is sufficient to get a weighted $C^0$ estimate for 
\begin{eqnarray*}
\label{C1est1}
\nabla \,\, \big[ \, \sum_{\vert\alpha\vert\le 2}\big[ D P^{2k-1}_\alpha(\vu, \nabla \vu, \nabla^2 \vu)\cdot {\boldsymbol{h}} \,  + O(\vert {\boldsymbol{h}}\vert^2) \big] \,\, \partial^\alpha (\tilde{a}^{1,+}_{j}\Psi_{\eta_1}^{j,+}) \, \,  \big]
\end{eqnarray*}
and for 
\begin{eqnarray*}
\label{C1est2}
\nabla \,\,  \big [ \,\, \big( d_{n,k}\,\vu(\cdot)^{\frac{2kn}{n-2k} -2}\,h_0 \, + O(\vert h_0\vert^2   \big) \, [\tboa^{1,+}_{j}\Psi_{\eta_1}^{j,1}] \,\, \big] \, , \quad d_{n,k}\,:=\,c_{n,k} \big(\tfrac{2kn}{n-2k}-1\big), 
\end{eqnarray*}
where the vector ${\boldsymbol{h}}$ has components 
$ h_i \, = \,\nabla^i(  \tilde{a}_{j}^{1,+} \Psi_{\eta_1}^{j,+} + O(\|w^{\top}\|^2_{\mathcal{W(M_\e)}})$, for $i\,=\,0,1,2$. 
%Before starting our estimates, we notice that since according to \eqref{approx-metric} the metric $\bg$ on $D_{\eta_1,\Ru}$ is conformal to the cylindrical metric $g_{cyl}$ through a conformal factor which is uniformly bounded from above and from below, it turns out that we can equivalently estimate \eqref{C1est1} and \eqref{C1est2} with $\nabla_{g_{cyl}}$ replacing 
%$\nabla_{\bg}$. 
We will outline only the estimate for the first term. 
%\eqref{C1est1}
A similar argument applies to the second.
%\eqref{C1est2}
First of all, from the definition of $\tboa^{1,+}\,:=\,\chi_{\Ru'}\boa^{1,+}$ (recall that the cut off function $\chi_{\Ru}$ is smooth and bounded with its derivatives) and from the definition of the Jacobi fields, we easily get
\begin{eqnarray*}
\label{C1est3}
\,\,\,\,\,\hspace{4mm}
%\vert \nabla \, (\tilde{a}_{j}^{1,+} \Psi_{\eta_{1}}^{j,+})\vert_{g_{cyl}}\,\le\, C \, \|z^{\top}\|_{\mathcal{W(M_\e)}}\,\,\,
% \hbox{ and }\,\,\,
 \,\, \vert \nabla \, \partial^\alpha(\tilde{a}_{j}^{1,+} \Psi_{\eta_{1}}^{j,+} )\vert \,\le \,C\, \|z^{\top}\|_{\mathcal{W(M_\e)}},\,\,\,\,\,\vert\alpha\vert\le 2
\end{eqnarray*}
Thus, since $\sup_{\ru\ge \Ru}\vert \nabla^i \vu\vert \, \,\le \,C$, for $i=0,1,2,3$, we get
\begin{equation*}
\sup_{(\ru,\theta)\in [\Ru,+\infty)\times \mathbb{S}^{n-1}}\vert \nabla Q_3 \vert (\ru,\theta) \,\le\, C \, \|z^{\top}\|_{\mathcal{W(M_\e)}}^2,
\end{equation*}
which implies, together with \eqref{stimaquad3}, 
\begin{equation}
\label{holdQ3}
\nor{Q_3}{C^{0,\b}_{-\d}(D_{\eta_1})}\,\le \,C \|z^{\top}\|_{\mathcal{W(M_\e)}}^2 \, .
\end{equation}
Now, we estimate the H\"older quotients for $Q_1$ and $Q_2$.
We start with $Q_1$. As we did for the $C^0$ estimate, we split $Q_1$ into $Q_1 = q_{1,1} + q_{1,2}$. We will detail only the estimate for the H\"older quotient of $q_{1,1}$, the one for $q_{1,2}$ being completely analogous. Recalling that $\boldsymbol{h}$ is the vector with components $h_i = \nabla_{\gc}^i\hat{z}$, for $i\,=\,0,1,2$, we can write
\begin{eqnarray}
\label{holdQ11}
q_{1,1}(r,\theta) \, - \, q_{1,1}(r',\theta')\,&=&\,\int_{0}^1\sum_{\vert\alpha\vert\,\le 2\,}\Big[ \, D P_\alpha^{2k-1}(\vu(r), \nabla\vu(r),\nabla^2 \vu(r)) \, [\tau\boldsymbol{h}(r,\theta)] \, \, \partial^\alpha\hat{z}(r,\theta) \nonumber\\
&& \quad \quad \quad \quad \quad \quad  - \, D P_\alpha^{2k-1}(\vu(r'),\nabla \vu(r'),\nabla^2\vu(r')) \,[ \tau\boldsymbol{h}(r',\theta') ] \,\, \partial^\alpha\hat{z}(r',\theta') \, \Big] \,\, d\tau\nonumber\\
& & +  \, \int_{0}^1\sum_{\vert\alpha\vert\,\le 2\,}\Big[ \, G(\tau\boldsymbol{h}(r,\theta)) \, \partial^\alpha \hat{z}(r,\theta) - G(\tau\boldsymbol{h}(r',\theta')) \, \partial^\alpha \hat{z}(r',\theta') \, \Big] \, d\tau,
\end{eqnarray}
where $G$ is a smooth function such that $G (\boldsymbol{v})\,=\,O(|\boldsymbol{v}|^2)$ and $ DG (\boldsymbol{v})\,=\,O(|\boldsymbol{v}|)$. Using the short notation 
$$
A_\alpha (r, \theta) \, :=  \, D P_\alpha^{2k-1}(\vu(r), \nabla\vu(r),\nabla^2 \vu(r))  \, ,
$$
we split integrand of the first summand in the expression above into
\begin{eqnarray*}
\sum_{\vert\alpha\vert\,\le 2\,}\Big[ \, A_\alpha (r, \theta) \, [\tau\boldsymbol{h}(r,\theta)]  \cdot  [\, \partial^\alpha\hat{z}(r,\theta) \, - \, \partial^\alpha \hat{z} (r', \theta') \, ] \, + \, A_\alpha (r', \theta') \, [ \tau\boldsymbol{h}(r,\theta)  - \tau\boldsymbol{h} 
(r',\theta')] \cdot \partial^\alpha \hat{z} (r', \theta') 
\\ + \, \big[ A_\alpha (r,\theta) - A_\alpha (r', \theta') \big] \, [\tau\boldsymbol{h}(r,\theta)] \cdot \partial^\alpha \hat{z} (r', \theta') \, \Big] \,.
\end{eqnarray*}
Using the fact that $(\cosh r)^{-\delta} < 1$, for $\delta >0$ it is now easy to bound the weighted H\"older quotient of each term by a constant times $\nor{\hat z}{C^{2,\b}_{-\d}(D_{\eta_1,\Ru})}^2$. Applying the same reasoning to the second summand in \eqref{holdQ11} and to the term $q_{1,2}$, one concludes that
\begin{eqnarray}
\label{holdQ12}
\sup_{\ru\ge \Ru +1 }(\cosh{\ru})^{\d}\left[Q_1\right]_{C^{0,\b}(\,(\ru-1, \ru +1) \times \mathbb{S}^{n-1}\,)}
 \,\le \,C \,  \nor{\hat z}{C^{2,\b}_{-\d}(D_{\eta_1,\Ru})}^2 \, .
\end{eqnarray}
%First of all, we recall that 
%\begin{eqnarray*}
%\left[ \,\mathcal{Q}(\vu,\bg)[ \, w;w \,]\,\right]_{C^{0,\b}(\,(\ru-1, \ru +1) \times \mathbb{S}^{n-1}\,)} \,:=\,\sup_{(r,\theta)\neq (r',\theta')\,
%%\in \,(\ru-1, \ru +1) 
%%\times \mathbb{S}^{n-1}
%}\frac{\vert {Q}(\vu,\gc)[ \, w;w \,](r,\theta)- {Q}(\vu,\gc)[ \, w;w \,](r',\theta')\vert}{\vert \hbox{dist}_{\gc}((r,\theta)\,,\,(r',\theta'))\vert }
%\end{eqnarray*}
Using the same arguments, one can deduce the same type of estimate for the H\"older quotient of $Q_2$, namely
\be
\label{holdQ2}
\sup_{ \ru \ge R_1 +1 }  \,\, (\cosh \ru)^{\d}\, [\,Q_2,]_{C^{0,\b}(\,(r-1, r+1) \times \mathbb{S}^{n-1}\,)}\,\le \,C\nor{\hat z}{C^{2,\b}_{-\d}(M_\e\setminus C_\e)}\nor{z}{C^{2,\b}_{-\d}(M_\e\setminus C_\e) \oplus\mathcal{W}(M_\e)}.
\ee
Thus, combining \eqref{stimaC0quad} with \eqref{holdQ3}, \eqref{holdQ12} and \eqref{holdQ2} and recalling that $\vu$ is uniformly bounded from above and from below, we obtain
\be
\label{stimaquadD1}
\nor{\mathcal{Q}(u_\e,\bg)[ \, w;w \,] }{C^{0,\b}_{-\d}(D_{\eta_1,\Ru})}\,\le\,C\rho^2.
\ee
As anticipated, the estimates of $\mathcal{Q}(u_\e(\cdot),\bg)[ \, \cdot\,\,;\,\,\cdot \,] $ on the other ends $D_{\etu,-R_1}, D_{\etd,R_2}, D_{\etd,-R_2}$ clearly follows from a similar argument.
Thus, \eqref{stimaquadD1} actually becomes
\begin{equation}
\label{stimaquadM-C}
\nor{\mathcal{Q}(u_\e,\bg)[ \, w;w \,] }{C^{0,\b}_{-\d}(M_\e\setminus C_\e)}\,\le\,C\rho^2.
\end{equation}
Finally, it remains to estimate
$\mathcal{Q}(u_\e(\cdot),\bg)$ on $C_\e\setminus N_\e$. Since on this region $u_\e(\tboa^{1,+},\tboa^{1,-},\tboa^{2,+},\tboa^{2,-},\cdot\,)$ coincides with $u_\e(\cdot)$, it turns out that the quadratic remainder can be written as 
\begin{eqnarray*}
\label{def-quadratic-compatta}
\mathcal{Q} (u_{\e}(\cdot) , \bar g) \,[w;w] & = & \mathcal{Q} (u_{\e}(\cdot) , \bar g) \,[\hat{w};\hat{w}]\nonumber\\
&=& \int_{0}^1\big[ \, \mathbb{L}(u_{\e}(\cdot) +\tau
 \hat{w}(\cdot),\bg) -\mathbb{L}(u_{\e}(\cdot) , \bg) \,\big][\hat{w}] \,\, d\tau \, .
\end{eqnarray*}
Thus, using an argument similar to the one used above (alternatively, one may  refer to \cite{cat-maz}), we have
\begin{equation*}
\label{stimaquadcomp}
\nor{\mathcal{Q}(u_\e,\bg)[ \, w;w \,] }{C^{0,\b}(C_\e\setminus N_\e)}\,\le\,C\rho^2.
\end{equation*}
Thus, the lemma is proven.
\end{proof}
\noindent
We are now in the position to conclude the proof of Theorem \ref{main1}. We need to prove that the sequence of the solutions to the iterative scheme 
\eqref{newton-scheme} (which exist thanks to Proposition 
\ref{isomorfismo-calibrato}) is equibounded in $C^{2,\b}_{-\d , \g - \frac{n-2k}{2k}}(M_\e) \oplus \mathcal{W}(M_\e)$. We start with the estimate on $w_1$. Thanks to the uniform \emph{a priori} estimate \eqref{stima-globale} for the linearized problem and to the estimate of the proper error term in Lemma \ref{error estimate}, we immediately have
\begin{equation}
\label{stimaw1}
\nor{w_1}{C^{2,\b}_{-\d , \g - \frac{n-2k}{2k}}(M_\e) \oplus\mathcal{W}(M_\e)}\,\,\le \,\,AL\e^{(\gamma + 2)\frac{n-2k}{n}},
\end{equation} 
where the constant $L=L(\d,\gamma,n,k)$ denotes the uniform bound on the norm of $\mathbb{L}(u_\e(\cdot),\bg)^{-1}$, while the constant $A=A(\d,\gamma,n,k)$ is the constant appearing in Lemma \ref{error estimate}. 

\medskip

We proceed with the estimate of $w_2$. From the very definition of $w_2$, we have
\be
\label{stimaw2-1}
\nor{w_2}{C^{2,\b}_{-\d , \g - \frac{n-2k}{2k}}(M_\e) \oplus\mathcal{W}(M_\e)} & \leq & L \, \nor{\, \mathcal{N}(u_\e,\bg) \, + \, \mathcal{Q}(u_\e,\bg){[w_1;w_1]}   }{C^{0,\b}_{-\d,\gamma-(n-2k)} (M_\e)} \\
& \leq & AL \, \e^{(\gamma + 2)\frac{n-2k}{n}} \, + \, L \, \nor{\,  \mathcal{Q}(u_\e,\bg){[w_1;w_1]}   }{C^{0,\b}_{-\d,\gamma-(n-2k)} (M_\e)} . \nonumber
\ee
Thus, we need to estimate the quadratic remainder. Recalling the definition of the global weighted norm \eqref{def-norma-holder} in $M_\e$, we have the following
\bea
 \disp\nor{ \mathcal{Q}(u_\e,\bg)[w_1;w_1]   }{C^{0,\b}_{-\d,\gamma-(n-2k)} (M_\e)}   =&  \nor{ \mathcal{Q}(u_\e,\bg)[w_1;w_1]   }{C^{0,\b}_{-\d} (M_\e \setminus N_\e)}
  +  
%\nor{\,  \Q{\bg}{u_\e}{w_1}   }{C^{0,\b}_\mu (M_2 \setminus B(p_2,1))} \\
%& & +
\sup_{N_\e} (\e\cosh t)^{\gamma - (n-2k)} \big|\mathcal{Q}(u_\e,\bg)[w_1;w_1] \big| \\
 & + \, 
 \sup_{ t \in (\log \e, -\log \e) } Ê\,\, (\e\cosh t)^{\gamma -(n-2k)}\, [\,\mathcal{Q}(u_\e,\bg)[w_1]\,]_{C^{0,\b}(\,(t-1, t+1) \times \mathbb{S}^{n-1}\,)} \,\,.
\eea
Thanks to Lemma \ref{lem-stima-quadratica-end-comp} and to estimate \eqref{stimaw1}, the first term is estimated in this way
\begin{equation*}
\label{stima-quadratic-1}
\nor{ \mathcal{Q}(u_\e,\bg)[w_1;w_1]   }{C^{0,\b}_{-\d} (M_\e \setminus N_\e)}\,\,\le\,\,C \,  \nor{w_1}{C^{2,\b}_{-\d , \g - \frac{n-2k}{2k}}(M_\e) \oplus\mathcal{W}(M_\e)  }^2 \le \, ACL \, \e^{(\gamma + 2)\frac{n-2k}{n}}.
\end{equation*}
We consider now the second term (the term containing the H\"older quotients will be then estimated in the same way, see \cite{cat-maz}). On the neck region $N_\e$,
$\mathcal{Q}(u_\e(\cdot),\bg)[w_1;w_1] $ has this form
\begin{eqnarray*}
\label{def-quadratic-neck}
\mathcal{Q} (u_{\e}(\cdot) , \bar g) \,[w_1;w_1] & = & \mathcal{Q} (u_{\e}(\cdot) , \bar g) \,[\hat{w}_1;\hat{w}_1]\nonumber\\
& = &
\int_{0}^1\big[\mathbb{L}(u_{\e}(\cdot) +\tau
 \hat{w}(\cdot),\bg) -\mathbb{L}(u_{\e}(\cdot))\big][\hat{w}_1] \, d\tau.
\end{eqnarray*}
Thus, by exploiting the structure of $\mathbb{L}(u_\e(\cdot),\bg)$ given in \eqref{general-linear}, using a first order expansion and recalling also \eqref{stimaw1} we have that there exists a positive constant independent of $\e$ such that 
\begin{eqnarray} 
\label{stima-quadratic-neck}
(\e\cosh t)^{\gamma - (n-2k)}\, \big|\mathcal{Q}(u_\e(\cdot),\bg)[w_1;w_1] \big| & \le & C \, (\e\cosh t)^{\gamma - (n-2k)}(\e\cosh t)^{(2k-2)\frac{n-2k}{2k}}(\e\cosh t)^{-2\gamma +\frac{n-2k}{k}}\nor{\hat{w}_1}{C^{2,\b}_{\g - \frac{n-2k}{2k}}(N_\e)}^2\nonumber\\
& \le  & ACL \, \e^{-\g}\e^{(\gamma+2)\frac{n-2k}{n}}\nor{w_1}{C^{2,\b}_{-\d , \g - \frac{n-2k}{2k}}(M_\e) \oplus\mathcal{W}(M_\e)}, 
\end{eqnarray}
where we have used also the fact that, for $j=0,1,2$, $\nabla^j_{\bg}(u_\e)\,=\,O((\e\cosh t)^{\frac{n-2k}{2k}})$ on the neck region. Now, since $-\g + (\gamma+2)\frac{n-2k}{n}\,>\,0$, for any $\g\,\in (0, \frac{n-2k}{k})$, we have that, setting
$$ B\,:=\,AL^2C\e^{-\g }\e^{(\gamma +2)\frac{n-2k}{n}},$$
there exists a positive number $\e_0\,=\,\e_0(\d,\g,n,k)$ such that, for any $\e\in (0,\e_0]$, we can choose $B\,\le\,\tfrac{1}{4}$. 
Consequentely, the estimate \eqref{stimaw2-1} for $w_2$ becomes
\begin{eqnarray*}
\label{stimaw2-2}
\nor{w_2}{C^{2,\b}_{-\d , \g - \frac{n-2k}{2k}}(M_\e) \oplus \mathcal{W}(M_\e)} & \leq & L \, \nor{\, \mathcal{N}(u_\e,\bg) \, + \, \mathcal{Q}(u_\e,\bg){[w_1;w_1]}   }{C^{0,\b}_{-\d,\gamma-(n-2k)} (M_\e)} \\
& \leq & AL \, \e^{(\gamma +2)\frac{n-2k}{n}} \, + \, L \, \nor{\,  \mathcal{Q}(u_\e,\bg){[w_1;w_1]}   }{C^{0,\b}_{-\d,\gamma-(n-2k)} (M_\e)} \nonumber\\
&\le & AL \, \e^{(\gamma + 2)\frac{n-2k}{n}} \, +\,\tfrac{1}{4}\,\nor{w_1}{C^{2,\b}_{-\d , \g - \frac{n-2k}{2k}}(M_\e) \oplus\mathcal{W}(M_\e)}.\nonumber
\end{eqnarray*}
Then, we can iterate the above estimate
%\eqref{stimaw2-2}
 obtaining, for $j\ge 1$,
\begin{eqnarray}
\label{stimawi}
\nor{w_{j+1}}{C^{2,\b}_{-\d , \g - \frac{n-2k}{2k}}(M_\e) \oplus\mathcal{W}(M_\e)} \, \leq \,  AL \, \e^{(\gamma + 2)\frac{n-2k}{n}} \,a_{j+1},
% + \, B \, \nor{w_j}{C^{2,\b}_{-\d , \g - \frac{n-2k}{2k}}(M_\e) +\mathcal{W}(M_\e)} \,\,\, \\
%&\leq &  \max\left\{1, AL\right\}\, \big( \hbox{$\sum_{i=0}^{j-1}$} B^i  \big) \, \e^{(\gamma + 2k)\frac{n-2k}{n}} \nonumber\\
%& \leq &  \max\left\{1,AL\right\}(1-B)^{-1} \, \e^{(\gamma + 2k)\frac{n-2k}{n}} \,\,.\nonumber
\end{eqnarray}
where the sequence $a_j$ is inductively defined as
\begin{eqnarray*}\label{induction-aj}
\begin{cases}
a_{1}\,\,\,\,\,\,:=\,1\\
a_{j+1}\,:=\,1 + \tfrac{1}{4}\,a_{j}^2,\,\,\,j\in\mathbb{N}.
\end{cases}
\end{eqnarray*} 
Now, a straightforward induction argument shows that $\sup_j a_j \,\le\,2$, thus estimate \eqref{stimawi} becomes
\begin{equation}
\label{stimawiunif}
\nor{w_j}{C^{2,\b}_{-\d , \g - \frac{n-2k}{2k}}(M_\e) \oplus\mathcal{W}(M_\e)} \,\le \,  2AL \, \e^{(\gamma + 2)\frac{n-2k}{n}}\,.
\end{equation} 
The previous estimate, combined with the fact that the embedding 
$C^{2,\b}_{-\d}(M_\e)\,\longrightarrow\,C^{2}_{-\d'}(M_\e)$ is compact for any $\d ' < \d$ (see \cite[Chapter $12$]{pacard}), implies (up to a subsequence) the convergence in $C^2_{-\d'}(M_\e)$ of $w_i$ to a fixed point 
$\,w_\e\,=\,(\hat{w}_\e,\boa^{\e,1,+},\boa^{\e,1,-},\boa^{\e,2,+},\boa^{\e,2,-})$ of the problem \eqref{nonlinear-fixed-point}. 
%Now, using the map $\mathfrak{I}$ we will often canonically identify the vector $(\hat{w}_\e,\boa^{\e,1,+},\boa^{\e,1,-},\boa^{\e,2,+},\boa^{\e,2,-})$ with $\hat{w}_\e + \tilde{a}^{\e,i,+}_{j}\tfrac{\Psi_{\eta_i}^{j,+}}{\vu} + \tilde{a}^{\e,i,-}_{j}\tfrac{\Psi_{\eta_i}^{j,-}}{\vu}, \,\,\,\,\,i=1,2, \,\,\,\, j=0,\ldots, n$. 

\medskip

Thanks to the canonical identification \eqref{ident} we will write, with a little abuse of notation, $w_\e\,=\, \hat{w}_\e + \tilde{a}^{\e,i,+}_{j} {\Psi_{\eta_i}^{j,+}} + \tilde{a}^{\e,i,-}_{j} {\Psi_{\eta_i}^{j,-}}$. Since \eqref{stimawiunif} is uniform with respect to $j$, $w_\e $ verifies
\begin{equation}
\label{stimawiunif-ep}
\nor{w_\e}{C^{2,\b}_{-\d , \g - \frac{n-2k}{2k}}(M_\e) \oplus\mathcal{W}(M_\e)} \,\le \,  2AL \, \e^{(\gamma + 2)\frac{n-2k}{n}}\,.
\end{equation}
%$w\,:=\,w_\e\,=
%\hat{w}_\e + \boa^{i,+}_{j}\Psi_{i}^{j,+} + \tilde{a}^{\e,i,-}_{j}\Psi_{i}^{j,-}, \,\,\,\,\,i=1,2, \,\,\,\, j=0,\ldots, n. $
%Now, using the map $\mathfrak{I}$ we can canonically identify the vector $(\hat{w}_\e,\boa^{\e,1,+},\boa^{\e,1,-},\boa^{\e,2,+},\boa^{\e,2,-})$ with $\hat{w}_\e + \boa^{i,+}_{j}\Psi_{\eta_i}^{j,+} + \tilde{a}^{\e,i,-}_{j}\Psi_{\eta_i}^{j,-}, \,\,\,\,\,i=1,2, \,\,\,\, j=0,\ldots, n$. With a little abuse of notations, we will write $w_\e\,=\, \hat{w}_\e + \boa^{i,+}_{j}\Psi_{\eta_i}^{j,+} + \tilde{a}^{\e,i,-}_{j}\Psi_{\eta_i}^{j,-}$. 
% \medskip

We claim now that there exists $\e_0>0$, such that for every $\e \in [0, \e_0)$ the exact solutions 
$$
u_\e(\boa^{\e,1,+},\boa^{\e,1,-},\boa^{\e,2,+},\boa^{\e,2,-},\cdot) \, + \, \hat{w}_\e(\cdot)
$$ 
are positive. To see this fact we first observe that up to choose $\e$ sufficiently small, the function $y_\e := u_\e(\boa^{\e,1,+},\boa^{\e,1,-},\boa^{\e,2,+},\boa^{\e,2,-},\cdot)$ is positive everywhere by definition. Secondly, since $\hat{w}_\e$ decays faster than $y_\e$ along the complete ends, the exact solution $y_\e + \hat{w}_	\e$ is certainly positive outside of a compact region $K_0 \subset M_\e$. Hence, since \eqref{stimawiunif-ep} implies that
\begin{equation}
\label{wpicc}
\nor{y_{\e}^{-1}\hat{w}_\e}{C^{2,\b}(K_0)}
%\oplus\mathcal{W}(M_\e) }
\,\le\,C\e^{-\g + (\g +2)\frac{n-2k}{n}},
\end{equation}
there holds that $u_\e(\boa^{\e,1,+},\boa^{\e,1,-},\boa^{\e,2,+},\boa^{\e,2,-},\cdot) + \hat{w}_\e\,=\,y_\e + \hat{w}_\e\,=\,y_\e(1 + y_{\e}^{-1}\hat{w}_\e )\,>\, 0$.
 
\medskip

For $\e\,\in\,(0,\e_0]$ we set
\begin{equation*}
\label{sol-metrica}
\widetilde{g_\e}\,:=\,(u_\e(\boa^{\e,1,+},\boa^{\e,1,-},\boa^{\e,2,+},\boa^{\e,2,-},\cdot) + \hat{w}_\e)^{\frac{4k}{n-2k}} \,  \bg. 
\end{equation*}
The above considerations imply that $\widetilde{g_\e}$ is the metric sought. We recall that the completeness of these metrics is a consequence of the decaying of $\hat{w}_\e$ on the ends of $M_\e$ and of the fact that $u_\e(\boa^{\e,1,+},\boa^{\e,1,-},\boa^{\e,2,+},\boa^{\e,2,-},\cdot)$ is, by construction, a complete solution to the $\sigma_k$-equation, locally on the end.

\medskip

Moreover, the family of metrics $\widetilde{g_\e}$ converges to the initial metric $g_i$ with respect to the $C^2$ topology on every compact subset of $D_{\eta_i}\setminus\left\{p_i\right\}$, 
for $i\,=\,1,2 $. This is evident on the four ends $D_{\eta_i,\pm R_i}$. In fact on these regions we have 
$\bg\,=\,g_i$ and  
\eqref{stimawiunif-ep} implies that on every compact subset of  $D_{\eta_i,\pm R_i}$ 
%that when $\e\,\rightarrow \,0$, $w_\e\,\longrightarrow\,0$ 
%with respect to the topology of $C^{2,\b}_{-\d}(D_{\eta_i,\pm R_i}) \times\mathbb{R}^{n+1} \times\mathbb{R}^{n+1} \times\mathbb{R}^{n+1} \times\mathbb{R}^{n+1}$. As a consequence, on every compact subset of $D_{\eta_i,\pm R_i}$ ($i=1,2$), we have 
$u_\e(\boa^{\e,1,+},\boa^{\e,1,-},\boa^{\e,2,+},\boa^{\e,2,-},\cdot) + \hat{w}_\e\longrightarrow u_\e(\cdot)\,=\,1$ in $C^2$ as $\e\rightarrow 0$. 
To see that $\widetilde{g_\e}\longrightarrow g_i$ on $D_{\eta_i}\setminus \{p_i\}\,\cap\,C_\e$, we recall, as before, that $u_\e(\boa^{\e,1,+},\boa^{\e,1,-},\boa^{\e,2,+},\boa^{\e,2,-},\cdot)\,\equiv\,u_\e$ on $C_\e$. Thus, the metric $\widetilde{g_\e}$ could be written as 
\begin{equation*}
\widetilde{g}_\e  \,\,  = \,\, (1+u_\e^{-1}w_{\e})^{\frac{4k}{n-2k}} \, g_\e.
\end{equation*}
Now, since by construction on every compact subset of $D_{\eta_i}\setminus \{p_i\}\,\cap\,C_\e$ the metric $g_\e$ converges to the initial metric $g_i$ in $C^2$ as $\e\rightarrow 0$, \eqref{wpicc}
%\begin{equation}
%\nor{u_{\e}^{-1}w_\e}{C^{2,\b}_{-\d}(M_\e)\oplus\mathcal{W}(M_\e) }\,\le\,C\e^{-\g + (\g +2)\frac{n-2k}{n}},
%\end{equation}
%which clearly implies
implies that also the exact solutions $\widetilde{g}_\e$ tend to the initial metric $g_i$ with respect to the $C^2$--topology on the compact subsets of $D_{\eta_i}\setminus \{p_i\}\,\cap\,C_\e$, for $i=1,2$, as $\e\rightarrow 0$. This concludes the proof of Theorem \ref{main1}.

\

\

\end{document}